\documentclass[reqno,oneside,10pt]{amsart}
\usepackage{fancyhdr}
\usepackage{xspace}
\usepackage{geometry}
\usepackage{tikz}
\usepackage{bbm}
\tikzset{node distance=1.5cm, auto}

\usepackage{graphicx}
\usepackage{enumerate}
\usepackage[all]{xy}
\usepackage{amsmath,amsfonts,amssymb}
\usepackage[latin9]{inputenc}

\newcommand{\C}{{\mathbb C}}

\newcommand{\G}{\mathcal{G}}

\newcommand{\K}{{\Bbbk}}

\newcommand{\JJ}{J} 

\newcommand{\cV}{\mathcal{V}}
\newcommand{\cM}{\mathcal{M}}
\newcommand{\coalg}{\operatorname{Coalg}(\K)}
\newcommand{\alg}{\operatorname{alg}}

\newtheorem{prop}{Proposition}[section]

\newtheorem{thm}[prop]{Theorem}
\newtheorem{defi}[prop]{Definition}
\newtheorem{exmp}[prop]{Example}
\newtheorem{rmk}[prop]{Remark}

\newcommand{\benu}{\begin{enumerate}}
	\newcommand{\enu}{\end{enumerate}}

\newcommand{\beqna}{\begin{eqnarray}}
	\newcommand{\eqna}{\end{eqnarray}}
\newcommand{\beqnast}{\begin{eqnarray*}}
	\newcommand{\eqnast}{\end{eqnarray*}}
\newcommand{\beqn}{\begin{equation}}
	\newcommand{\eqn}{\end{equation}}
\newcommand{\beqnst}{\begin{equation*}}
	\newcommand{\eqnst}{\end{equation*}}
\newcommand{\RepG}{\mathcal{R}\mathrm{ep}_\K(\mathcal{G})}
\newcommand{\aga}[3]{h^{(#1)\,\,\,(#2)}_{\hspace{7pt}(#3)}}

\usepackage{latexsym}

\usepackage{xspace}
\usepackage{amscd}
\usepackage{amssymb}
\usepackage{amsfonts}
\usepackage{ mathrsfs }
\usepackage{ stmaryrd }

\makeatother

\newcommand{\bema}{\left ( \begin{array}}
	\newcommand{\ema}{\end{array} \right )}

\newcommand{\End}{\operatorname{End}} 

\newcommand{\ot}{\otimes}

\begin{document}
	
	\title[Quantum Inverse Semigroups]{Quantum Inverse Semigroups}

	\author[M.M. Alves]{Marcelo Muniz  Alves}
	\address{Departamento de Matem\'atica, Universidade Federal do Paran{\'a}, Brazil}
	\email{marcelomsa@ufpr.br}
	\author[E. Batista]{Eliezer  Batista}
	\address{Departamento de Matem\'atica, Universidade Federal de Santa Catarina, Brazil}
	\email{ebatista@mtm.ufsc.br}
	\author[F.K. Boeing]{Francielle Kuerten Boeing}
	\address{Departamento de Matem\'atica, Universidade Federal de Santa Catarina, Brazil}
	\email{frankuertenboeing@gmail.com}
	
	
	\flushbottom
	
	\begin{abstract} 
		In this work, the notion of a quantum inverse semigroup is introduced as a linearized generalization of inverse semigroups. Beyond the algebra of an inverse semigroup, which is the natural example of a quantum inverse semigroup, several other examples of this new structure are presented in different contexts, those are related to Hopf algebras, weak Hopf algebras, partial actions	and Hopf categories. Finally, a generalized notion of local bisections are defined for commutative Hopf algebroids over a commutative base algebra giving rise to new examples of quantum inverse semigroups associated to Hopf algebroids in the same sense that inverse semigroups are related to groupoids.
	\end{abstract}
	
	\maketitle
	
	\section{Introduction}
	
	The very basic notion of a group has undergone several generalizations in different contexts, giving rise to a myriad of new mathematical structures. Since groups are inherently related to symmetries, one can consider these new structures arising from groups as new tools to understand the deep and subtle aspects of symmetries. In one direction, it is possible to extend groups by weakening their operations. For example, when someone weakens the group inversion, also giving up the uniqueness of units, one ends up with regular semigroups and inverse semigroups. By the widely known theorem due to Wagner and Preston \cite{Preston,Wagner}, every inverse semigroup can be viewed as a semigroup of partially defined bijections in a set, with the operation given by the composition. These partially defined bijections also evoke another mathematical structure which generalizes the notion of a group, namely, the groupoid structure. For the case of groupoids, what is modified is the definition of the binary operation, which is not globally defined anymore. It is easier to understand why groupoids are a generalization of groups if we consider a group as a one object category, taking the elements of the group as the endomorphisms of that object and the group product as the composition. In this case, a groupoid is a ``multi-object group", more precisely, a small category in which every morphism is an isomorphism.
	
	The relationship between inverse semigroups and groupoids has been elucidated in the literature in several ways. For example, starting from an inverse semigroup $S$, one can naturally associate a groupoid whose unit space is the set of units $E(S)$ and whose operation is the restriction of the operation in $S$. This groupoid has a partial order induced by the partial order of the semigroup itself, in fact, it is an inductive groupoid, meaning that its set of units is a meet semilattice. On the other hand, given an inductive groupoid, one can associate to it a new inverse semigroup. This exchange between inverse semigroups and groupoids composes the content of the Ehresmann-Nambooripad-Schein theorem, which establishes a categorical isomorphism between the category of inverse semigroups with prehomomorphisms and the category of inductive groupoids and ordered functors \cite{Ehresmann,Nambooripad,Schein}.
	
	One can also observe the interchange between inverse semigroups and groupoids considering the case of \'etale groupoids. This connection was first explored in the context of operator algebras \cite{Paterson}. An \'etale groupoid is a topological groupoid in which the source and target maps are local homeomorphisms \cite{MR}. Given an \'etale groupoid $\mathcal{G}$, the set of its local bisections $\mathfrak{B}(\mathcal{G})$ constitutes an inverse semigroup \cite{ruy3}. In turn, given an inverse semigroup $S$, one can define an action of this semigroup on the set of characters of its unit space and, from this action, associate its germ groupoid $Gr(S)$, which is an \'etale groupoid \cite{MR}. More precisely, considering the category of inverse semigroups with semigroup morphisms and the category of \'etale groupoids with algebraic morphisms\footnote{ An algebraic morphism between the groupoids $\mathcal{G}$ and $\mathcal{H}$ is a left action of $\mathcal{G}$ over the arrows of $\mathcal{H}$ commuting with the right action of $\mathcal{H}$ over itself by the multiplication in $\mathcal{H}$ \cite{Buneci}.}, the functor which associates to each inverse semigroup the germ groupoid of the canonical action on the characters of its unit space is left adjoint to the functor which associates to each \'etale groupoid its semigroup of bisections \cite{buss}.
	
	Another completely different direction in which it is possible to generalize groups is via Hopf algebras, which can be considered as a kind of ``linearized version of groups". Hopf algebras have nice properties relative to duality and representation theory and, due to the emergence of quantum groups \cite{Drinfeld}, became more popular in the nineties, even among the physicists, when quantum groups started to be considered seriously as symmetries of quantum systems, for example, as symmetries of the spectrum of diatomic molecules \cite{Chang} or symmetries of Landau states in the quantum Hall effect \cite{Sato}. There are several different generalizations of Hopf algebras in the literature. Here we mention only three structures which generalize both Hopf algebras and groupoids: weak Hopf algebras \cite{BNS}, Hopf algebroids\cite{Bohm,BM} and Hopf categories \cite{BCV}. Among the aforementioned structures, Hopf algebroids are, in certain sense, the richest and most promising option to generalize groupoids in the Hopf context. However, it does not exist so far in the literature a good generalization of inverse semigroups and Hopf algebras which can play the same role with respect to Hopf algebroids as inverse semigroups do to groupoids.
	
	Our aim in this work is exactly to start filling this gap that existed so far by introducing the quantum inverse semigroups. As characters in search of an author, this subject appeared as the story of examples in search of a theory. The lessons coming from the study of partial actions of Hopf algebras and some aspects of the theory of Hopf algebroids motivated examples of what should be a quantum inverse semigroup. After some mathematical preliminaries concerning Hopf algebroids in Section 2, we introduce quantum inverse semigroups in Section 3, giving examples relative to inverse semigroups, Hopf algebras, weak Hopf algebras, partial representations of Hopf algebras and Hopf categories and an example inspired in a quantum version of Hadamard's matrices \cite{Banica}. In Section 4, we introduce the notion of a local biretraction of a Hopf algebroid, as a dual version of local bisections in groupoids. First we concentrate on commutative Hopf algebroids and, after, we extend to not necessarily commutative Hopf algebroids over a commutative base algebra whose left and right bialgebroid structures are tied in a specific way. The set of local biretractions of a Hopf algebroid is a regular semigroup, rarely being an inverse semigroup, but the algebra generated by the set of local biretractions defines a quantum inverse semigroup after all. We construct many explicit examples of biretractions in Hopf algebroids and characterize their semigroups, making connections, in various aspects, with some very interesting constructions.

	\section{Mathematical preliminaries}
	Throughout this text, $\K$ will denote a field of characteristic $0$ and unadorned tensor products will denote tensor products over the base field $\K$.
	
	\subsection{Hopf algebroids}
	\begin{defi} \label{leftbialgebroid} \cite{Bohm}
		A left bialgebroid over a $\K$-algebra $A$ is a quintuple $(\mathcal{H}, s_l ,t_l , \Delta_l , \varepsilon_l)$ in which:
		\begin{enumerate}
			\item[\emph{(LB1)}] $s_l: A\rightarrow \mathcal{H}$ is an algebra morphism and $t_l : A\rightarrow \mathcal{H}$ is an algebra antimorphism  such that $s_l (a) t_l (b) =t_l (b) s_l (a)$, for  every $a,b \in A$,  making $\mathcal{H}$ an $A$-bimodule with the structure
			\[
			a\triangleright h \triangleleft b =s_l (a) t_l (b) h .
			\]
			\item[\emph{(LB2)}] $(\mathcal{H} , \Delta_l , \varepsilon_l )$ is an $A$-coring with the above mentioned $A$-bimodule structure.
			\item[\emph{(LB3)}] $\Delta_l (\mathcal{H}) \subseteq  \mathcal{H} \times_A^l \mathcal{H} =\{ \sum h_i \otimes k_i \in \mathcal{H} \otimes_A \mathcal{H} \; | \; 
			\sum h_i t_l (a) \otimes k_i =\sum h_i \otimes k_i s_l (a) ,\; \forall a\in A \}$  and the co-restriction map is an algebra morphism.
			\item[\emph{(LB4)}] $\varepsilon_l (hk)=\varepsilon_l (h s_l(\varepsilon_l (k))) =\varepsilon_l (h t_l(\varepsilon_l (k)))$.
		\end{enumerate} 
	\end{defi}
	
	\begin{defi} \label{rightbialgebroid} \cite{Bohm}
		A right bialgebroid over a $\K$-algebra $A$ is a quintuple $(\mathcal{H}, s_r ,t_r , \Delta_r , \varepsilon_r)$ in which:
		\begin{enumerate}
			\item[\emph{(RB1)}] $\mathcal{H}$ is a $\K$-algebra, $s_r: A\rightarrow \mathcal{H}$ is an algebra morphism and $t_r : A\rightarrow \mathcal{H}$ is an algebra antimorphism  such that $s_r (a) t_r (b) =t_r (b) s_r (a)$, for  every $a,b \in A$  making $\mathcal{H}$ an $A$-bimodule with the structure
			\[
			a\blacktriangleright h \blacktriangleleft b =ht_r (a) s_r (b)  .
			\]
			\item[\emph{(RB2)}] $(\mathcal{H} , \Delta_r , \varepsilon_r )$ is an $A$-coring with the above-mentioned $A$-bimodule structure.
			\item[\emph{(RB3)}] $\Delta_r (\mathcal{H}) \subseteq  \mathcal{H} \times_A^r \mathcal{H} =\{ \sum h_i \otimes k_i \in \mathcal{H} \otimes_A \mathcal{H} \; | \; 
			\sum s_r (a)h_i \otimes k_i =\sum h_i \otimes t_r (a) k_i ,\; \forall a\in A \}$  and the co-restriction map is an algebra morphism.
			\item[\emph{(RB4)}] $\varepsilon_r (hk)=\varepsilon_r ( s_r(\varepsilon_r (h))k) =\varepsilon_r ( t_r(\varepsilon_r (h))k)$.
		\end{enumerate} 
	\end{defi}
	
	\begin{defi} \label{hopfalgebroid} \cite{Bohm}
		Let $A$ and $\overline{A}$ be $\K$ algebras. A Hopf algebroid over the base algebras $A$ and $\overline{A}$ is a triple $\mathcal{H}=(\mathcal{H}_l , \mathcal{H}_r , S)$ such that.
		\begin{enumerate}
			\item[\emph{(HA1)}] $\mathcal{H}_l =\mathcal{H}$ is a left bialgebroid over $A$ and $\mathcal{H}_r =\mathcal{H}$ is a right bialgebroid over $\overline{A}$.
			\item[\emph{(HA2)}] $s_l \circ \varepsilon_l \circ t_r =t_r, \qquad$ $t_l \circ \varepsilon_l \circ s_r =s_r, \qquad$ $s_r \circ \varepsilon_r \circ t_l =t_l \qquad$ and $\quad t_r \circ \varepsilon_r \circ s_l =s_l$.
			\item[\emph{(HA3)}] $(\Delta_l \otimes_{\overline{A}} \mathcal{H})\circ \Delta_r =(\mathcal{H}\otimes_A \Delta_r)\circ \Delta_l$ and  $(\Delta_r \otimes_{A} \mathcal{H})\circ \Delta_l =(\mathcal{H}\otimes_{\overline{A}} \Delta_l)\circ \Delta_r$.
			\item[\emph{(HA4)}] $S: \mathcal{H}\rightarrow \mathcal{H}$ is a $\K$-linear map such that for all $a\in A$, $b\in \overline{A}$ and $h\in \mathcal{H}$, $S(t_l (a)ht_r (b))=s_r (b)S(h)s_l (a)$.
			\item[\emph{(HA5)}] Denoting by $\mu_l$ and $\mu_r$, respectively, the multiplication in $\mathcal{H}$ as left and right bialgebroid, we have
			\[
			\mu_l \circ (S\otimes_A \mathcal{H})\circ \Delta_l =s_r \circ \varepsilon_r ,\qquad \mbox{ and } \mu_r \circ (\mathcal{H} \otimes_{\overline{A}} S)\circ \Delta_r =s_l \circ \varepsilon_l .
			\]
		\end{enumerate}
	\end{defi}
	
	\begin{rmk}
		As consequences of the Hopf algebroid's axioms, we have the following properties:
		\begin{itemize}
			\item $S$ is antimultiplicative \cite{Bohm};
			\item $S$ maps unity to unity, because
			$$1_\mathcal{H}=s_r\circ \varepsilon_r(1_\mathcal{H}) = S(1_\mathcal{H})\, 1_\mathcal{H}= S(1_\mathcal{H});$$
			\item $S$ is anticomultiplicative \cite{Bohm}. More precisely, the following identities are satisfied:
			$$\Delta_l\circ S = (S\otimes_A S)\circ \Delta_r^{cop},\qquad  \Delta_r\circ S = (S\otimes_{\overline{A}} S)\circ \Delta_l^{cop}.$$
		\end{itemize}
		
	\end{rmk}
	
	Let us consider in more detail the case of a commutative Hopf algebroid $\mathcal{H}$ over a commutative base algebra $A=\overline{A}$. In this case, the source and target maps $s_l$, $s_r$, $t_l$ and $t_r$ are all morphisms of algebras; moreover, the commutativity of $\mathcal{H}$ implies that $s_l =t_r$ and $s_r =t_l$ and therefore one can choose arbitrarily one laterality for the bialgebroid structure. Throughout this work we shall denote by $s$ the right source map and by $t$ the right target map. Also in the commutative case, the left and right Takeuchi tensor products, $\mathcal{H} \times^l_A \mathcal{H}$ and $\mathcal{H} \times^r_A \mathcal{H}$, are identified with the tensor product $\mathcal{H} \otimes_A \mathcal{H}$, so the left and right comultiplications and counits coincide; the counit also turns out to be an algebra morphism. Finally, we can rewrite axiom (HA5) in a more suitable way: for any $h\in \mathcal{H}$, we have
	\[
	S(h_{(1)})h_{(2)} =s(\varepsilon (h)),\qquad \mbox{ and } \qquad h_{(1)}S(h_{(2)})=t(\varepsilon (h)).
	\]
	In particular, we can deduce the following very useful identities:
	\[
	h_{(1)}S(h_{(2)}) h_{(3)} =h, \qquad \mbox{ and } \qquad S(h_{(1)})h_{(2)}S(h_{(3)}) =S(h) .
	\]
	
	\begin{exmp}\label{AA}
		Let $A$ be a commutative algebra and consider $\mathcal{H}=A\otimes A$. This algebra is endowed with a Hopf algebroid structure by 
		\[
		s(a) =1_A \otimes a , \qquad t(a)=a \otimes 1_A , \qquad \Delta (a\otimes b) =a\otimes 1_A\otimes_A 1_A \otimes b,
		\]
		\[
		\varepsilon (a\otimes b) =ab \qquad  \mbox{ and }\qquad  \mathcal{S}(a\otimes b) =b\otimes a.
		\]
	\end{exmp}
	
	\begin{exmp}\label{AALaurent}
		A slight generalization of the previous example is the algebra of Laurent polynomials, $\mathcal{H}=(A\otimes A) [x,x^{-1}]$, for $A$ being a commutative algebra. This algebra is also a Hopf algebroid with 
		\[
		s(a) =1_A \otimes a, \qquad t(a) = a \otimes 1_A , \qquad \Delta ((a\otimes b)x^n ) =(a\otimes 1_A )x^n \otimes_A (1_A \otimes b)x^n,
		\]
		\[
		\varepsilon ((a\otimes b)x^n ) =ab \qquad  \mbox{ and }\qquad  \mathcal{S}((a\otimes b)x^n ) =(b\otimes a)x^{-n}.
		\]
	\end{exmp}

	\subsection{The Hopf algebroid of the representative functions of a discrete groupoid} \label{RepFun}
	Let $\G$ be a discrete groupoid. An $n$-dimensional  $\G$-representation consists on \cite{Laiachi}:
	\begin{itemize}
		\item $\mathcal{E}= \bigsqcup_{x\in\G^{(0)}} E_x$ disjoint union of $n$-dimensional $\K$-vector spaces $E_x$  and linear isomorphisms $\varphi_x: \K^n \rightarrow E_x$ for every $x\in \G^{(0)}.$
		
		\item A family of linear isomorphisms $\rho^\mathcal{E}_g : E_{s(g)}\rightarrow E_{t(g)}$ for every $g\in \G$ such that for every $x\in\G^{(0)}$ and composable $g,h\in\G,$
		$$\rho^\mathcal{E}_{ i(x)}=id_{E_x},\,\,\,\,\,\,\,\,\,\,\, \rho^{\mathcal{E}}_{gh}=\rho^{\mathcal{E}}_g \rho^{\mathcal{E}}_h.$$
	\end{itemize}
	
	For example, $\mathcal{I}=\bigsqcup_{x\in\G^{(0)}}\mathcal{I}_x,$ with $\mathcal{I}_x=\K$ for every $x\in\G^{(0)}$ and $\rho^{\mathcal{I}}_g=id_{\K}$ for every $g\in\G$ is a $\G$-representation.
	
	A morphism $\lambda$ between $\G$-representations $(\mathcal{E},\rho^{\mathcal{E}})$ and $(\mathcal{F},\rho^{\mathcal{F}})$ is a family of linear maps $\{\lambda_x\}_{x\in\G_{(0)}}$ with $\lambda_x: E_x\rightarrow F_x$ such that for every $g\in\G,$
	$$\rho^{\mathcal{F}}_g\lambda_{s(g)}=\lambda_{t(g)}\rho^\mathcal{E}_g.$$
	
	Denote by $\RepG$ the category of the $\G$-representations in $\K$-vector spaces, which tensor product and duals for $\G$-representations $(\mathcal{E},\rho^{\mathcal{E}})$ and $(\mathcal{F},\rho^{\mathcal{F}})$ are given by
	$$(\mathcal{E},\rho^{\mathcal{E}})\otimes (\mathcal{F},\rho^{\mathcal{F}}):= (\mathcal{E}\otimes\mathcal{F},\rho^{\mathcal{E}}\otimes\rho^\mathcal{F}) =\left(\bigsqcup_{x\in\G^{(0)}}E_x\otimes_\K F_x, \left\{\rho^\mathcal{E}_g\otimes_\K \rho^\mathcal{F}_g\right\}_{g\in\G}\right)$$
	
	$$(\mathcal{E},\rho^{\mathcal{E}})^*=\left(\bigsqcup_{x\in\G^{(0)}}{E_x}^*,\{\rho^{\mathcal{E}^*}_g\}_{g\in\G}\right),$$
	where $\rho^{\mathcal{E^*}}: {E_{s(g)}}^*\rightarrow {E_{t(g)}}^*$ with $\rho^{\mathcal{E^*}}_g (\varphi)= \varphi\circ\rho^{\mathcal{E}}_{g^{-1}}$ for every $g\in\G$ and $\varphi\in {E_{s(g)}}^*.$
	
	Now setting $A=Fun(\G^{(0)},\K),$ we have that
	$$\Gamma(\mathcal{E})=\{p:\G^{(0)}\rightarrow \mathcal{E}\, | \,p(x)\in E_x\,\,\forall x\in\G^{(0)}\}$$
	is a finitely generated and projective $A$-module \cite{Laiachi}. Also, for any $\G-$representation $(\mathcal{E},\rho^{\mathcal{E}}),$ let $T_\mathcal{E}:= \mathrm{End}_{\RepG} (\mathcal{E},\rho^{\mathcal{E}}),$ 
	$T_{\mathcal{E},\mathcal{F}}:= \mathrm{Hom}_{\RepG} \left( (\mathcal{E},\rho^\mathcal{E}),(\mathcal{F},\rho^\mathcal{F})\right)$ and consider the direct sum of tensor products 
	$$\Gamma:=\bigoplus_{(\mathcal{E},\rho^{\mathcal{E}})\in\RepG} \Gamma(\mathcal{E}^*)\otimes_{T_\mathcal{E}}\Gamma(\mathcal{E}),$$
	being the sum over the isomorphism classes of representations of $\mathcal{G}$. It is important to take care in this point to assure that the direct sum is done upon a set of indexes, not a proper class. For each $n\in \mathbb{N}$, a representation is given by the data $\{ (\varphi_x )_{x\in \mathcal{G}^{(0)}} ,(\rho_\gamma )_{\gamma \in \mathcal{G}} \}$. The tuple $(\varphi_x )_{x\in \mathcal{G}^{(0)}}$is an element of the product
	\[
	\prod_{x\in \mathcal{G}^{(0)}} \text{Hom}_\K (\K^n ,E_x) \cong \prod_{x\in \mathcal{G}^{(0)}} \text{Hom}_\K (\K^n ,\K^n)
	\]
	and the tuple $(\rho_\gamma )_{\gamma \in \mathcal{G}} \}$ can be seen as an element of
	\[
	\prod_{x,y\in \mathcal{G}^{(0)}} \text{Hom}_\K (E_x ,E_y) \cong \prod_{x,y\in \mathcal{G}^{(0)}} \text{Hom}_\K (\K^n ,\K^n),
	\]
	putting the null linear transformation in the components $\text{Hom}_\K (E_x ,E_y)$ for which there is no $\gamma \in \mathcal{G}$ such that $s(\gamma)=x$ and $t(\gamma)=y$. Therefore, the set of isomorphism classes of representations of $\mathcal{G}$, can be seen as a subset of
	\[
	\coprod_{n\in \mathbb{N}} \left( \prod_{x\in \mathcal{G}^{(0)}} \text{Hom}_\K (k^n ,k^n) \times \prod_{x,y\in \mathcal{G}^{(0)}} \text{Hom}_\K (\K^n ,\K^n) \right) .
	\]
	Once established that it is well-defined, the direct sum $\Gamma$ is a commutative $(A\otimes_\K A)-$algebra with the product
	$$(\varphi\otimes_{T_\mathcal{E}}p)(\psi\otimes_{T_\mathcal{F}}q)= (\varphi\otimes_A \psi)\otimes_{T_{\mathcal{E\otimes\mathcal{F}}}}(p\otimes_A q)$$
	for every $\varphi\in\Gamma(\mathcal{E^\ast}),$ $\psi\in\Gamma(\mathcal{F}),$ $p\in\Gamma(\mathcal{E})$ and $q\in\Gamma(\mathcal{F}).$ Finally, take the quotient 
	$$\mathcal{R}_\K (\G):=\frac{\bigoplus_{(\mathcal{E},\rho^{\mathcal{E}})\in\RepG} \Gamma(\mathcal{E}^*)\otimes_{T_\mathcal{E}} \Gamma(\mathcal{E})}{\mathscr{J}_{\RepG}}$$
	of $\Gamma$ by the ideal
	$$\mathscr{J}_{\RepG}=\langle \varphi\otimes_{T_\mathcal{F}}\lambda p - \varphi\lambda\otimes_{T_\mathcal{E}}p\,|\,\varphi\in\Gamma(\mathcal{F}^*), p\in\Gamma(\mathcal{E}),\lambda\in T_{\mathcal{E},\mathcal{F}}\rangle$$
	is a $(A\otimes_{\K}A)$-algebra with the inherited product from $\Gamma$ and is called the algebra of the representative functions on the groupoid $\G.$ 
	The elements of the algebra are denoted by $\overline{\varphi\otimes_{T_\mathcal{E}}p}.$ Also, $\mathcal{R}_\K (\G)$ has a commutative Hopf algebroid structure over the commutative base algebra $A:$ for every $a\in A,$ $\overline{\varphi\otimes_{T_\mathcal{E}}p}\in\mathcal{R}_\K(\G)$ and $x\in\G^{(0)},$
	$$\overline{s}(a)=\overline{1_A\otimes_{T_\mathcal{I}}a},\,\,\,\,\,\,\,\,\,\,\,\,\,\,\, \overline{t}(a)=\overline{a\otimes_{T_\mathcal{I}} 1_A}, \,\,\,\,\,\,\,\,\,\,\,\,\,\,\,\Delta (\overline{\varphi\otimes_{T_\mathcal{E}} p})= \sum_{i=1}^{n}\overline{\varphi \otimes_{T_\mathcal{E}} e_i} \otimes_A \overline{{e_i}^*\otimes_{T_\mathcal{E}} p},$$
	$$\varepsilon (\overline{\varphi\otimes_{T_\mathcal{E}} p})(x)=\varphi(x)\left(p(x)\right)\,\,\,\,\,\,\,\,\,\,   and\,\,\,\,\,\,\,\,\,\, S(\overline{\varphi\otimes_{T_\mathcal{E}} p})=\overline{\tilde{p}\otimes_{T_{\mathcal{E}^*}}\varphi}, \,\, with\,\, \tilde{(\_)}:\mathcal{E}\cong (\mathcal{E}^*)^*,$$
	where $\{{e_i}^*,e_i\}$ is the dual basis of the $A$-module $\Gamma(\mathcal{E})$ and $\mathcal{I}$ is the trivial $\G-$representation $\mathcal{I}=\bigsqcup_{x\in\G^{(0)}}\K,$ with $\rho^{\mathcal{I}}_g=\textrm{Id}_{\K}.$
	
	\begin{exmp}
		A group $G$ can be seen as a groupoid $\G=G$ with $\G^{(0)}= \{1_G\}.$ A $G$-representation is a finite dimensional vector space $V$ together with linear isomorphisms $\rho^V_g:V\rightarrow V$ such that $\rho^V_g\rho^V_h=\rho^V_{gh}$ for every $g,h\in G.$ Hence the representations of the groupoid $G$ are the same as the representations of the group. Also, $A=Fun(\G^{(0)},\K)\cong \K,$
		$$ \Gamma(V) = \{ p:\{1_G\}\rightarrow V\} \cong V$$
		and $\Gamma (V^\ast)\cong V^\ast.$  Moreover, an endomorphism for the representation $(V,\rho^V)$ is a linear map $\alpha:V\rightarrow V$ such that $$\alpha\circ \rho^V_g=\rho^V_g\circ \alpha$$
		for every $g\in G.$ Thus $\alpha=\lambda\,\emph{\textrm{Id}}_V$ for some $\lambda\in \K$ and  $T_V\cong \K.$ Consequently, the ideal $\mathscr{J}_{\mathcal{\RepG}}(G)=0.$ Then the algebra of representative functions of $G$ is the algebra
		$$\mathcal{R}_\K(G) \cong \bigoplus_{(V,\rho^V)\in\RepG} V^\ast\otimes_\K V$$
		and an element of $\mathcal{R}_\K(G)$ can be written as a triple $(\varphi,v,\rho^V)$ with $\varphi\in V^\ast,$ $v\in V$ and $\rho^V$ a $G-$representation, which can be identified as the representative function for the group $G$
		\[
		f :   G  \longrightarrow  \K, \ \ \ \ \
		g  \longmapsto \varphi(\rho^V(g)(v))
		\]
		Therefore $\mathcal{R}_\K(G)$ is exactly the commutative Hopf $\K-$algebra $R(G)$ of the representative functions on the group $G.$
	\end{exmp}
	
	\begin{exmp}
		Let $\G$ be a groupoid consisting only of its units. In this case, for each $x\in \G^{(0)}$, the isotropy group $G_x=\{g\in\G\,| \, s(g)=t(g)=x\}$ contains only one element. This groupoid is known as the \emph{unit groupoid} because we have $\G=\G^{(0)}$ and source and target $s=t=$ \emph{Id}$_{G^{(0)}}.$ Then a $\G-$representation is given by a disjoint union
		$$\mathcal{E}=\bigsqcup_{x\in\G^{(0)}}E_x\cong \bigsqcup_{x\in \G^{(0)}}V,$$
		where $V$ is a $n-$dimensional vector space and the linear isomorphisms $\rho^\mathcal{E}_x:E_x\rightarrow E_x$ are the identity map for every $x\in\G^{(0)}.$ Hence the $\G-$representation is simply the set $V\times \G^{(0)}.$ Also, observe that
		$$\Gamma (V\times \G^{(0)})=\{p:\G^{(0)}\rightarrow V\times \G^{(0)}\, |\, p(x)\in V\times \{x\}\}\cong A^n,$$
		where $A=Fun(\G^{(0)},\K).$ Similarly, $\Gamma((V\times\G^{(0)})^\ast) \cong A^n$ and morphisms between $\G-$representations are $T_{V\times\G^{(0)},W\times\G^{(0)}}\cong M_{n,m}(A),$ where $W$ is a $m-$dimensional vector space. Therefore, the Hopf algebroid of the representative functions of $\G$ is given by the quotient
		$$\mathcal{R}_\K(\G)= \frac{\bigoplus_{n\in\mathbb{N}}A^n\otimes_{M_n(A)}A^n}{\Big\langle u\otimes_{M_n(A)}(\lambda_{ij})v - u(\lambda_{ij})\otimes_{M_m (A)} v\Big\rangle_{u\in A^n, v\in A^m, (\lambda_{ij})\in M_{n,m}(A)}}.$$
		This quotient indeed coincides with the algebra $A$. Consider, for example, the following element of $\mathcal{R}_{\K} (\G)$
		\[
		\overline{(f^1, \ldots , f^n) \otimes_{M_n (A)} \left( \begin{array}{c} g^1 \\
				\vdots \\
				g^n \end{array} \right)} .
		\]
		The vector $(f^1 , \ldots , f^n)\in A^n$ can be viewed as the product $\mathbf{1} (f^1 , \ldots , f^n)$, in which $\mathbf{1}:\G^{(0)}\rightarrow \K$ is the constant unit function, and $(f^1 , \ldots , f^n)\in M_{1\times n} (A)$ then  
		\[
		\overline{(f^1, \ldots , f^n) \otimes_{M_n (A)} \left( \begin{array}{c} g^1 \\
				\vdots \\
				g^n \end{array} \right)} =\overline{\mathbf{1} \otimes_{M_1 (A)} (f^1, \ldots , f^n) \left( \begin{array}{c} g^1 \\
				\vdots \\
				g^n \end{array} \right)} =\overline{\mathbf{1} \otimes_A \sum_i f^i g^i}
		\]
		Therefore, $\mathcal{R}_k (\mathcal{G})\cong A$.
	\end{exmp}
	
	\begin{rmk}\cite[Proposition 2.2]{Laiachi}\label{zeta}
		For a groupoid $\G$ and $A=Fun(\G^{(0)},\K),$ one can construct a $(A\otimes_\K A)-$algebra morphism from $\mathcal{R}_\K(\G)$ to the commutative algebra $B:=Fun(\G,\K)$ given by
		\begin{center}
			\begin{tabular}{r c l}
				$\zeta:$ & $\mathscr{R}_\K (\G)$ & $\longrightarrow B$  \\
				& $\overline{\varphi\otimes_{T_\mathcal{E}}p}$ & $\longmapsto \zeta(\overline{\varphi\otimes_{T_\mathcal{E}} p}),$
			\end{tabular}    
		\end{center}
		with $\zeta(\overline{\varphi\otimes_{T_\mathcal{E}} p})(g)=\varphi(t(g))\left(\rho^\mathcal{E}_g(p(s(g))\right)$ for each $g\in \G,$ which is a well-defined map because of the definitions of the ideal $\mathscr{J}_{\RepG}$ and of the morphisms between $\G-$representations. One can also prove that $\zeta$ is injective and satisfies
		\begin{itemize}
			\item[\emph{(1)}] $ i^*\circ \zeta =\varepsilon,$ with $ i:\G^{(0)}\rightarrow \G$ being the inclusion map;
			\item[\emph{(2)}] $\zeta\circ S(\overline{\varphi\otimes_{T_\mathcal{E}}p})(g)=\zeta (\overline{\varphi\otimes_{T_\mathcal{E}}p})(g^{-1});$
			\item[\emph{(3)}] For every $g,h\in G$ such that $s(g)=t(h)$ and every $F\in \mathscr{R}_\K (\G),$ we have that
			$$\zeta (F)(gh)=\zeta(F_{(1)})(g)\zeta(F_{(2)})(h),$$
		\end{itemize}
		where $\Delta (F)=F_{(1)}\otimes_A F_{(2)}.$
		
	\end{rmk}

	\begin{exmp}\cite{Laiachi}\label{XGX}
		Consider the groupoid $\G=X\times G\times X,$ where $X$ is a set, $G$ is a group, $(x,g,y)^{-1}=(y,g^{-1},x)$ and
		$$(x,g,y)\cdot (y,h,z)=(x,gh,z)$$
		for every $x,y,z\in X$ and $g,h\in G.$ Also consider $\G^{(0)}= X$ and the source and target maps being the projections on the third and first coordinates, respectively. Let $A=Fun(X,\K)$ the set of all maps from $X$ to $\K.$ 
		
		Using the $\zeta$ map from Remark \ref{zeta}, a representative function $\overline{\varphi\otimes_{T_\mathcal{E}}p}$ of $\G$ can be seen as a map from $\G$ onto $\K$ given by
		\begin{equation}\label{eq1}
			\zeta(\overline{\varphi\otimes_{T_\mathcal{E}}p})(x,g,y)= \varphi(x)\left(\rho^\mathcal{E}_{(x,g,y)}(p(y))\right) 
		\end{equation}
		for every $x,y\in X$ and $g\in\G.$ Then fixing $x_0\in X,$ for a $n-$dimensional $\G-$representation $(\mathcal{E},\rho^\mathcal{E}),$ we have that
		$$\rho^\mathcal{E}_{(x,g,y)}=\rho^\mathcal{E}_{(x,1_G,x_0)}\rho^\mathcal{E}_{ (x_0,g,x_0)}\rho^\mathcal{E}_{(x_0,1_G,y)}$$
		for every $(x,g,y)\in\G.$ And letting $\left(a_{ij}^g\right)_{1\leq i,j\leq n}$ the $n-$square matrix representing the $\K-$linear isomorphism $\rho^\varepsilon_{(x_0,g,x_0)},$ the expression (\ref{eq1}) can be written as
		$$\zeta(\overline{\varphi\otimes_{T_\varepsilon}p})(x,g,y) = \sum_{i,j,k,l=1}^n \varphi_i(x)\, a_{jk}^g\,\, p_l(y)$$
		with $\varphi_i,p_l\in A$ for all $i,l=1,\ldots n.$
		
		In addition, if $a,b\in A$ and $f$ is in the Hopf algebra $R(G)$ of the representative functions on the group $G,$ then $f:G\rightarrow\K$ can be written as
		$$f(g)=F(\rho(g)(v)) \qquad \forall g\in G$$
		with $v$ being an element of a $n-$dimensional vector space $V,$ $F:V\rightarrow\K$ and $\rho:G\rightarrow GL(V)$ a representation of the group $G.$ Thus $\mathcal{E}=\bigsqcup_{x\in X} V$ and $\rho^\mathcal{E}_{(x,g,y)}=\rho(g):V\rightarrow V$ form a $\G-$representation and defining
		\begin{center}
			\begin{tabular}{rl c rl}
				$\varphi:X$   & $\longrightarrow V^\ast$ & $\qquad$ & $p:X$ & $\longrightarrow V$  \\
				$x$ & $\longmapsto \varphi(x): w\mapsto b(x)\, F(w)$  & $\qquad$ & $x$ & $\longmapsto a(x)\, v$
			\end{tabular},
		\end{center}
		we have that
		\begin{equation}\label{afb}
			\zeta(\overline{\varphi\otimes_{T_\mathcal{E}}p})(x,g,y) = a(x)\, f(g)\, b(y)
		\end{equation}
		for every $(x,g,y)\in\G.$ Consequently, the image of $\mathscr{R}_\K (\G)$ in $Fun(X\times G\times X,\K)$ by $\zeta$ coincides with the image of the canonical map
		$$A\otimes_\K R (G)\otimes_\K A \hookrightarrow Fun(X\times G\times X,\K).$$
		And since the two maps are injective, we have an isomorphism of $A-$bimodules 
		$$\mathscr{R}_\K (\G) \cong A\otimes_\K R (G)\otimes_\K A,$$
		which is also an isomorphism of $A-$Hopf algebroids with the Hopf algebroid structure on $A\otimes_\K R(G)\otimes_\K A$
		\begin{align}\label{structure}
			& s'(a)=1_A\otimes 1_{R(G)}\otimes a \nonumber\\
			& t'(a)=a\otimes 1_{R(G)}\otimes 1_A\nonumber\\
			& \Delta'(a\otimes f\otimes b)= (a\otimes f_{(1)}\otimes 1_A)\otimes_A (1_A\otimes f_{(2)}\otimes b)\\
			& \varepsilon' (a\otimes f\otimes b)(x)=a(x)b(x)f(1_G)\nonumber\\
			& S'(a\otimes f\otimes b)(x\otimes g\otimes y)= a(y)b(x)f(g^{-1}).\nonumber
		\end{align}
		for every $a,b\in A,$ $x,y\in X$ and $f\in R(G).$
	\end{exmp}

	\begin{rmk}\label{transitivo}
		A transitive groupoid $\G$ with \emph{source} $s$ and \emph{target} $t$ is a groupoid such that for every pair $x,y\in X$ there exists an element $g\in \G$ that satisfies $x=s(g)$ and $y=t(g).$ Then, fixing $x_0\in \G^{(0)},$ we have $\G \cong \G^{(0)}\times G_{x_0}\times \G^{(0)}$ ($G_{x_0}$ is the isotropy group for $x_0$) with the isomorphism
		$$\psi(g)=\left( t(g), \phi^{-1}_{t(g)}\,g\,\phi_{s(g)}, s(g)\right),$$
		where  for every $y\in \G^{(0)},$ $\phi_y$ gives the element of $\G$ that satisfies $x=s(\phi_y)$ and $y=t(\phi_y)$ and \emph{source} and \emph{target} given by the projections on third and first coordinates, respectively.
		
		Therefore, we have from Example \ref{XGX} that the Hopf algebroid of the representative functions of a transitive groupoid is
		$$\mathcal{R}_\K(\G)\cong A\otimes_\K R(G)\otimes_\K A,$$
		where $A=Fun(\G^{(0)},\K)$ and $R(G)$ is the Hopf algebra of the representative functions on the isotropy group $G=G_{x_0}$ for some fixed $x_0\in\G^{(0)}.$
	\end{rmk}
	
	\begin{exmp}\label{cartesiano}
		Consider a set $X$ and the groupoid $\G=X\times X$ with 
		$$(x,y)(y,z)=(x,z) \qquad\qquad (x,y)^{-1}=(y,x).$$
		Observe that this groupoid can be seen as a particular case from Example \ref{XGX} with the isotropy group  $G$ being a unitary group $\{e\}.$ Since the Hopf algebra $R(\{e\})$ is isomorphic to $\K$ and, consequently, 
		$$\mathscr{R}_\K (\G)\cong  A\otimes_\K A$$
		with the same Hopf algebroid structure seen in the Example \ref{AA}.
	\end{exmp}

	\section{Quantum inverse semigroups}
	
	The basic motivation for quantum inverse semigroups is the theory of inverse semigroups and its role played in describing partial symmetries \cite{law}. A semigroup $S$ is said to be an inverse semigroup, if every element $s\in S$ admits a unique pseudo-inverse, that is, a unique element $s^* \in S$ such that $ss^* s=s$ and $s^* s s^* =s^*$. Let $S$ be an inverse semigroup, denote by $E(S)$ the set of idempotent elements of $S$. One can prove that the uniqueness of the pseudo-inverse is equivalent to $E(S)$ being a commutative subsemigroup of $S$. 
	
	For an example of an inverse semigroup, take $\mathcal{I}(X)$ as the set of bijections between subsets of $X$, that is
	\[
	\mathcal{I}(X) =\{ f: \mbox{Dom}(f)\subseteq X \rightarrow \mbox{Im}(f) \subseteq X \; | \; f \mbox{ bijective } \} .
	\]
	The semigroup operation is given by the composition:
	\[
	fg =f\circ g : g^{-1} (\mbox{Dom}(f) \cap \mbox{Im}(g))\rightarrow f(\mbox{Dom}(f) \cap \mbox{Im}(g)) .
	\]
	This inverse semigroup is a monoid, because it contains the identity map $\mbox{Id}_X :X\rightarrow X$. Also, $\mathcal{I}(X)$ has a zero element, given by the empty map $\emptyset : \emptyset \subseteq X \rightarrow \emptyset \subseteq X$. In fact, this example is paradigmatic, because the Wagner-Preston theorem states that every inverse semigroup can be embedded into $\mathcal{I}(X)$ for some set $X$ \cite{Preston,Wagner}.
	
	We want the definition of quantum inverse semigroups to be a generalization of inverse semigroups in the same sense that Hopf algebras can be thought as a generalization of groups.
	
	\begin{defi} \label{quantuminversesemigroup}
		A  quantum inverse semigroup (QISG) is a triple $(H, \Delta , \mathcal{S})$ in which
		\begin{enumerate}
			\item[(QISG1)] $H$ is a (not necessarily unital) $\K$-ring.
			\item[(QISG2)] 
			$\Delta :H\rightarrow H\otimes H$ is multiplicative and coassociative.
			\item[(QISG3)] $\mathcal{S}:H\rightarrow H$ is a $\K$-linear map, called pseudo-antipode, satisfying 
			\begin{enumerate}
				\item[(i)] $\mathcal{S}(hk)=\mathcal{S}(k)\mathcal{S}(h)$, for all $h,k\in H$.
				\item[(ii)] $I\ast \mathcal{S} \ast I =I$ and $\mathcal{S}\ast I \ast \mathcal{S} =\mathcal{S}$ in the convolution algebra $\mbox{End}_{\K} (H)$.	
			\end{enumerate}
			\item[(QISG4)] The sub $\Bbbk$-rings generated by the images of $I\ast \mathcal{S}$ and $\mathcal{S}\ast I$ mutually commute, that is, for every $h,k\in H$, 
			\[
			h_{(1)} \mathcal{S}(h_{(2)}) \mathcal{S}(k_{(1)})k_{(2)} =\mathcal{S}(k_{(1)})k_{(2)}h_{(1)} \mathcal{S}(h_{(2)}) .
			\]
		\end{enumerate}
		A QISG is unital if $H$ is a unital $\K$-algebra and $\mathcal{S}(1_H )=1_H$. A QISG is co-unital if $H$ is a $\K$-coalgebra and $\varepsilon_H \circ \mathcal{S} =\varepsilon_H$.
	\end{defi}
	
	\begin{rmk}
		There are some natural questions about the previous definition which are worth to consider.
		\begin{enumerate}
			\item[(i)] Contrary to inverse semigroups, in general it is not possible to assure the uniqueness of the pseudo-antipode in QISG. In the case where the idempotents in the convolution algebra $\mbox{End}_{\K} (H)$ commute, then the pseudo-antipode is unique. Indeed, supposing $\mathcal{S}$ and $\overline{\mathcal{S}}$ being two linear endomorphisms in $H$ such that 
			\[
			I\ast \mathcal{S} \ast I =I ; \quad I\ast \overline{\mathcal{S}}\ast I=I ; \quad \mathcal{S}\ast I \ast \mathcal{S} =\mathcal{S} ; \quad \overline{\mathcal{S}}\ast I \ast \overline{\mathcal{S}} =\overline{\mathcal{S}} ,
			\]
			then 
			\begin{eqnarray*}
				\mathcal{S} & = & \mathcal{S} \ast I \ast \mathcal{S} = \mathcal{S} \ast I \ast \overline{\mathcal{S}} \ast I \ast \mathcal{S}  \\
				& = & \overline{\mathcal{S}} \ast I \ast \mathcal{S} \ast I \ast  \mathcal{S} =\overline{\mathcal{S}} \ast I \ast \mathcal{S} \\
				& = & \overline{\mathcal{S}} \ast I \ast \overline{\mathcal{S}} \ast I \ast \mathcal{S} =\overline{\mathcal{S}} \ast I \ast \mathcal{S} \ast I \ast \overline{\mathcal{S}} \\
				& = & \overline{\mathcal{S}} \ast I  \ast \overline{\mathcal{S}} =\overline{\mathcal{S}} .
			\end{eqnarray*}
			\item[(ii)]	Axiom (QISG4) also follows automatically if the QISG $H$ is counital and the idempotents in the convolution algebra $\mbox{Hom}_{\K}(H^{\otimes 2},H)$ commute. Let $e,\overline{e} :H\otimes H \rightarrow H$ defined as
			\[
			e(h\otimes k)=h_{(1)}\mathcal{S}(h_{(2)}) \varepsilon (k) ,\qquad \mbox{ and } \; \overline{e} (h\otimes k) =\varepsilon (h) \mathcal{S}(k_{(1)})k_{(2)} .
			\] 
			It is easy to see that both $e$ and $\overline{e}$ are idempotents with respect to the convolution product. The commutation relation $e\ast \overline{e} =\overline{e}\ast e$ is equivalent to 
			\[
			h_{(1)}\mathcal{S}(h_{(2)})\mathcal{S}(k_{(1)})k_{(2)} =\mathcal{S}(k_{(1)})k_{(2)}h_{(1)}\mathcal{S}(h_{(2)}) .
			\] 
			\item[(iii)] In axiom (QISG3), it is imposed that the pseudo antipode is anti-multiplicative, even though in most examples of QISG it is possible to show this property directly from other intrinsic characteristics of those particular examples. On the other hand, it is not required that the pseudo-antipode is anticomultiplicative, that is, $\Delta \circ \mathcal{S} =(\mathcal{S} \otimes \mathcal{S})\circ \Delta^{op}$. Although this is true in most examples, there are cases where this property is not valid. 
			\item[(iv)] In reference \cite{Li}, the author introduced a notion somewhat related to our QISG, called there as ``weak Hopf algebras''.This notion of a weak Hopf algebra does not correspond to the usual notion of weak Hopf algebra in the literature \cite{BNS}, basically because they were bialgebras, while usual weak Hopf algebras don't satisfy the unitality of the comultiplication nor the multiplicativity of the counit. Despite the fact that the notion of pseudo antipode was introduced there, we must highlight some essential differences between a QISG and the so called ``weak Hopf algebras'' (WHA for short). First, a QISG need not be unital nor counital, while the WHA are bialgebras, then they are unital and counital, therefore, even the algebra of an inverse semigroup could not be, in general, an example of a WHA. In axiom (QISG3) we demanded the antimultiplicativity of the pseudo antipode, while for WHA this condition was not postulated, but it is assumed in many points in order to obtain relevant results. Finally, for a WHA there is nothing similar to axiom (QISG4). 
			\item[(v)] We also acknowledge another similar construction in \cite{Aukhadiev} (although it was not so far published elsewhere), also called quantum inverse semigroups. The difference is that the notion of a quantum inverse semigroup given there is a $C^*$-algebra with a dense bialgebra with a pseudo-antipode satisfying (QISG3). Here we do not demand a QISG to be unital or counital. Also, the author does not demand any condition similar to our axiom (QISG4). 
		\end{enumerate}
	\end{rmk}
	
	\begin{exmp} \label{inversesemigroupalgebra}
		Let $S$ be an inverse semigroup. The inverse semigroup algebra
		\[
		\K S=\{ \sum_{s\in S} a_s \delta_s \; | \; a_s \in \K \}
		\]
		can be endowed with a structure of a counital QISG with
		\[
		\Delta (\delta_s )=\delta_s \otimes \delta_s , \ \ \varepsilon (\delta_s )=1 , \ \  \mathcal{S}(\delta_s )=\delta_{s^*}
		\] 
		When $S$ is an inverse monoid, then $\K S$ is a unital and counital QISG with $1_{\K S}=\delta_{1_S}$. The axiom (QISG4) is automatically satisfied, because the algebras generated by the images $I\ast \mathcal{S}$ and $\mathcal{S}\ast I$ both coincide with $\K E(S)$, which is a commutative algebra.
	\end{exmp}
	
	\begin{exmp} \label{example.affine.scheme}
		An affine inverse semigroup scheme is a functor $\Sigma$ from the category of commutative $\K$-algebras to the category of inverse semigroups whose composition with the forgetful functor $U:\underline{InvSgrp} \rightarrow \underline{Set}$ becomes an affine scheme, that is, a representable functor from the category of algebras to the category of sets. Let $\Sigma$ be an inverse semigroup scheme and $H$ the commutative algebra which represents it, that is, 
		\[
		\Sigma (\underline{\;}) =\mbox{Hom}_{ComAlg} ({H},\underline{\quad}) .
		\]
		The assumption that $\Sigma (A)$ is an inverse semigroup and that for any algebra morphism $\varphi :A\rightarrow B$ induces a semigroup morphism $\Sigma (\varphi ):\Sigma (A) \rightarrow \Sigma (B)$ leads to the conclusion that the multiplications in each  semigroup $S=\Sigma (A)$, define a natural transformation, $m: \Sigma \times \Sigma \Rightarrow \Sigma$. As the functor $S$ is representable, one can write the multiplication as
		\[
		m: \text{Hom}_{ComAlg} (H,\underline{\quad}) \times \text{Hom}_{ComAlg} (H,\underline{\quad}) \Rightarrow \text{Hom}_{ComAlg} (H,\underline{\quad}) ,
		\]
		or yet, via the canonical natural isomorphism
		\[
		\text{Hom}_{ComAlg} (H,\underline{\quad}) \times \text{Hom}_{ComAlg} (H,\underline{\quad}) \cong \text{Hom}_{ComAlg} (H\otimes H,\underline{\quad}),
		\]
		write it as the associated natural transformation
		\[
		\widetilde{m}: \text{Hom}_{ComAlg} (H\otimes H,\underline{\quad}) \Rightarrow \text{Hom}_{ComAlg} (H,\underline{\quad}) .
		\]
		By Yoneda's lemma, this natural transformation induces a morphism of algebras 
		\[
		\Delta :{H}\rightarrow {H}\otimes {H}
		\]
		such that, for each algebra $A$ and each pair of algebra morphisms $x,y: H\rightarrow A$ we have $x \cdot y =m_A (x,y)= (x\otimes y)\circ \Delta$.
		
		In the same way, the pseudo inverse operation can be viewed as a natural transformation  
		$(\underline{\;})^{\ast} :\Sigma^{op}\Rightarrow \Sigma$. Again, by Yoneda's lemma, this natural transformation induces a morphism of algebras (as the algebras are commutative, also an anti-morphism of algebras) $\mathcal{S}: {H}\rightarrow H$. 
		
		Given a commutative algebra $A$, the identities $ss^*s=s$ and $s^*ss^* =s^*$ for each $s\in \mbox{Hom}_{ComAlg}(H,A)$ are equivalent to the expressions $I\ast \mathcal{S} \ast I =I$ and $\mathcal{S} \ast I \ast \mathcal{S} =\mathcal{S}$. Indeed, for any $h\in H$ and for any algebra morphism $s:H\rightarrow A$
		\[
		s(h)=ss^*s (h)=s(h_{(1)})s^* (h_{(2)})s(h_{(3)})=s(h_{(1)})s(\mathcal{S}(h_{(2)}))
		s(h_{(3)}) =s(h_{(1)}\mathcal{S}(h_{(2)})h_{(3)}).
		\] 
		As this equality is valid for every algebra morphism $s:H\rightarrow A$ and for every commutative algebra $A$, we have
		\[
		h=h_{(1)}\mathcal{S}(h_{(2)})h_{(3)}, \forall h\in H.
		\]
		
		Finally, axiom (QISG4) is trivially verified because all algebras are commutative, then, for every $h,k \in H$ the elements $I\ast \mathcal{S} (h)$ and $\mathcal{S}\ast I (k)$ do commute. Therefore, the algebra $H$, representing the affine inverse semigroup scheme is a QISG.
	\end{exmp}
	
	\begin{exmp}
		Given an inverse semigroup $S$, let $H_S$ be the polynomial algebra generated by all the matrix coordinate functions of  isomorphism classes of finite dimensional $\K$-linear representations $\pi$ of $S$, that is
		\[
		H_S=\K [\pi_{i,j} \; | \; \pi :S\rightarrow M_n (\K),\quad 1\leq i,j \leq n ] ,
		\]
		in which $\pi (s) =\left( \pi_{i,j} (s) \right)_{i,j =1}^n$. Define the comultiplication on the generators by
		\[
		\Delta (\pi_{i,j})=\sum_{k=1}^{n} \pi_{i,k} \otimes \pi_{k,j}
		\]
		and extend to an algebra morphism $\Delta :H_S \rightarrow H_S \otimes H_S$ by the universal property of the polynomial algebra. Considering the natural embedding of $H_S\otimes H_S$ as a subalgebra of the algebra of functions from $S\times S$ to $\K$, the comultiplication can be written in the following way:
		\[
		\Delta (\pi_{i,j})(s,t)=\pi_{i,j}(st) =\sum_{k=1}^{n} \pi_{i,k} (s) \pi_{k,j} (t).
		\]
		Also, one can define the pseudo-antipode on the generators as
		\[
		\mathcal{S}(\pi_{i,j} )(s)=\pi_{i,j} (s^* ), \qquad \forall s\in S ,
		\]
		and extend it by the universal property of the polynomial algebra to an algebra morphism $\mathcal{S}:H\rightarrow H$ (which is also an anti-algebra morphism because of the commutativity).
		
		It is easy to verify that $(H_S, \Delta, \mathcal{S})$ is a unital quantum inverse semigroup. The unit of the polynomial algebra can be seen as the constant function $\mathbbm{1}_{H_S}:S\rightarrow \K = M_1 (\K)$ which sends every element of the semigroup $S$ into $1_\K$, and the pseudo-antipode $\mathcal{S}$, as algebra morphism, naturally sends $\mathbbm{1}_{H_S}$ to $\mathbbm{1}_{H_S}$. Axiom (QISG3) can be checked only on generators, then taking a generator $\pi_{i,j}$ with $1\leq i,j\leq n$ and any element $s\in S$, we have
		\begin{eqnarray*}
			I\ast \mathcal{S} \ast I (\pi_{i,j})(s) & = & \sum_{k,l=1}^n \pi_{i,k} (s) \mathcal{S}(\pi_{k,l})(s) \pi_{l,j}(s) \\
			& = & \sum_{k,l=1}^n \pi_{i,k} (s) \pi_{k,l}(s^*) \pi_{l,j}(s) \\
			& = & \pi_{i,j} (s s^* s)\\
			& = & \pi_{i,j}(s). 
		\end{eqnarray*}
		Therefore $I\ast \mathcal{S} \ast I =I$. Similar reasoning for $\mathcal{S}\ast I\ast \mathcal{S} =\mathcal{S}$. Axiom (QISG4) is satisfied because the algebra $H_S$ is commutative.
	\end{exmp}
	
	\begin{exmp}
		Every Hopf algebra $(H,\mu , \eta , \Delta, \varepsilon ,S)$ is a unital and counital QISG. The axiom (QISG4) follows from the antipode axiom in the Hopf algebra, then the images of $I\ast S = S\ast I= \eta \circ \varepsilon$ are contained in the commutative subalgebra $\K \cdot\mathbbm{1}_H$.
	\end{exmp}
	
	\begin{exmp}\label{weak}
		Every  weak Hopf algebra is a QISG. A weak Hopf algebra \cite{BNS} is a sextuple $(H,\mu , \eta , \Delta, \varepsilon ,S)$ such that $(H,\mu,\eta)$ is a unital algebra and $(H,\Delta,\varepsilon)$ is a coalgebra. The comultiplication $\Delta :H\rightarrow H\otimes H$ satisfies 
		\[
		\Delta (hk)=\Delta (h) \Delta (k) , \qquad (\Delta (1) \otimes 1)(1\otimes \Delta (1)) =(1\otimes \Delta(1))(\Delta (1)\otimes 1)=(\Delta \otimes I)\circ\Delta (1) .
		\]
		The counit $\varepsilon :H\rightarrow \K$ satisfies $\varepsilon (hkl)=\varepsilon (hk_{(1)})\varepsilon (k_{(2)}l)=\varepsilon (hk_{(2)})\varepsilon (k_{(1)}l)$. The antipode $S:H\rightarrow H$ in a weak Hopf algebra satisfies the following axioms:
		\begin{eqnarray*}
			h_{(1)}S(h_{(2)}) & = & \varepsilon_t (h)=\varepsilon (1_{(1)}h)1_{(2)} ,\\
			S(h_{(1)})h_{(2)} & = & \varepsilon_s (h)=1_{(1)} \varepsilon (h1_{(2)}) ,\\
			S(h_{(1)})h_{(2)}S(h_{(3)}) & = & S(h).
		\end{eqnarray*}	
		With these axioms, one can prove that $S$ is an algebra antimorphism and a coalgebra antimorphism. Moreover, for any $h\in H$,
		\[
		h_{(1)}S(h_{(2)})h_{(3)}=\varepsilon (1_{(1)}h_{(1)})1_{(2)}h_{(2)} =\varepsilon (h_{(1)})h_{(2)}=h.
		\]
		Finally, for every $h,k \in H$,
		\begin{eqnarray*}
			h_{(1)}S(h_{(2)})S(k_{(1)})k_{(2)} & = & \varepsilon (1_{(1)}h)1_{(2)} 1_{(1')} \varepsilon (k1_{(2')}) \\
			& = & \varepsilon (1_{(1)}h) 1_{(1')} 1_{(2)} \varepsilon (k1_{(2')}) \\
			& = & S(k_{(1)})k_{(2)}h_{(1)}S(h_{(2)}).
		\end{eqnarray*}	 
		Therefore, the axiom (QISG4) is satisfied, turning the weak Hopf algebra $H$ into a QISG. 
	\end{exmp}
	
	\begin{exmp}
		A nontrivial example of a Quantum Inverse Semigroup was inspired in the work of Theodor Banica and Adam Skalski \cite{Banica} on a quantum version of Hadamard's matrices. Consider the polynomial $\K$-algebra generated by the set $\{ u_{ij} \; | \; 1\leq i,j,\leq n \}$ and then consider the quotient
		\[
		H=\K [ u_{ij} | 1\leq i,j \leq n ]/\mathcal{I} ,
		\]
		in which $\mathcal{I}$ is the ideal generated by elements of the type
		\begin{enumerate}
			\item $u_{ij} u_{ik}-\delta_{j,k} u_{ij}$ ,
			\item $u_{ij} u_{kj}-\delta_{i,k} u_{ij}$ .
		\end{enumerate}
		Defining the function
		\[
		\begin{array}{rccc} \widetilde{\Delta} :& \{ u_{ij} \}_{1\leq i,j \leq n } & \rightarrow & H\otimes H \\
			\, & u_{ij} & \mapsto & \sum_{k=1}^n u_{ik} \otimes u_{kj} \end{array} ,
		\]
		one can lift it to a morphism of algebras $\overline{\Delta} : \K [ u_{ij} |1\leq i,j \leq n ] \rightarrow  H\otimes H$ doing the same on generators. We need to check that $\overline{\Delta}(\mathcal{I}) \subseteq \mathcal{I} \otimes \K [ u_{ij} |1\leq i,j \leq n ] +\K [u_{ij} |1\leq i,j \leq n ] \otimes \mathcal{I}$. Indeed, 
		\begin{eqnarray*}
			\overline{\Delta} (u_{ij} u_{ik}-\delta_{j,k} u_{ij}) & = & \sum_{p,q =1}^n u_{ip}u_{iq} \otimes u_{pj}u_{qk} -\sum_{p=1}^n \delta_{j,k} u_{ip} \otimes u_{pj} \\
			& = & \sum_{p,q =1}^n u_{ip}u_{iq} \otimes u_{pj}u_{qk} -
			\sum_{p,q=1}^n \delta_{p,q} u_{ip} \otimes u_{pj}u_{qk} \\
			& \, & +\sum_{p,q=1}^n \delta_{p,q} u_{ip} \otimes u_{pj}u_{qk} -\sum_{p=1}^n \delta_{j,k} u_{ip} \otimes u_{pj} \\
			& = & \sum_{p,q =1}^n (u_{ip}u_{iq}-\delta_{p,q} u_{ip}) \otimes u_{pj}u_{qk} + \sum_{p=1}^n u_{ip} \otimes (u_{pj} u_{pk} -\delta_{j,k} u_{pj}) ,
		\end{eqnarray*}
		analogous for $\overline{\Delta} (u_{ij}u_{kj} -\delta_{i,k} u_{ij})$. 
		Therefore, there is a well-defined algebra map
		$\Delta :H\rightarrow H\otimes H$ defined on generators as $\Delta (u_{ij})= \sum_{k=1}^n u_{ik}\otimes u_{kj}$.
		
		Also, one can define a function 
		\[
		\begin{array}{rccl} \widetilde{S} :& \{ u_{ij} \}_{1\leq i,j \leq n } & \rightarrow & H=H^{op} \\
			\, & u_{ij} & \mapsto & u_{ji} \end{array} ,
		\]
		also, lifting to an algebra morphism $\overline{S}:\K [u_{ij} | 1\leq i,j \leq n ] \rightarrow H^{op}$. It is easy to see that $\overline{S}(\mathcal{I})\subseteq \mathcal{I}$, therefore we have a well-defined algebra morphism $S:H\rightarrow H=H^{op}$.
		
		Let us verify that $(H,\Delta , S)$ defined as above is indeed a QISG. First note that
		\[
		I\ast S (u_{ij})  =  \sum_{k=1}^n u_{ik}S(u_{kj})  =  \sum_{k=1}^n u_{ik}u_{jk} =  \delta_{i,j} \sum_{k=1}^n u_{ik}
		\]
		and
		\[
		S\ast I (u_{ij}) =  \sum_{k=1}^n S(u_{ik})u_{kj}  =  \sum_{k=1}^n u_{ki}u_{kj}  =  \delta_{i,j} \sum_{k=1}^n u_{kj}.
		\]
		Then, we have
		\begin{eqnarray*}
			I\ast S \ast I (u_{i_1 j_1} \ldots u_{i_N j_N}) & = & \sum_{k_1,l_1=1}^n \ldots \sum_{k_N,l_N=1}^n u_{i_1 k_1} \ldots u_{i_N k_N} S(u_{k_1 l_1} \ldots 
			u_{k_N l_N}) u_{l_1 j_1} \ldots u_{l_N j_N} \\
			& = &  \sum_{k_1,l_1=1}^n \ldots \sum_{k_N,l_N=1}^n u_{i_1 k_1} \ldots u_{i_N k_N} u_{l_N k_N} \ldots 
			u_{l_1 k_1} u_{l_1 j_1} \ldots u_{l_N j_N}  \\
			& = &  \sum_{k_1,l_1=1}^n \ldots \sum_{k_N,l_N=1}^n u_{i_1 k_1} \ldots u_{i_N k_N} \delta_{i_1 , l_1} \ldots 
			\delta_{i_N ,l_N} u_{l_1 j_1} \ldots u_{l_N j_N}  \\
			& = &  \sum_{k_1=1}^n \ldots \sum_{k_N=1}^n u_{i_1 k_1} \ldots u_{i_N k_N}  u_{i_1 j_1} \ldots u_{i_N j_N}  \\
			& = &  \sum_{k_1=1}^n \ldots \sum_{k_N=1}^n \delta_{j_1 k_1} \ldots \delta_{j_N k_N}  u_{i_1 j_1} \ldots u_{i_N j_N}  \\
			& = & u_{i_1 j_1} \ldots u_{i_N j_N}.
		\end{eqnarray*}
		leading to $I\ast S \ast I=I$ and, analogously, $S\ast I \ast S =S$. The elements of the form $I\ast S (h)$, naturally commute with elements of the form $S\ast I (k)$ due to the commutativity of $H$, satisfying (QISG4). Therefore $(H,\Delta , S)$ is a QISG. 
		
		Moreover, this QISG is unital and counital: first, it is unital because $H$ is a unital algebra and, by construction, $S(1_H)=1_H$. Also, it is counital because one can define a function $\widetilde{\epsilon} :\{u_{ij} \}_{1\leq i,j \leq n} \rightarrow \K$ given by $\widetilde{\epsilon}(u_{ij})=\delta_{i,j}$, this can be lifted to an algebra morphism $\overline{\epsilon}:\K [u_{ij}|1\leq i,j \leq n] \rightarrow \K$ doing the same. It is straightforward to verify that $\overline{\epsilon}(\mathcal{I})=0$, therefore, there exists an algebra morphism $\epsilon :H\rightarrow H$, making, in particular, $H$ to be a commutative bialgebra. It is also easy to check that $S\circ \epsilon =\epsilon$. Note that $H$ is an example of a QISG which not a Hopf algebra, because $I\ast S (u_{ij}) \neq \delta_{i,j} 1_H=\epsilon (u_{ij})1_H$ neither a weak Hopf algebra, because $\Delta (1_H) =1_H \otimes 1_H$. One can see also that $H$ is not a QISG comin from the structure of an inverse semigroup algebra, as in Example \ref{inversesemigroupalgebra}. Even though the set of monomials $u_{i_1 j_1} \ldots u_{i_n j_n}$ form an inverse monoid, beeing the peseudo inverse $(u_{i_1 j_1} \ldots u_{i_n j_n})^*$ equal to $u_{i_1 j_1} \ldots u_{i_n j_n}$ itself, the comultiplication map for an inverse semigroup algebra to be a QISG is done making all the elements of the semigroup group-like, which is not the case of the comultiplication defined in $H$.
	\end{exmp}
	
	\subsection{QISG and partial representations}
	
	Partial representations of Hopf algebras were introduced in \cite{ABV} as an extension of the concept of partial representation of a group. 
	
	\begin{defi} \cite{ABV,ABCQV}
		Let $H$ be a Hopf $\K$-algebra, and let $B$ be a unital $\K$-algebra. A partial representation of $H$ in $B$ is a linear map $\pi: H \rightarrow B$ such that 
		\begin{enumerate}
			\item[(PR1)] $\pi (1_H)  =  1_B$, 
			\item[(PR2)] $\pi (h) \pi (k_{(1)}) \pi (S(k_{(2)}))  =   \pi (hk_{(1)}) \pi (S(k_{(2)})) $, for every $h,k\in H$.
			\item[(PR3)] $\pi (h_{(1)}) \pi (S(h_{(2)})) \pi (k)  =   \pi (h_{(1)}) \pi (S(h_{(2)})k)$, for every $h,k\in H$.
		\end{enumerate}
	\end{defi} 
	The original definition of partial representation in \cite{ABV} included two more axioms, but later it was proved that these follow from the previous ones. 
	
	\begin{prop}
		Let $H$ be a Hopf $\K$-algebra, and let $B$ be a unital $\K$-algebra. Every partial representation $\pi: H \rightarrow B$ satisfies the properties
		\begin{enumerate}
			\item[(PR4)] $\pi (h) \pi (S(k_{(1)})) \pi (k_{(2)}) = \pi (hS(k_{(1)})) \pi (k_{(2)})$, for every $h,k\in H$.
			\item[(PR5)] $\pi (S(h_{(1)}))\pi (h_{(2)}) \pi (k) = \pi (S(h_{(1)}))\pi (h_{(2)} k)$, for every $h,k\in H$.   
		\end{enumerate}
		Moreover, every linear map $\pi : H \to B$ that satisfies (PR1), (PR4) and (PR5) also satisfies (PR2) and (PR3). 
	\end{prop}
	
	\begin{defi} \label{hpar} \cite{ABV} Let $H$ be a Hopf algebra 
		and let $T(H)$ be the tensor algebra of the vector space $H$. The partial Hopf algebra $H_{par}$ is the quotient
		of $T(H)$ by the ideal $I$ generated by elements of the form 
		\begin{enumerate}
			\item $1_H - 1_{T(H)}$; 
			\item $h \otimes k_{(1)} \otimes S(k_{(2)}) - hk_{(1)} \otimes S(k_{(2)})$, for all $h,k \in H$;
			\item $h_{(1)} \otimes S(h_{(2)}) \otimes k - h_{(1)} \otimes S(h_{(2)})k$, for all $h,k \in H$;
		\end{enumerate}
	\end{defi}
	
	Denoting the class of $h\in H$ in $H_{par}$ by $[h]$, it is easy to see that the map 
	\[
	\begin{array}{rccc} [\underline{\; }]:& H & \rightarrow & H_{par} \\
		\, & h & \mapsto & [h] \end{array}
	\]
	is a partial representation of the Hopf algebra $H$ on $H_{par}$.
	
	The partial Hopf algebra $H_{par}$ has the following universal property: for every partial representation $\pi : H\rightarrow B$, there is a unique morphism of algebras $\overline{\pi}:H_{par} \rightarrow B$ such that $\pi =\overline{\pi} \circ [\underline{\; }]$. In \cite{ABV}, it was shown that $H_{par}$ has the structure of a Hopf algebroid over the base algebra 
	\[
	A_{par} (H) =\langle  \varepsilon_h =[h_{(1)}][S(h_{(2)})] \; | \; h\in H \rangle .
	\] 
	When $H$ is a co-commutative Hopf algebra, things become much simpler. For instance, the base algebra $A_{par}$ is commutative. In this case, the following result is valid for the universal algebra $H_{par}$. For $H$ being a group algebra, $\Bbbk G$ the partial algebra $H_{par}$ is just the partial group algebra $\Bbbk_{par} G$ \cite{DEP}, which is the inverse semigorup algebra $\Bbbk S(G)$. For $H$ being the universal enveloping algebra $\mathcal{U} (\mathfrak{g})$ of a Lie algebra $\mathfrak{g}$, the partial algebra $H_{par}$ coincides with the algebra $H$ itself, since every partial reresentation of $H$ is global \cite{ABV}. 
	
	\begin{thm}
		Let $H$ be a co-commutative Hopf algebra over a field $\K$. Then the partial Hopf algebra $H_{par}$ has the structure of a unital QISG.
	\end{thm}
	
	\begin{proof}
		First, one needs to define a comultiplication $\Delta : H_{par} \rightarrow H_{par} \otimes H_{par}$ which is multiplicative. For this, define the linear map
		\[
		\delta :  H  \rightarrow  H_{par}\otimes H_{par}, \ \ \ \ 
		h  \mapsto      [h_{(1)}]\otimes [h_{(2)}].
		\]
		One can prove that the map $\delta$ is a partial representation of $H$. For example, let us verify axiom (PR2):
		\begin{eqnarray*}
			\delta (h) \delta (k_{(1)}) \delta (S(k_{(2)})) & = & [h_{(1)}][k_{(1)}][S(k_{(4)})] \otimes [h_{(2)}][k_{(2)}][S(k_{(3)})] \\
			& = & [h_{(1)}][k_{(1)}][S(k_{(4)})] \otimes [h_{(2)}k_{(2)}][S(k_{(3)})] \\
			& = & [h_{(1)}][k_{(1)}][S(k_{(2)})] \otimes [h_{(2)}k_{(3)}][S(k_{(4)})] \\
			& = & [h_{(1)}k_{(1)}][S(k_{(2)})] \otimes [h_{(2)}k_{(3)}][S(k_{(4)})] \\
			& = & [h_{(1)}k_{(1)}][S(k_{(4)})] \otimes [h_{(2)}k_{(2)}][S(k_{(3)})] \\
			& = & \delta (hk_{(1)}) \delta (S(k_{(2)})) . 
		\end{eqnarray*}	
		Therefore, there exists a unique algebra morphism $\Delta : H_{par} \rightarrow H_{par} \otimes H_{par}$ given by
		\[
		\Delta ([h^1]\ldots [h^n]) =[h^1_{(1)}]\ldots [h^n_{(1)}] \otimes [h^1_{(2)}]\ldots [h^n_{(2)}] .
		\]
		In order to define the pseudo-antipode, consider the linear map 
		\[
		\widetilde{S} :  H  \rightarrow  H_{par}^{op} \ \ \ \ 
		h  \mapsto      [S(h)].              
		\]
		For every $h,k\in H$, we have
		\begin{eqnarray*}
			\widetilde{S} (h) \cdot_{op} \widetilde{S} (k_{(1)}) \cdot_{op} \widetilde{S} (S(k_{(2)})) & = & [S(h)] \cdot_{op} [S(k_{(1)})] \cdot_{op} [S(S(k_{(2)}))] \\
			& = & [S(S(k)_{(1)})][S(k)_{(2)}][S(h)] \\
			& = & [S(S(k)_{(1)})][S(k)_{(2)}S(h)] \\
			& = & [S(S(k_{(2)}))][S(hk_{(1)})] \\
			& = & [S(hk_{(1)})] \cdot_{op} [S(S(k_{(2)}))]\\
			& = & \widetilde{S} (hk_{(1)}) \cdot_{op} \widetilde{S}(S(k_{(2)})),
		\end{eqnarray*}	
		and the other axioms of partial representations are easily verified in the same way. Therefore $\widetilde{S}$ is a partial representation of $H$ in $H_{par}^{op}$, inducing a morphism of algebras $\mathcal{S} :H_{par} \rightarrow H_{par}^{op}$, or equivalently, an anti-morphism of algebras $\mathcal{S} :H_{par} \rightarrow H_{par}$ given by
		\[
		\mathcal{S} ([h^1]\ldots [h^n]) = [S(h^n)]\ldots [S(h^1)] .
		\]
		In order to verify the identities $I\ast \mathcal{S} \ast I =I$ and $\mathcal{S} \ast I \ast \mathcal{S}$, first note that, for any $h,k \in H$
		\begin{eqnarray*}
			[h]\varepsilon_k & = & [h][k_{(1)}][S(k_{(2)})] = [hk_{(1)}][S(k_{(2)})] \\
			& = & [h_{(1)}k_{(1)}][S(h_{(2)}k_{(2)})] [h_{(3)}k_{(3)}] [S(k_{(4)})] \\
			& = & [h_{(1)}k_{(1)}][S(h_{(2)}k_{(2)})] [h_{(3)}k_{(3)}S(k_{(4)})] \\
			& = & [h_{(1)}k_{(1)}][S(h_{(2)}k_{(2)})] [h_{(3)}] \\
			&= & \varepsilon_{h_{(1)}k}[h_{(2)}] .
		\end{eqnarray*}
		This implies, in particular, that the elements $\varepsilon_h$ do commute \cite{ABV}. Indeed,
		\begin{eqnarray*}
			\varepsilon_h \varepsilon_k & = & [h_{(1)}][S(h_{(2)})]\varepsilon_k =[h_{(1)}]\varepsilon_{S(h_{(3)})k}[S(h_{(2)})]\\
			& = & \varepsilon_{h_{(1)}S(h_{(4)})k}[h_{(2)}][S(h_{(3)})]=\varepsilon_{h_{(1)}S(h_{(2)})k}[h_{(3)}][S(h_{(4)})]\\
			& = & \varepsilon_k[h_{(1)}][S(h_{(2)})] =\varepsilon_k \varepsilon_h.
		\end{eqnarray*}
		Let us prove the identity $I\ast \mathcal{S} \ast I ([h^1]\ldots [h^n])=[h^1]\ldots [h^n]$ by induction on $n\geq1$. For $n=1$, we have
		\[
		I\ast \mathcal{S}\ast I ([h]) =[h_{(1)}][S(h_{(2)})][h_{(3)}] =[h_{(1)}S(h_{(2)})][h_{(3)}]=[h].
		\]
		Assume the claim valid for $n$, then
		\begin{eqnarray*}
			I\ast \mathcal{S} \ast I ([h^1]\ldots [h^{n+1}]) & = & [h^1_{(1)}]\ldots [h^{n+1}_{(1)}] [S(h^{n+1}_{(2)})]\ldots [S(h^{1}_{(2)})] [h^1_{(3)}]\ldots [h^{n+1}_{(3)}] \\
			& = & [h^1_{(1)}]\ldots [h^{n}_{(1)}] \varepsilon_{h^{n+1}_{(1)}}[S(h^{n}_{(2)})]\ldots [S(h^{1}_{(2)})] [h^1_{(3)}]\ldots [h^n_{(3)}][h^{n+1}_{(3)}] \\
			& = & \varepsilon_{h^1_{(1)}\ldots h^{n}_{(1)}h^{n+1}_{(1)}}[h^1_{(2)}]\ldots [h^{n}_{(2)}] [S(h^{n}_{(3)})]\ldots [S(h^{1}_{(3)})] [h^1_{(4)}]\ldots [h^n_{(4)}][h^{n+1}_{(3)}] \\
			& = & \varepsilon_{h^1_{(1)}\ldots h^{n}_{(1)}h^{n+1}_{(1)}}[h^1_{(2)}]\ldots [h^{n}_{(2)}] [h^{n+1}_{(2)}] \\
			& = & [h^1]\ldots [h^{n}] \varepsilon_{h^{n+1}_{(1)}}[h^{n+1}_{(2)}] \\
			& = & [h^1]\ldots [h^{n}] [h^{n+1}_{(1)}] [S(h^{n+1}_{(2)})][h^{n+1}_{(3)}] \\
			& = & [h^1]\ldots [h^{n+1}] .
		\end{eqnarray*}
		For the identity $\mathcal{S}\ast I \ast \mathcal{S}=\mathcal{S}$, consider $[h^1]\ldots [h^n] \in H_{par}$, then
		\begin{eqnarray*}
			\mathcal{S}\ast I \ast \mathcal{S}([h^1]\ldots [h^n]) & = & [S(h^n_{(1)})] \ldots [S(h^1_{(1)})] [h^1_{(2)}]\ldots [h^n_{(2)}][S(h^n_{(3)})] \ldots [S(h^1_{(3)})]\\
			& = & [S(h^n_{(3)})] \ldots [S(h^1_{(3)})] [S(S(h^1_{(2)}))]\ldots [S(S(h^n_{(2)}))][S(h^n_{(1)})] \ldots [S(h^1_{(1)})]\\ 
			& = & [S(h^n)_{(1)}] \ldots [S(h^1)_{(1)}] [S(S(h^1)_{(2)})]\ldots [S(S(h^n)_{(2)})][S(h^n)_{(3)}] \ldots [S(h^1)_{(3)}]\\
			& = & [S(h^n)] \ldots [S(h^1)] \\
			&= & \mathcal{S}([h^1]\ldots [h^n]).
		\end{eqnarray*}	
		
		Finally, in order to verify Axiom (QISG4), note that 
		\begin{eqnarray*}
			I\ast \mathcal{S} ([h^1]\ldots [h^n]) & = & [h^1_{(1)}]\ldots [h^n_{(1)}] [S(h^n_{(2)})]\ldots [S(h^1_{(2)})] \\
			& = & [h^1_{(1)}]\ldots [h^{n-1}_{(1)}] \varepsilon_{h^n}[S(h^{n-1}_{(2)})]\ldots [S(h^1_{(2)})] \\
			& = & \varepsilon_{h^1_{(1)}\ldots h^{n-1}_{(1)}h^n}[h^1_{(2)}]\ldots [h^{n-1}_{(2)}] [S(h^{n-1}_{(3)})]\ldots [S(h^1_{(3)})] \\
			& = & \varepsilon_{h^1_{(1)}\ldots h^{n-1}_{(1)}h^n}[h^1_{(2)}]\ldots [h^{n-2}_{(2)}]\varepsilon_{h^{n-1}_{(2)}} [S(h^{n-2}_{(3)})]\ldots [S(h^1_{(3)})] \\
			& = & \varepsilon_{h^1_{(1)}\ldots h^{n-1}_{(1)}h^n} \varepsilon_{h^1_{(2)}\ldots h^{n-2}_{(2)}h^{n-1}_{(2)}}[h^1_{(3)}]\ldots [h^{n-2}_{(3)}] [S(h^{n-2}_{(4)})]\ldots [S(h^1_{(4)})] \\
			& \vdots & \\
			& = & \varepsilon_{h^1_{(1)}\ldots h^{n-1}_{(1)}h^n} \varepsilon_{h^1_{(2)}\ldots h^{n-2}_{(2)}h^{n-1}_{(2)}} \ldots \varepsilon_{h^1_{(n)}},
		\end{eqnarray*}
		while, on the other hand,
		\begin{eqnarray*}
			\mathcal{S} \ast I ([h^1]\ldots [h^n]) & = & [S(h^n_{(1)})] \ldots [S(h^1_{(1)})] [h^1_{(2)}]\ldots [h^n_{(2)}] \\
			& = & [S(h^n_{(2)})] \ldots [S(h^1_{(2)})] [S(S(h^1_{(1)}))]\ldots [S(S(h^n_{(1)}))] \\
			& = & [S(h^n)_{(1)}] \ldots [S(h^1)_{(1)}] [S(S(h^1)_{(2)})]\ldots [S(S(h^n)_{(2)})] \\
			& = & \varepsilon_{S(h^n)_{(1)} \ldots S(h^2)_{(1)}S(h^1)} \varepsilon_{S(h^n)_{(2)} \ldots S(h^2)_{(2)}} \ldots \varepsilon_{S(h^n)_{(n)}} \\
			& = & \varepsilon_{S(h^n_{(n)}) \ldots S(h^2_{(n)})S(h^1)} \varepsilon_{S(h^n_{(n-1)}) \ldots S(h^2_{(n-1)})} \ldots \varepsilon_{S(h^n_{(1)})} .
		\end{eqnarray*}	
		As both expressions can be written in terms of combinations of products of elements $\varepsilon_x$, for $x\in H$, then they commute among themselves. Therefore, for a cocommutative Hopf algebra $H$, the universal Hopf algebra $H_{par}$ is a QISG. 
	\end{proof}

	\subsection{QISG and Hopf categories}
	
	Hopf   categories were introduced in \cite{BCV} in the context of categories enriched over the monoidal category of coalgebras of a strict braided monoidal category $\cV$. 
	In this section we will consider the  case $\cV = {}_\K \cM$, the symmetric monoidal category of left $\K$-modules over a commutative ring $\K$, and we will introduce Hopf categories as categories enriched over the monoidal category  $\coalg$ of $\K$-coalgebras, or simply ``$\coalg$-categories'' for short.
	
	Unraveling the definition, a (small) $\coalg$-category $H$ over the set $X$ 
	consists of a family $\{ H_{x,y} \}_{x,y \in X}$ of $\K$-coalgebras, with structure morphisms 
	\begin{equation}
		\Delta_{x,y} : H_{x,y} \to H_{x,y} \otimes H_{x,y}, \ \ \ \ \varepsilon_{x,y} : H_{x,y} \to \K, \label{eq.Kcoalgebra}
	\end{equation}
	for every $x, y \in X$,
	plus $\K$-linear mappings  $\mu_{x,y,z} : H_{x,y} \otimes H_{y,z} \to H_{x,z}$ and $\eta_{x}: \K \to H_{x,x}$, for every $x,y,z \in X$, such that 
	\begin{eqnarray}
		\mu_{x,y,t}\circ (H_{x,y}\ot \mu_{y,z,t})& =& \mu_{x,z,t}\circ (\mu_{x,y,z}\ot H_{z,t}) ;\label{eq.associativity}\\
		\mu_{x,x,y}\circ (\eta_x\ot H_{x,y})& =& H_{x,y}=\mu_{x,y,y}\circ ( H_{x,y}\ot \eta_y) \label{eq.identity}.
	\end{eqnarray}
	The coalgebra structure is required to be compatible with the multiplications $\mu_{x,y}$ and unit mappings $\eta_x$ in the following sense: first, $\Delta$ satisfies the equalities \begin{eqnarray}
		\Delta_{x,z} \circ \mu_{x,y,z} & = & (\mu_{x,y,z} \otimes \mu_{x,y,z})  \circ (H_{x,y} \otimes \tau_{H_{x,y},H_{y,z}} \otimes H_{y,z} ) \circ (\Delta_{x,y} \otimes \Delta_{y,z}), \label{eqn.delta.multiplicative}\\
		\Delta_{x,x} \circ \eta_x & = & \eta_x \otimes \eta_x,\label{eqn.delta.unital}
	\end{eqnarray}
	where $\tau_{H_{x,y},H_{y,z}}$ is the twist map 
	\[
	\tau_{H_{x,y},H_{y,z}} : H_{x,y}\otimes H_{y,z} \to H_{y,z} \otimes H_{x,y}, \ \ \ \ h \otimes k \mapsto k \otimes h;
	\]
	the equalities respective to the counit mappings are  
	\begin{eqnarray}
		\varepsilon_{x,y} \otimes \varepsilon_{y,z} & = & \varepsilon_{x,z} \circ \mu_{x,y,z},\label{eqn.epsilon.multiplicative} \\
		\varepsilon_{x,x} \circ \eta_x & = & \K \label{eqn.epsilon.unital}.
	\end{eqnarray}
	
	So let $H$  be a $\coalg$-category and consider the $\K$-module 
	$$\alg(H) = \oplus_{x,y \in H_0} H_{x,y} ,$$
	in which $H_0$ denotes the set of objects in the caegory $H$. Since $\alg(H)$ is a direct sum of coalgebras it has a canonical coalgebra structure as follows: denoting by 
	``$a_{x,y}$'' an element of the component $H_{x,y}$, the map 
	$\Delta: \alg(H) \to \alg(H) \otimes \alg(H) $ is defined component-wise by
	$$\Delta(a_{x,y}) = \Delta_{x,y} (a_{x,y}),$$ and  $\varepsilon: \alg(H) \otimes \K$ is defined by $$\varepsilon (a_{x,y}) = \varepsilon_{x,y} (a_{x,y})$$
	
	We can also define a product $\mu : \alg(H) \otimes \alg(H) \to \alg(H)$ by 
	\begin{center}
		$\mu (a_{x,y} \otimes b_{y,z})$ =   $\mu_{x,y,z} (a_{x,y} \otimes b_{y,z})$ 
		\ \ and \ \ 
		$\mu(a_{x,y} \otimes b_{w,z})  = 0$  whenever $y \neq w$.	
	\end{center}
	It can be verified that the triple $\alg(H)= (\alg(H), \mu, \Delta)$ satisfies conditions (QISG1) and (QISG2). 
	In fact, it follows from equalities \eqref{eq.associativity}-\eqref{eqn.epsilon.unital} that $\Delta$ and $\varepsilon$ are multiplicative, 
	and also that $\alg(H)$ is an algebra, which is unital if and only if $X$ is finite. 
	Moreover, in any case $\alg(H)$ has, at least, a system of local units: the idempotents $\eta_x(1)$ are mutually orthogonal and 
	the set of the finite sums $\eta_{x_1}(1) + \cdots + \eta_{x_n}(1)$, where $n \geq 1$ and the elements $x_1,x_2, \ldots, x_n$ are distinct, is a system of local units for $\alg(H)$.
	
	\bigskip 
	
	A \textbf{Hopf $\K$-category} is a $\coalg$-category $H$ with an antipode which, in this context, is a family of $\K$-linear maps $S_{x,y} : H_{x,y} \to H_{y,x}$ such that
	\begin{eqnarray}
		\mu_{x,y,x}\circ (H_{x,y}\ot S_{x,y})\circ \Delta_{x,y}&=& \eta_x\circ \varepsilon_{x,y}:\ H_{x,y}\to H_{x,x};\\
		\mu_{y,x,y}\circ ( S_{x,y}\ot H_{x,y})\circ \Delta_{x,y}&=& \eta_y\circ \varepsilon_{x,y}:\ H_{x,y}\to H_{y,y},
	\end{eqnarray}
	for all $x,y\in X$. This family induces a $\K$-linear map $S : \alg(H) \to \alg(H)$ which satisfies, in 
	Sweedler notation, the equalities
	\begin{eqnarray} 
		(h_{x,y})_{(1)} S( (h_{x,y})_{(2)})&=&  \varepsilon_{x,y}(h_{x,y}) \eta_y(1), \label{antipode.1} \\
		S( (h_{x,y})_{(1)} )(h_{x,y})_{(2)}&=&  \varepsilon_{x,y}(h_{x,y}) \eta_x(1),
		\label{antipode.2}
	\end{eqnarray}
	for every $h_{x,y} \in H_{x,y}$ and for all $x,y \in X$.
	
	We claim that $\alg(H)$ is a quantum inverse semigroup with $S$ as its pseudo-antipode. 
	Axiom (QISG3)(ii) follows easily from 
	\eqref{antipode.1} and \eqref{antipode.2}. 
	(QISG3)(i) follows from \cite[Lemma 3.6]{BCV}, where it is proved that
	\begin{eqnarray}
		S_{x,z}\circ \mu_{x,y,z}&=&\mu_{z,y,x}\circ (S_{y,z}\ot S_{x,y})\circ \tau_{H_{x,y},H_{y,z}}\\
		\Delta_{y,x}\circ S_{x,y}&=&\tau_{H_{y,x},H_{y,x}}\circ (S_{x,y}\ot S_{x,y})\circ \Delta_{x,y}.
	\end{eqnarray}
	Hence for $h_{x,y} \in H_{x,y},k_{y,z} \in H_{y,z}$ we have 
	\[
	S(h_{x,y}k_{y,z}) = S(k_{y,z}) S(h_{x,y}), \ \ \ \ S((h_{x,y})_{(1)}\otimes S(h_{x,y})_{(2)} =S((h_{x,y})_{(2)})\otimes S((h_{x,y})_{(1)}) 
	\]
	and it follows that $S$ is antimultiplicative and anticomultiplicative.
	
	Finally, axiom (QISG4) also follows from \eqref{antipode.1} and \eqref{antipode.2}:
	\begin{eqnarray*}
		(h_{x,y})_{(1)} \mathcal{S}((h_{x,y})_{(2)}) \mathcal{S}((k_{z,w})_{(1)})(k_{z,w})_{(2)} & =& \varepsilon_{x,y} (h_{x,y}) \varepsilon_{z,w} (k_{z,w})  \eta_x(1) \eta_w(1)  \\
		& =&  \varepsilon_{z,w} (k_{z,w}) \varepsilon_{x,y} (h_{x,y}) \eta_w(1)  \eta_x(1)   \\
		& = &  \mathcal{S}((k_{z,w})_{(1)})(k_{z,w})_{(2)} (h_{x,y})_{(1)} \mathcal{S}((h_{x,y})_{(2)}).
	\end{eqnarray*}     
	Therefore, if $H$ is a Hopf $\K$-category then $\alg(H)$ is a counital quantum inverse semigroup with anticomultiplicative pseudo-antipode.

	We mention that if a Hopf category $H$ has a finite set of objects then we can obtain this same QISG structure on $\alg(H)$ by other means: in (\cite{BCV}, Prop 7.1) it is proved that, in this case, $\alg(H)$ is a weak Hopf algebra, and we know by Example \ref{weak} that every weak Hopf algebra has a structure of quantum inverse semigroup; this structure of QISG coincides with the one introduced above.
	
	\begin{exmp}
		Consider a Hopf category $H$ whose set of objects is $H_0=\mathbb{N}$ and for $i,j\in \mathbb{N}$, $H_{i,j}=\Bbbk$ with the trivial coalgebra structure ($\Delta (\lambda)=\lambda \otimes 1=1\otimes \lambda$, for $\lambda \neq 0$, $\Delta (0)=0\otimes 0$,  $\epsilon =Id_{\Bbbk}$). The multiplication maps $\mu_{i,j,k} :H_{i,k} \otimes H_{k,j} \rightarrow H_{i,j}$ is given by the multiplication on the field $\Bbbk$ and the unit map $\eta_i :\Bbbk \rightarrow H_{i,i}$ is the identity map on the field $\Bbbk$. Finally, the antipode map $S_{i,j} :H_{i,j} \rightarrow H_{j,i}$ is simply the identity linear map on $\Bbbk$.
		
		The algebra $\alg (H)$ is the algebra $M_{\mathbb{N}}(\Bbbk )_{00}$, which is the algebra of row and column finite $\mathbb{N} \times \mathbb{N}$ matrices. The QISG structure on this specific Hopf category can be written as
		\[
		\Delta \left( (a_{ij})_{i,j} \right) =({a_{ij}}_{(1)} )_{i,j} \otimes ({a_{ij}}_{(2)} )_{i,j} ,
		\]
		in which ${a_{ij}}_{(1)}  \otimes {a_{ij}}_{(1)} =a_{ij} \otimes 1$ for $a_{ij} \neq 0$ and ${a_{ij}}_{(1)}  \otimes {a_{ij}}_{(1)} =0$, for $a_{ij} =0$,
		\[
		\varepsilon \left( (a_{ij})_{i,j} \right)=(a_{ij})_{i,j} ,
		\]
		and 
		\[
		S \left( (a_{ij})_{i,j} \right)=(a_{ji})_{i,j} .
		\]
		It is easy to see that this satisfy (QISG1), (QISG2), (QISG3) and (QISG4) of Definition \ref{quantuminversesemigroup}.
		
		If the set of objects $H_0$ is a finite set, $H_0 =\{ 1, \ldots , n \}$ then the algebra $\alg (H)$ coincides with the standard weak Hopf algebra $M_n (k)$ of $n\times n$ matrices.
	\end{exmp}

	\section{Generalized bisections on Hopf algebroids}
	
	The interplay between groupoids and inverse semigroups has been vastly explored in the literature \cite{Ehresmann,law,Nambooripad,Paterson,Schein}. One of the most important sources of inverse semigroups associated to groupoids are the bisections of \'etale topological groupoids \cite{buss,ruy3,MR}. A topological groupoid is a groupoid which is a topological space and whose structural maps (source, target, multiplication, unit map and inversion) are topological maps. A topological groupoid is \'etale if the source and target maps are local homeomorphisms. Given an \'etale groupoid $\mathcal{G}$, a local bisection of $\mathcal{G}$ is an open subset $U\subseteq \mathcal{G}$ such that the restriction of the source map to it is injective (this implies automatically that the target map restricted to $U$ is also injective). Denoting by $\mathfrak{B}(\mathcal{G})$ the set of all bisections of $\mathcal{G}$, one can prove that this set has the following structure of an inverse semigroup: 
	\begin{enumerate}
		\item Given $U,V \in \mathfrak{B}(\mathcal{G})$, define their multiplication as
		\[
		UV =\{ \gamma \delta \in \mathcal{G} \; |\; (\gamma , \delta ) \in U\times V \; \mbox{ and } \; s(\gamma)=t(\delta) \} .
		\]
		\item For $U\in \mathfrak{B}(\mathcal{G})$, define its pseudo-inverse as
		\[
		U^\ast =\{ \gamma^{-1}\in \mathcal{G} \; |\; \gamma \in U \} .
		\]
	\end{enumerate}
	
	\'Etale groupoids have nice properties relative to their bisections. For example, the set of bisections form a basis for the topology of the entire groupoid, being the idempotents of the inverse semigroup of bisections correspondent to the open subsets of the unit space $\mathcal{G}^{(0)}$ of the groupoid, which is, in its turn, a clopen subset of the entire groupoid.
	
	For what comes next, one needs a more algebraic characterization of bisections of groupoids. 
	
	\begin{defi}\label{bisection}
		A local bisection of a groupoid $\mathcal{G}$ is a pair $(u,X)$ in which $X$ is a subset of $\mathcal{G}^{(0)}$ and $u: X\rightarrow \mathcal{G}$ is a function such that
		\begin{enumerate}
			\item[(i)] $s\circ u=\mbox{Id}_X$.
			\item[(ii)]  $t\circ u: X\rightarrow t(u(X))$ is a bijection.
		\end{enumerate}
		The set $X$ is called the domain of the bisection $(u,X)$. A global bisection is a local bisection whose domain is $X=\mathcal{G}^{(0)}$.
	\end{defi}
	
	Note that, item (ii) implies that the function $u: X\rightarrow \mathcal{G}$ is injective. Denote again by $\mathfrak{B}(\mathcal{G})$ the set of the local bisections of the groupoid $\mathcal{G}$ and by $\mbox{Gl}\mathfrak{B}(\mathcal{G})$ the set of its global bisections.
	
	\begin{rmk}
		\begin{enumerate}
			\item[(i)] The two notions of a bisection, as a subset of the groupoid restricted to what the source map is injective and as a pair of a subset of the unit set and a function are in fact related. On one hand, given a subset $U\subseteq \mathcal{G}$ for which $s|_{U}:U\rightarrow \mathcal{G}^{(0)}$ is injective, define $X=s(U) \subseteq \mathcal{G}^{(0)}$ and $u:X\rightarrow \mathcal{G}$ as the inverse of $s|_{U}$. On the other hand, given a pair $(u,X)$, as in Definition \ref{bisection}, define $U=u(X)$, as $u$ is already injective, the corestriction $u:X\rightarrow U$ is bijective. As the left inverse of $u$ is $s$, by definition, then it is the inverse of that corestriction, making $s|U$ injective.
			\item[(ii)] A topological version of the Definition \ref{bisection} can be seen in \cite{Resende}. There, the subset $X\subseteq \mathcal{G}^{(0)}$ is an open subset and the map $u:X\rightarrow \mathcal{G}$ is a continuous function. When we consider only sets with the discrete topology, the two definitons coincide.
		\end{enumerate}
	\end{rmk}
	
	Several instances of the following result have already appeared in the literature (see, \cite{garner} Example 17, for a version closer to our approach), but we present a full proof here in order to introduce some notations and techniques which will be useful throughout this work.  
	
	\begin{prop}
		Let $\mathcal{G}$ be a groupoid, then the set $\mathfrak{B}(\mathcal{G})$, of its local bisections, defines an inverse semigroup and the set $\mbox{Gl}\mathfrak{B}(\mathcal{G})$, of its global bisections, is a group.
	\end{prop}
	
	\begin{proof}
		Consider $(u,X)$ and $(v,Y)$ two local bisections of $\mathcal{G}$, define their product as $(u,X)\cdot (v,Y)=(uv, Z)$, in which
		\[
		Z=(t\circ v)^{-1} (t\circ v (Y) \cap X ) \quad \mbox{ and } \quad	(uv) (y) = u(t\circ v(y)) v(y) .
		\]
		This product is associative. Indeed, for $(u,X),(v,Y),(w,Z) \in \mathfrak{B}(\mathcal{G})$, we have
		\begin{eqnarray*}
			((u,X)\cdot (v,Y))\cdot (w,Z) & = & (uv, (t\circ v)^{-1}(t\circ v (Y) \cap X) ) \cdot (w,Z) \\
			& = & ((uv)w,(t\circ w)^{-1}(t\circ w(Z)\cap (t\circ v)^{-1} (t\circ v (Y)\cap X ))) ,
		\end{eqnarray*}	
		while
		\begin{eqnarray*}
			(u,X)\cdot ((v,Y)\cdot (w,Z)) & = & (u,X)\cdot (vw, (t\circ w)^{-1}(t\circ w (Z) \cap Y) ) \\
			& = & (u(vw),(t\circ vw)^{-1}(t\circ vw((t\circ w)^{-1}(t\circ w (Z) \cap Y) )\cap X )) ,
		\end{eqnarray*}
		In order to show that these bisections are equal, first note that, for any $z\in Z$
		\[
		t\circ vw (z) =t(v(t\circ w (z))w(z)) =t(v(t\circ w (z)))=t\circ v\circ t\circ w (z) .
		\]
		Now, take $z\in (t\circ vw)^{-1}(t\circ vw((t\circ w)^{-1}(t\circ w (Z) \cap Y) )\cap X )$, then $z=(t\circ vw)^{-1} (x)$ for an element $x\in  t\circ vw((t\circ w)^{-1}(t\circ w (Z) \cap Y) )\cap X$, or $x=t\circ vw (z)=t\circ v \circ t \circ w (z)$. Denote by $y=t\circ w (z)$ this element belongs to $t\circ w (Z) \cap Y$ then $x=t\circ v (y)\in t\circ v (Y) \cap X$. Finally $y=(t\circ v)^{-1}(x) =t\circ w (z) \in t\circ w(Z)\cap (t\circ v)^{-1} (t\circ v (Y)\cap X )$, Therefore, $z\in (t\circ w)^{-1}(t\circ w(Z)\cap (t\circ v)^{-1} (t\circ v (Y)\cap X ))$. Reciprocally, one can see that $z\in (t\circ w)^{-1}(t\circ w(Z)\cap (t\circ v)^{-1} (t\circ v (Y)\cap X ))$ implies in $z\in (t\circ vw)^{-1}(t\circ vw((t\circ w)^{-1}(t\circ w (Z) \cap Y) )\cap X )$, proving the equality between the domains of these two bisections, henceforth denoted by $T$, only for sake of simplicity.
		
		Now, for $z\in T$,
		\[
		u(vw)(z)  =  u(t\circ vw (z))vw(z) =u(t\circ v \circ t \circ w (z))v(t\circ w (z))w(z) 
		\]
		and
		\[
		(uv)w (z) = uv(t\circ w (z))w(z) =u(t\circ v\circ t \circ w (z))v(t\circ w (z))w(z) .
		\] 
		Therefore $((u,X)\cdot (v,Y))\cdot (w,Z) =(u,X)\cdot ((v,Y)\cdot (w,Z))$.
		
		For any bisection $(u,X) \in \mathfrak{B}(\mathcal{G})$ define $(u,X)^* =(\overline{u},t\circ u(X))$, in which, for any $x\in X$, $\overline{u} (t\circ u (x))=u(x)^{-1}$. Then for any $x\in X$,
		\[
		\overline{u}u (x)=\overline{u}(t\circ u (x))u(x) =u(x)^{-1}u(x) = i(s(u(x))) = i(x) ,
		\]
		in which $ i:\mathcal{G}^{(0)} \rightarrow \mathcal{G}$ is the unit map of the groupoid. We conclude that
		\[
		u\overline{u} u(x) =u(t\circ \overline{u}u (x))\overline{u}u(x) =u(t\circ  i(x)) i(x) =u(x) 
		\]
		and
		\begin{eqnarray*}
			\overline{u} u\overline{u} (t\circ u(x)) & = & \overline{u}u (t\circ \overline{u} (t\circ u (x))) \overline{u}(t\circ u (x)) \\ 
			& = & \overline{u}u (t(u(x)^{-1}))  \overline{u}(t\circ u (x)) = \overline{u}u (x)  \overline{u}(t\circ u (x)) \\ 
			& = &  i(x) \overline{u}(t\circ u (x)) =\overline{u}(t\circ u (x)) .
		\end{eqnarray*}
		Therefore
		\[
		(u,X)\cdot (u,X)^* \cdot (u,X) =(u,X)\cdot (\overline{u},t\circ u (X))\cdot (u,X) =(u,X)\cdot (\overline{u}u,X)=(u,X)\cdot ( i,X) =(u,X)
		\]
		and
		\[
		(u,X)^* \cdot (u,X) \cdot (u,X)^* =( i,X) \cdot (\overline{u} ,t\circ u (X)) =(\overline{u} ,t\circ u (X))=(u,X)^* .
		\]
		It remains to prove that the idempotents in $\mathfrak{B}(\mathcal{G})$ commute among themselves. If $(u,X)$ is an idempotent element, then
		\[
		(u,X)=(u,X)\cdot (u,X) =(uu,(t\circ u)^{-1}(t\circ u(X) \cap X)) ,
		\]
		implying that $t\circ u (X)=X$ and $u(t\circ u(x))u(x)=u(x)$. Multiplying the last equality on the right by $u(x)^{-1}$ we end up with $u(t\circ u(x))= i(t\circ u (x))$. As $t\circ u (X)=X$, for any $x \in X$ there exists $y\in X$ such that $x=t\circ u(y)$, then $u= i$ and $(u,X)=( i,X)$. Multiplying two of such idempotents we have
		\[
		( i,X)\cdot ( i,Y) =( i,X\cap Y) =( i,Y) \cdot ( i,X) .
		\]
		Therefore, $\mathfrak{B}(\mathcal{G})$ is an inverse semigroup.
		
		The global bisections are of the form $(u,\mathcal{G}^{(0)})$ and clearly, global bisections $\mbox{Gl}\mathfrak{B}(\mathcal{G})$ form a subsemigroup of $\mathfrak{B}(\mathcal{G})$.But the only idempotent possible in $\mbox{Gl}\mathfrak{B}(\mathcal{G})$ is the unit $(i,\mathcal{G}^{(0)})$. An inverse semigroup with only one idempotent is a group, therefore $\mbox{Gl}\mathfrak{B}(\mathcal{G})$ is a group.
	\end{proof}	
	
	\subsection{Biretractions on commutative Hopf algebroids}
	
	In what follows, unless stated otherwise, we shall consider commutative Hopf algebroids over a commutative base algebra.
	
	\begin{defi}
		Let $\mathcal{H}$ be a commutative Hopf algebroid over the base algebra $A$. A \emph{local biretraction} in $\mathcal{H}$ is a linear and multiplicative map $\alpha :\mathcal{H}\rightarrow A$ such that
		\begin{itemize}
			\item[(BRT1)] $\alpha \circ s (a) =a \,\alpha (1_{\mathcal{H}})$ for every $a\in A.$
			\item[(BRT2)] There exists $e^\alpha\in A$ such that $\alpha\circ t(e^\alpha)=\alpha(1_\mathcal{H})$ and
			$$\alpha \circ t|_{A\, e^\alpha} : A\, e^\alpha \longrightarrow A\, \alpha(1_\mathcal{H})$$ is a bijection.
		\end{itemize} 
		A local biretraction $\alpha$ is global if $\alpha (1_{\mathcal{H}})=1_A$. Denote the set of local biretractions of $\mathcal{H}$ by $\mathfrak{Brt} (\mathcal{H}, A)$ and the set of global biretractions of $\mathcal{H}$ by $\mbox{Gl}\mathfrak{Brt} (\mathcal{H}, A)$.
	\end{defi}

	\begin{prop}
		Let $(\mathcal{H}, s,t,\Delta,\,\varepsilon)$ be a commutative Hopf algebroid over $A$ and let 
		$\alpha: \mathcal{H}\rightarrow A$ be a local biretraction. 
		\begin{itemize}
			\item[(1)] $\alpha(1_\mathcal{H})$ is an idempotent in $A$ and $\alpha(\mathcal{H})$ coincides with the ideal $A\,\alpha(1_\mathcal{H}).$
			\item[(2)] The element $e^\alpha$ is idempotent.
			\item[(3)] The element $e^\alpha$ is unique. 
		\end{itemize}
	\end{prop}
	\begin{proof}
		\begin{itemize}
			\item[(1)]  $\alpha(1_\mathcal{H})$ is an idempotent because $\alpha$ is multiplicative. Moreover, for every $h\in \mathcal{H}$ and $a\in A,$
			$$\alpha(h)=\alpha(h)\,\alpha(1_\mathcal{H}) \in A\,\alpha(1_\mathcal{H}) \,\,\,\,\,\textrm{ and }\,\,\,\,\,a\,\alpha(1_\mathcal{H})=\alpha\circ s(a)\in \alpha(\mathcal{H}).$$
			So, we have  that the image $\alpha(\mathcal{H})$ coincides with the ideal $A\,\alpha(1_\mathcal{H}) \trianglelefteq A.$ Also, note that $\alpha(1_\mathcal{H})$ is the unity of the ideal $A\,\alpha(1_\mathcal{H}).$ 
			
			\item[(2)] From \emph{(BRT2)} it follows that 
			\begin{align*}
				\alpha\circ t(e^\alpha) =\alpha(1_\mathcal{H})  = \alpha(1_\mathcal{H}) \,\alpha( 1_\mathcal{H})& = \alpha\circ t(e^\alpha)\, \alpha\circ t (e^\alpha)= \alpha\circ t (e^\alpha\, e^\alpha).
			\end{align*}
			
			Since $\alpha \circ t|_{A\, e^\alpha}$ is bijective and $e^\alpha \, e^\alpha \in A\, e^\alpha,$ we have that
			$$e^\alpha\, e^\alpha=e^\alpha.$$
			
			\item[(3)] Suppose that there exist $e^\alpha$ and $f^\alpha$ in $A$ such that $\alpha\circ t(e^\alpha) =\alpha(1_\mathcal{H})=\alpha\circ t(f^\alpha)$ and the maps
			$$\alpha \circ t|_{A\, e^\alpha} : A\, e^\alpha \longrightarrow A\, \alpha(1_\mathcal{H})$$
			and
			$$\alpha \circ t|_{A f^\alpha} : A f^\alpha \longrightarrow A\, \alpha(1_\mathcal{H})$$
			are both bijections. Then,
			$$\alpha\circ t \left(e^\alpha  f^\alpha\right) = \alpha\circ t\left( e^\alpha\right) \alpha\circ t\left( f^\alpha\right) = \alpha(1_\mathcal{H})=\alpha\circ t\left( e^\alpha\right) .$$
			Also, we have
			\[
			\alpha\circ t \left(e^\alpha  f^\alpha\right) =\alpha\circ t\left( f^\alpha\right)
			\]
			Since the element $e^\alpha  f^\alpha$ lies in both ideals $A\, e^\alpha$ and  $A f^\alpha,$ and $\alpha \circ t$ is injective, we obtain
			$$e^\alpha= e^\alpha f^\alpha = f^\alpha.$$
			Therefore, the element $e^\alpha$ from \emph{(BRT2)} is unique.       
		\end{itemize}
	\end{proof}
	
	\begin{rmk}
		\begin{itemize}
			\item[(i)] Observe that a local biretraction $\alpha$, being multiplicative and satisfying (BRT1) is automatically a morphism of right $A$-modules between $\mathcal{H}$ and $A$. Indeed, for $h\in \mathcal{H}$ and $a\in A$,
			\begin{eqnarray*}
				\alpha (h\triangleleft a) & = & \alpha (hs(a)) =\alpha (h) \alpha (s(a))\\
				& = & \alpha (h) a \alpha ( 1_{\mathcal{H}})=\alpha (h) a.
			\end{eqnarray*}
			Therefore, $\mathfrak{Brt}(\mathcal{H},A) \subseteq \text{Hom}_A (\mathcal{H}, A)$.
			\item[(ii)] For a local biretraction $\alpha: \mathcal{H}\rightarrow A,$ the map $\alpha \circ t|_{A\, e^\alpha} : A\, e^\alpha \longrightarrow A\, \alpha(1_\mathcal{H})$ is an element of the inverse semigroup $\mathcal{I}(A)$ of the partial bijections between unital ideals of $A.$   
			\item[(iii)]
			For a commutative Hopf algebroid over a domain $A$ we only have global biretractions, since the only idempotent element in $A$ is $1_A.$
		\end{itemize}
	\end{rmk}
	
	As we have seen before, the set of local bisections of a groupoid $\mathcal{G}$ is an inverse semigroup. Let us explore deeply the algebraic structure of the set of biretractions of a commutative Hopf algebroid.
	
	\begin{thm} \label{theorem.Brt(H,A).is.a.regular.monoid}
		Let $(\mathcal{H} , s, t, \Delta , \varepsilon , S)$ be a commutative Hopf algebroid  over a commutative algebra $A$. Then the set $\mathfrak{Brt} (\mathcal{H}, A)$ of local biretractions of $\mathcal{H}$ is a regular monoid. 
	\end{thm}
	\begin{proof}
		First, let us construct a convolution multiplication in the set of local biretractions of $\mathcal{H}$. As the local biretractions are morphisms of right $A$-modules, we use the following convolution product: for $\alpha ,\beta \in \mathfrak{Brt} (\mathcal{H}, A)$ define for any $h\in \mathcal{H}$
		\[
		(\alpha \ast \beta )(h)= \beta ( \alpha (h_{(1)}) \triangleright h_{(2)})  =\beta \circ t \circ \alpha (h_{(1)})\, \beta (h_{(2)}) .
		\]
		
		In what follows, in order to make the reading more fluid, we are omitting the sign of composition of maps whenever it is clear, since the the maps are already applied to elements. The convolution product $\alpha\ast \beta$ is a local biretraction, because $\alpha\ast \beta$ is multiplicative and satisfies (BRT1) and (BRT2).
		\begin{itemize}
			\item[(BRT1)] for each $a\in A,$
			\begin{align*}
				(\alpha\ast \beta)\circ s (a) & = \beta t \alpha(1_\mathcal{H})\,\beta(s(a))\\
				& = a\, \beta t \alpha(1_\mathcal{H}) \beta(1_\mathcal{H})\\
				& = a \,(\alpha\ast \beta)(1_\mathcal{H}).
			\end{align*} 
			\item[(BRT2)] Since $t$ represents the left action, we can use the fact that $\Delta  (t(a)) =t(a)\, \otimes_A 1_{\mathcal{H}}$ for every $a\in A,$  which implies that 
			\begin{equation} \label{equation.precomposition.convolution.with.t}
				(\alpha\,\ast\,\beta)\,\circ\, t = \beta\circ t\circ\alpha\circ t.    
			\end{equation}
			Therefore, the domain of $(\alpha \ast \beta )\circ t$ must be the preimage by $\alpha \circ t$ of the intersection of the ideal $A\alpha (1_{\mathcal{H}})$ with the ideal $Ae^\beta$, then the idempotent associated to $\alpha \ast \beta$ must be $(\alpha \circ t)^{-1} (e^\beta \alpha (1_{\mathcal{H}}))$. Indeed,
			\begin{align*}
				(\alpha\ast \beta )\circ t \left( (\alpha\circ t)^{-1} \left(e^\beta\, \alpha(1_\mathcal{H})\right)\right) & = (\beta\circ t) (\alpha\circ t)  (\alpha\circ t)^{-1} \left( e^\beta \, \alpha(1_\mathcal{H})\right)\\
				& = \beta\circ t \left( \alpha(1_\mathcal{H})\, e^\beta\right)\\
				& =  \beta t \alpha (1_\mathcal{H})\,\beta (1_\mathcal{H})\\
				& = (\alpha\ast \beta )(1_\mathcal{H}).
			\end{align*}
			Here, we are simplifying the notation by using $(\alpha\circ t)^{-1}= \left((\alpha\circ t)|_{A\, e^\alpha}\right)^{-1}.$
			
			We need to check that the map
			$$(\alpha\ast\beta) \circ t|_{A\, (\alpha\circ t)^{-1} \left(e^\beta\, \alpha(1_\mathcal{H})\right)} : A\, (\alpha\circ t)^{-1} \left(e^\beta\, \alpha(1_\mathcal{H})\right) \longrightarrow A\, (\alpha\ast\beta)(1_\mathcal{H})$$
			is a bijection. In fact,
			\vspace{10pt}
			\begin{itemize}
				\item[$\bullet$] $(\alpha\,\ast\,\beta)\, \circ \,t\,|_{A\, (\alpha\,\circ\, t)^{-1} \left(e^\beta\, \alpha(1_\mathcal{H})\right)}$
				is surjective: for each $a\in A,$
				\begin{align*}
					a\,(\alpha\ast\beta)(1_\mathcal{H}) & =a\, (\beta\circ t) (\alpha(1_\mathcal{H}))\,\beta(1_\mathcal{H}) \\
					& = a\, (\beta\circ t) \left(\alpha(1_\mathcal{H})\, e^\beta\right)\\
					& =  (\beta\circ t) (\beta\circ t)^{-1} \left( a\,(\beta\circ t) \left(\alpha(1_\mathcal{H}) \, e^\beta \right)\, \beta(1_\mathcal{H}) \right)\\
					& = (\beta \circ t) \left( (\beta \circ t)^{-1} (a \beta (1_{\mathcal{H}}))  \alpha (1_{\mathcal{H}}) e^\beta \right) \\
					& = (\beta \circ t) (\alpha \circ t) (\alpha \circ t)^{-1} \left( (\beta \circ t)^{-1} (a \beta (1_{\mathcal{H}}))  \alpha (1_{\mathcal{H}}) e^\beta \right) \\
					& = (\beta \circ t) (\alpha \circ t) \left( (\alpha \circ t)^{-1} \left( (\beta \circ t)^{-1} (a \beta (1_{\mathcal{H}}))  \alpha (1_{\mathcal{H}}) \right) (\alpha \circ t)^{-1} \left(\alpha (1_{\mathcal{H}}) e^\beta \right) \right) \\
					& = (\alpha \ast \beta )\circ t \left( (\alpha \circ t)^{-1} \left( (\beta \circ t)^{-1} (a \beta (1_{\mathcal{H}}))  \alpha (1_{\mathcal{H}}) \right) (\alpha \circ t)^{-1} \left(\alpha (1_{\mathcal{H}}) e^\beta \right)\right) 
				\end{align*}
				
				\item[$\bullet$] $(\alpha\ast\beta) \circ t|_{A\, (\alpha\circ t)^{-1} \left(e^\beta\,\alpha(1_\mathcal{H})\right)}$ is injective: suppose that, for some $a\in A,$
				$$(\alpha\ast\beta)\circ t \left(a\, (\alpha\circ t)^{-1} \left(e^\beta\,\alpha(1_\mathcal{H})\right) \right) = 0.$$
				Since $\alpha\circ t|_{A\, \alpha(1_\mathcal{H})}$ and $\beta\circ t|_{A\, \beta(1_\mathcal{H})}$ are injective, we have
				
				\begin{align*}
					& 0 = (\beta\circ t) (\alpha\circ t) \left(a\, (\alpha\circ t)^{-1} \left(e^\beta\,\alpha(1_\mathcal{H})\right) \right) \\
					& = (\beta \circ t)(\alpha\circ t) \left(a\, e^\alpha (\alpha\circ t)^{-1} \left(e^\beta\,\alpha(1_\mathcal{H})\right) \right) \\
					& = (\beta \circ t) \left( (\alpha\circ t) \left(a\, e^\alpha\right) e^\beta\,\alpha(1_\mathcal{H}) \right)\\
					\Rightarrow\,\,\,\, & 0 =  (\alpha\circ t) \left(a\, e^\alpha\right) e^\beta\,\alpha(1_\mathcal{H})  = \alpha\circ t \left( a\, e^{\alpha} (\alpha\circ t)^{-1}\left(e^\beta \,\alpha(1_\mathcal{H})\right)\, \right)\\
					\Rightarrow\,\,\,\, & 0 = a\, e^\alpha (\alpha\circ t)^{-1}\left(e^\beta \, \alpha(1_\mathcal{H})\right)  = a\, \,(\alpha\circ t)^{-1}\left(e^\beta \,\alpha(1_\mathcal{H})\right).
				\end{align*}
				
			\end{itemize}
			Therefore, $\alpha\ast\beta$ is a local biretraction with $e^{\alpha\ast\beta}=(\alpha \circ t)^{-1} \left(e^\beta\, \alpha(1_\mathcal{H})\right).$
		\end{itemize}

		The associativity of the convolution product of biretractions of $\mathcal{H}$ can be viewed directly once we know that local birretractions are right $A$ module morphisms and the convolution product on $\text{Hom}_A (\mathcal{H} ,A)$, given by 
		\[
		\alpha \ast \beta (h)= \beta (\alpha (h_{(1)})\triangleright h_{(2)}) ,
		\]
		is associative \cite{BW}. Nevertheless, for sake of completeness, we are going to check directly the associtivity here using the notations and conventions for biretractions. Consider $\alpha ,\beta ,\gamma \in \mathfrak{Brt} (\mathcal{H}, A)$, then for any $h\in \mathcal{H}$,
		\begin{eqnarray*}
			((\alpha \ast \beta ) \ast \gamma) (h) & = & \gamma \, t \,(\alpha\ast \beta ) (h_{(1)})\, \gamma (h_{(2)}) \\
			& = & \gamma \, t \, (\beta \, t \, \alpha (h_{(1)}) \,\beta (h_{(2)}))\, \gamma (h_{(3)}) \\
			& = & \gamma \, t \, \beta \, t \, \alpha (h_{(1)}) \,\gamma  \, t  \,\beta (h_{(2)})) 
			\gamma (h_{(3)}) \\
			& = & (\beta \ast \gamma ) \, t \, \alpha (h_{(1)}) \,(\beta \ast \gamma ) (h_{(2)})\\
			& =& (\alpha \ast (\beta \ast \gamma ))(h).
		\end{eqnarray*}	
		
		The counit $\varepsilon :\mathcal{H} \rightarrow A$ is a global biretraction, because it is linear, multiplicative, $\varepsilon (1_{\mathcal{H}})=1_A$ and $\varepsilon\circ t= \varepsilon\circ s = \textrm{Id}_A.$ The counit $\varepsilon$ is the unit for the convolution product. Indeed, for any local biretraction $\alpha$ and any $h\in \mathcal{H}$, we have
		\begin{eqnarray*}
			\varepsilon \ast \alpha (h) & = & \alpha \, t \, \varepsilon (h_{(1)}) \,\alpha (h_{(2)})\\
			& = & \alpha \left( t(\varepsilon (h_{(1)})) \,h_{(2)} \right) \\
			& = & \alpha (h) ,
		\end{eqnarray*}	
		and
		\begin{eqnarray*}
			\alpha \ast \varepsilon (h) & = & \varepsilon \, t \, \alpha (h_{(1)})\, \varepsilon (h_{(2)}) \\
			& = & \alpha (h_{(1)})\, \varepsilon (h_{(2)}) \\
			& = & \alpha (h_{(1)}) \,\alpha \, s  (\varepsilon (h_{(2)})) \\
			& = & \alpha \left( h_{(1)} \, s (\varepsilon (h_{(2)})) \right) \\
			& = & \alpha (h).
		\end{eqnarray*}	
		
		Therefore, the set $\mathfrak{Brt} (\mathcal{H}, A)$ is a monoid relative to the above defined convolution product.
		
		Now, we have to define a pseudo-inverse for any biretraction $\alpha \in \mathfrak{Brt} (\mathcal{H}, A)$. Define
		\[
		\alpha^* = (\alpha \circ t)^{-1} \circ \alpha \circ S,
		\]
		where $S:\mathcal{H}\rightarrow\mathcal{H}$ is the Hopf algebroid map from Definition \ref{hopfalgebroid}. Then $\alpha^*$ is well-defined, because $(\alpha\, \circ\, t)^{-1}$ is applied to an element belonging to $\alpha (\mathcal{H})=A\alpha (1_{\mathcal{H}})$ and it is multiplicative. Observe that
		$$\alpha^\ast (1_\mathcal{H}) = (\alpha\circ t)^{-1}\circ\alpha \circ S(1_\mathcal{H})= (\alpha\circ t)^{-1}\circ\alpha (1_\mathcal{H}) = e^\alpha.$$
		So, $\alpha^\ast $ is a biretraction, because
		\begin{align*}
			\alpha^\ast \circ s(a) & = (\alpha\circ t)^{-1}\, \alpha\, (S\circ s)(a)\\
			& = (\alpha\circ t)^{-1}\, (\alpha\circ t)(a)\\
			& = a\, e^\alpha\\
			& = a \,\alpha^\ast (1_\mathcal{H})
		\end{align*}
		and 
		\[
		\alpha^\ast\circ t|_{A\,\,\alpha(1_\mathcal{H})}=(\alpha \circ t)^{-1} \circ \alpha \circ S \circ t|_{A\,\,\alpha(1_\mathcal{H})}= (\alpha \circ t)^{-1} \circ \alpha \circ s|_{A\,\,\alpha(1_\mathcal{H})}= (\alpha\circ t)^{-1}|_{A\,\,\alpha(1_\mathcal{H})}.
		\]
		This implies that $\alpha^\ast\circ t|_{A\,\,\alpha(1_\mathcal{H})}: A\,\,\alpha(1_\mathcal{H}) \rightarrow A\,\,e^\alpha$ is a bijection and  $e^{\alpha^\ast}=\alpha(1_\mathcal{H}).$
		
		Finally, we need to prove that every biretraction $\alpha:\mathcal{H}\rightarrow A$ satisfies $\alpha \ast \alpha^*  \ast \alpha= \alpha$ and $\alpha^\ast \ast \alpha \ast \alpha^\ast=\alpha^\ast.$ First, observe that
		
		\begin{eqnarray}\label{aaestrela}
			(\alpha \ast \alpha^*) (h) & = & \alpha^\ast \, t \, \alpha(h_{(1)})\, \alpha^\ast (h_{(2)}) \nonumber\\
			& = & (\alpha \circ t)^{-1} \, \alpha \circ (S \circ t) \, \alpha (h_{(1)})\, (\alpha \circ t)^{-1} \, \alpha \, S(h_{(2)})  ) \nonumber\\
			& = & (\alpha \circ t)^{-1} \left( (\alpha \circ s) ( \alpha (h_{(1)}))\, \alpha (S(h_{(2)}))  \right) \nonumber\\
			& = & (\alpha \circ t)^{-1} \left( \alpha (h_{(1)}) \,\alpha (S(h_{(2)}))  \right)\nonumber \\
			& = & (\alpha \circ t)^{-1} \, \alpha (h_{(1)} S(h_{(2)})) \nonumber  \\
			& = & (\alpha \circ t)^{-1} \, (\alpha \circ t) (\varepsilon (h)) \nonumber\\
			& = & \varepsilon (h) \, e^\alpha 
		\end{eqnarray}	
		and
		\begin{eqnarray}\label{aestrelaa}
			(\alpha^* \ast \alpha) (h) & = &  (\alpha \circ t) \, \alpha^* (h_{(1)}) \,\alpha  (h_{(2)})\nonumber \\
			& = & (\alpha \circ t) \, (\alpha \circ t)^{-1} \, \alpha ( S(h_{(1)})) \, \alpha (h_{(2)}) \nonumber\\
			& = & \alpha (S(h_{(1)}))\, \alpha (h_{(2)})\nonumber \\
			& = & \alpha (S(h_{(1)})\, h_{(2)})\nonumber \\
			& = & (\alpha \circ s)  \varepsilon (h)\nonumber \\
			& = & \varepsilon (h) \, \alpha (1_{\mathcal{H}}) 
		\end{eqnarray}
		for every $h\in\mathcal{H}.$ Then, for any $h\in \mathcal{H}$
		\begin{eqnarray}\label{uustaru}
			\alpha \ast \alpha^*  \ast \alpha (h) & = & (\alpha \circ t) \, (\alpha \ast \alpha^*)(h_{(1)})\, \alpha (h_{(2)}) \nonumber \\
			& = & (\alpha\circ t) (\varepsilon(h_{(1)})\, e^\alpha)\, \alpha(h_{(2)})\nonumber\\
			& = & \alpha (t(\varepsilon (h_{(1)}))\, h_{(2)}) \nonumber \\
			& = & \alpha (h) \nonumber
		\end{eqnarray}	
		and
		\begin{eqnarray}\label{ustaruustar}
			\alpha^* \ast \alpha \ast \alpha^* (h) & = & (\alpha \ast \alpha^*)\, t \, \alpha^* (h_{(1)}) (\alpha \ast \alpha^* )(h_{(2)}) \nonumber \\
			& = & \varepsilon (t(\alpha^* (h_{(1)}))) \varepsilon (h_{(2)}) e^\alpha \nonumber \\
			& = & \alpha^* (h_{(1)}) \varepsilon (h_{(2)}) \alpha^* (1_{\mathcal{H}}) \nonumber \\
			& = & \alpha^* (h_{(1)} s(\varepsilon (h_{(2)}))) \nonumber\\
			& = & \alpha^* (h). \nonumber
		\end{eqnarray}	
		
		Therefore, $\mathfrak{Brt} (\mathcal{H}, A)$ is a regular monoid.
	\end{proof}
	
	\begin{rmk}
		We cannot prove, in general, that $\mathfrak{Brt} (\mathcal{H}, A)$ is an inverse semigroup. That's because its idempotents do not always commute. Consider an idempotent $E\in \mathfrak{Brt}(\mathcal{H}, A)$ and denote its associated idempotent in $A$ by $e^E$ then, for any $a\in A$
		\begin{eqnarray*}
			E\circ t (a)  & = &  (E\ast E )(t(a)) =E\, t \, E (t(a)) E(1_{\mathcal{H}}) \\
			& = & E\, t \, E \, t (a).
		\end{eqnarray*}
		Then, $E\circ t :A \rightarrow A$ is a linear and multiplicative map in $A$ which is idempotent with respect to the composition. Moreover,
		\[
		E\circ t (e^E)=  E(1_{\mathcal{H}}) ,
		\]
		which leads to
		\[
		E\, t \, E \, t (e^E)  E\, t (e^E) =  E(1_{\mathcal{H}})E(1_{\mathcal{H}}) =E(1_{\mathcal{H}}).
		\]
		
		Applying $(E\circ t)^{-1}$  on both sides, we obtain
		\[
		E (t(e^E)) e^E =(E\circ t)^{-1} (E(1_{\mathcal{H}})),
		\]
		that is
		\[
		E(1_{\mathcal{H}}) e^E=e^E,
		\]
		which implies that $Ae^E \subseteq AE(1_{\mathcal{H}})$. But the map $(E\circ t):Ae^E \rightarrow AE(1_{\mathcal{H}})$ is bijective, then there exists an element $a\in A$ such that $e^E =E\circ t (ae^E)$, then
		\[
		E(1_{\mathcal{H}})=E\circ t (e^E) =E\circ t \circ E\circ t (ae^E)=E\circ t(ae^E)=e^E.
		\]
		Therefore, $E\circ t$ is also bijective restricted to the ideal $A. E(1_{\mathcal{H}})=A.E\circ t (1_A)$. 
		
		Denote by $\tilde{P}Aut(A)$ the set of all linear and multiplicative maps $\varphi: A\rightarrow A$ which are idempotent with relation to composition and bijective when restricted to the unital ideal $A \varphi (1_A)$. Let us call such maps as partial automorphisms of $A$. Observe that for each partial automorphism $\varphi\in \tilde{P}Aut(A)$  we have that $\varphi(a)=a\varphi(1_A)$ for every $a\in A,$ because
		$$\varphi \left( \varphi(a)-a\varphi(1_A) \right) =\varphi(\varphi(a))-\varphi(a) \varphi(\varphi(1_A))=\varphi(a)-\varphi(a)\varphi(1_A) = \varphi(a)-\varphi(a)=0.$$
		Then 
		\[
		\varphi(a)-a\varphi(1_A)=\varphi(a)\varphi (1_A)-a\varphi(1_A)=(\varphi(a)-a)\varphi(1_A)\in A\, \varphi(1_A)
		\]
		and $\varphi$ being bijective in $A\varphi(1_A)$ imply that $\varphi(a)-a\varphi(1_A)=0.$

		Lastly, observe that for two idempotents $E,F\in \mathfrak{Brt} (\mathcal{H}, A)$, we have 
		\begin{eqnarray*}
			E\ast F (h) & = & F\circ t (E(h_{(1)})) F(h_{(2)}) \\
			& = & E(h_{(1)}) F\circ t(1_A) F(h_{(2)}) \\
			& = & E(h_{(1)}) F(1_{\mathcal{H}}) F(h_{(2)}) \\
			& = & E(h_{(1)}) F(h_{(2)}) .
		\end{eqnarray*}	
		Similarly, $F\ast E (h) =F(h_{(1)}) E(h_{(2)})$. Therefore, unless for the case of the Hopf algebroid $\mathcal{H}$ being cocommutative, there is no a priori reason to suppose that the idempotents of $\mathfrak{Brt}(\mathcal{H},A)$ should commute in general.
	\end{rmk}
	
	\begin{rmk}\label{convinv}
		Let $\alpha,\beta\in \mathfrak{Brt} (\mathcal{H}, A).$ Then,
		$$\left(((\alpha\ast \beta)\circ t)|_{A\, e^{\alpha\ast \beta} }\right)^{-1}= \left((\alpha\circ t)|_{A\, e^\alpha}\right)^{-1} \circ \left((\beta\circ t)|_{A\, e^\beta}\right)^{-1},$$
		or simplifying the notation as before,
		$$((\alpha\ast\beta)\circ t)^{-1} = (\alpha\circ t)^{-1}\circ (\beta\circ t)^{-1}.$$
		In fact, we have
		
		\begin{align*}
			(\alpha\circ t)^{-1}\circ & \,(\beta\circ t)^{-1} \circ(\alpha\ast\beta)\circ t \left(a\, e^{\alpha\ast\beta}\right)= \\
			& = (\alpha\circ t)^{-1}\circ (\beta\circ t)^{-1} \circ \beta\circ t \circ \alpha\circ t  \left(a\, (\alpha\circ t)^{-1} \left( e^\beta\, \alpha (1_\mathcal{H})\right) \right)\\
			& = (\alpha\circ t)^{-1}\circ (\beta\circ t)^{-1} \circ \beta\circ t \left( (\alpha \circ t)(a) e^\beta \alpha (1_{\mathcal{H}}) \right)\\
			& = (\alpha\circ t)^{-1} \left( (\alpha\circ t)(a)\, e^\beta \, \alpha(1_\mathcal{H})  \right)\\
			& = (\alpha\circ t)^{-1} \circ (\alpha\circ t) \left( a \, (\alpha\circ t)^{-1}\left(e^\beta\, \alpha(1_\mathcal{H})\right)  \right)\\
			& =  a \, (\alpha\circ t)^{-1}\left(e^\beta\, \alpha(1_\mathcal{H})\right) \\
			& = a\, e^{\alpha\ast\beta}
		\end{align*}
		and
		\begin{align*}
			( (\alpha\ast\beta)\circ t)&\,  \circ (\alpha\circ t)^{-1}\circ \,(\beta\circ t)^{-1}  \left(a\, (\alpha\ast\beta) (1_\mathcal{H})\right) =\\
			& = \beta\circ t\circ \alpha\circ t\circ (\alpha\circ t)^{-1}\circ (\beta\circ t)^{-1}\left( a\,\beta\circ t\circ \alpha (1_\mathcal{H})\, \beta(1_\mathcal{H})\right)\\
			& = \beta\circ t\circ \alpha\circ t\circ (\alpha\circ t)^{-1} \left( (\beta\circ t)^{-1} (a\,\beta(1_\mathcal{H}))\, e^\beta\, \alpha(1_\mathcal{H}) \right) \\
			& = \beta\circ t \left( (\beta\circ t)^{-1} (a\,\beta(1_\mathcal{H}))\,  \alpha(1_\mathcal{H}) \, e^\beta\right)\\
			& = \beta\circ t\circ (\beta\circ t)^{-1} \left(a\, \beta\circ t\left(\alpha(1_\mathcal{H})\, e^\beta\right)\, \beta(1_\mathcal{H})\right)\\
			& = a\, \beta\circ t\left(\alpha(1_\mathcal{H})\, e^\beta\right)\\
			& = a\, \beta \circ t \circ \alpha (1_{\mathcal{H}}) \beta (1_{\mathcal{H}})\\
			& = a\, (\alpha\ast\beta) (1_\mathcal{H})
		\end{align*}
		for every $a\in A.$
	\end{rmk}

	\begin{rmk} Theorem \ref{theorem.Brt(H,A).is.a.regular.monoid} shows that $\mathcal{B}rt(\mathcal{H},A)$ is a regular monoid, and it was already noted that, from \emph{(BRT2)}, each biretraction  $\alpha: \mathcal{H}\rightarrow A,$ yields the isomorphism of ideals $\alpha \circ t \mid_{Ae^{\alpha}}: A e^{\alpha} \to A \alpha(1_{\mathcal{H}})$, which is an element of the inverse semigroup $\mathcal{I}(A)$ of the partial bijections between unital ideals of $A$. This is in fact a surjective antimorphism of semigroups.
		
		Let $End_{Alg}(A)$ denote the $\K$-algebra of multiplicative (not necessarily unital) $\K$-linear endomorphisms of $A$.
		Since $(\alpha\ast\beta)\circ t=\beta\circ t\circ\alpha\circ t$ for all $\alpha,\beta\in\mathcal{B}rt(\mathcal{H},A)$ (see Equation \eqref{equation.precomposition.convolution.with.t}), 
		there exists an antimultiplicative map $\xi : \mathcal{B}rt(\mathcal{H},A) \to End_{Alg}(A)$ given by $\alpha\mapsto \alpha\circ t$.

		Now consider the subset $\JJ(A) \subset End_{Alg}(A)$ of the multiplicative endomorphisms ${\varphi :A \to A}$ such that there exists an idempotent $e \in A$ satisfying
		\begin{enumerate}
			\item[(i)] $\varphi(A) = A\varphi(e)$;
			\item[(ii)] $\varphi|_{A\, e}: A\, e\rightarrow A\, \varphi(e)$ is an isomorphism of algebras. 
		\end{enumerate}
		By \emph{(BRT2)}, if $\alpha \in \mathcal{B}rt(\mathcal{H},A)$ then $\xi (\alpha) \in \JJ(A) $.
		Notice that it follows from (i) that ${\varphi (a) = \varphi(a) \varphi(e) = \varphi(ae)}$ for any $a \in A$ and any $\varphi \in \JJ(A)$.
		
		$\JJ(A)$ is a multiplicative subsemigroup of $End_{Alg}(A)$. Let $\varphi, \psi \in \JJ(A)$ and $a \in A$, with corresponding idempotents $e$ and $f$; we claim that $\psi \circ \varphi (A) = A .\psi (\varphi(e)) \psi(f)$. In fact, 
		\[
		\psi ( \varphi (a)) = \psi (\varphi (ae)) = \psi(\varphi (ae) f) = \psi \circ \varphi (a) \psi (\varphi(e)) \psi(f),
		\]
		so that $\psi \circ \varphi (A) \subset A .\psi (\varphi(e)) \psi(f)$. Conversely, given $b \in A$, it follows from (ii) that there exist $b', b'' \in A$ such that $b \psi (f) = \psi (b' f)$ and $b' f  \varphi(e) = \varphi (b'' e)$, hence 
		\[
		b \ (\psi (\varphi(e)) \psi(f) = \psi (b' f) \psi (\varphi(e)) = \psi (b' f  \varphi(e)) = \psi (\varphi (b'' e)),
		\]
		therefore $\psi \circ \varphi (A) = A .\psi (\varphi(e)) \psi(f)$.
		
		Let $g = (\varphi\mid_{Ae})^{-1} (\varphi(e)f)$. Then $\psi \circ \varphi \mid_{Ag} : Ag \to  A.\psi (\varphi(e)) \psi(f) $ is an isomorphism of algebras. This map is injective, since it can be written as a composition of injetive maps; and it is surjective, given that 
		\[
		\psi(\varphi(ag)) = \psi (\varphi(a)) \psi (\varphi(g)) = \psi (\varphi(a)) \psi ( \varphi(e)f) = \psi( \varphi (ae)f) =\psi(\varphi (a)f)= \psi (\varphi (a))
		\]
		and that $\psi (\varphi (A)) = A .\psi (\varphi(e)) \psi(f)$. Therefore $\JJ(A) $ is a multiplicative subsemigroup of $\End_{Alg}(A)$.
		
		The map $\xi :\mathcal{B}rt(\mathcal{H},A) \to  \JJ(A) $ has a section and therefore is surjective.  Indeed, for each $\varphi \in\JJ(A)$ with associated idempotent $e$ we define a biretraction $\alpha_{\varphi}:\mathcal{H}\rightarrow A$ by 
		$$\alpha_{\varphi}(h)=\varepsilon(h_{(2)})\,\varphi \left(\varepsilon(h_{(1)})\, e\right).$$
		Then
		$$\alpha_{\varphi}\circ t(a)=\varepsilon(1_\mathcal{H}) \,\varphi\left(\varepsilon\circ t(a)\, e\right)=\varphi(ae) = \varphi(a),$$
		which implies that $\xi(\alpha_{\varphi})= \varphi.$ Also, $\alpha_{\varphi}$ is a biretraction because
		$$\alpha_{\varphi}\circ s(a)=\varepsilon\circ s(a)\, \varphi(\varepsilon(1_\mathcal{H})e)= a\,\varphi(e)$$
		and
		$$\alpha_{\varphi}\circ t(a)= \varphi(ae),$$
		leading to $\alpha_{\varphi}\circ t \mid_{A\, e}=\varphi\mid_{A\, e} : Ae\to A \varphi(e)$.
		
		Finally, let $\mathcal{I} (A)$ be the inverse semigroup of isomorphisms between unital ideals of $A$. We have a morphism of semigroups 
		$$ F : \JJ(A) \to \mathcal{I} (A), \ \ \varphi \mapsto \varphi\mid_{Ae} : Ae \to A \varphi(e),$$
		where $e$ is, as before, the idempotent associated to $\varphi$.
		This morphism is actually an isomorphism: its inverse takes an isomorphism of ideals $\theta : Ae \to A\theta(e)$ to the map $\varphi_\theta : A \to A$ given by $\varphi_\theta (a) = \theta(ae)$. Therefore, we obtain a surjective antimorphism of semigroups $$F \circ \xi : \mathcal{B}rt(\mathcal{H},A) \to \mathcal{I}(A)$$ defined by 
		$$(F \circ \xi) (\alpha) = \alpha \circ t \mid_{Ae^{\alpha}}: A e^{\alpha} \to A \alpha(1_{\mathcal{H}}).$$ 
		
	\end{rmk}
	\color{black}
	
	\vspace{15pt}
	Consider now the free vector space generated by the biretractions of $\mathcal{H}$ and extend linearly the convolution product to this space. Then, we have an algebra structure on the space $\K \mathfrak{Brt} (\mathcal{H}, A)$, henceforth denoted by $\mathfrak{B} (\mathcal{H})$. 
	
	\begin{thm}
		Let $\mathcal{H}$ be a commutative Hopf algebroid over a commutative algebra $A$. Then the algebra $\mathfrak{B}(\mathcal{H})$, generated by the set of biretractions of $\mathcal{H}$ with the convolution product is a unital quantum inverse semigroup.
	\end{thm}
	
	\begin{proof}
		As we have already proven in last theorem, $\mathfrak{Brt} (\mathcal{H}, A)$ is a regular monoid, then the algebra $\mathfrak{B} (\mathcal{H})$ is a unital algebra.
		
		Define first the comultiplication $\underline{\Delta} :\mathfrak{B} (\mathcal{H}) \rightarrow \mathfrak{B} (\mathcal{H}) \otimes \mathfrak{B} (\mathcal{H})$  on the basis elements $\alpha \in \mathfrak{Brt} (\mathcal{H}, A)$ as $\underline{\Delta} (\alpha ) =\alpha \otimes \alpha$,  
		then extend linearly for the whole algebra $\mathfrak{B} (\mathcal{H})$. 
		Then, it is obvious that the comultiplication $\Delta$ is multiplicative, since for any $\alpha , \beta \in \mathfrak{Brt} (\mathcal{H}, A)$ the convolution product $\alpha \ast \beta$ also belongs to $\mathfrak{Brt} (\mathcal{H}, A)$ 
		
		Define also the pseudo antipode on the basis elements as $\mathcal{S}(\alpha )=\alpha^* = (\alpha \circ t)^{-1} \circ \alpha \circ S$ then extend linearly to $\mathfrak{B}(\mathcal{H})$. Again, to prove that $\mathcal{S}$ is antimultiplicative, it is enough to check on biretractions.
		
		Then, for $h \in \mathcal{H}$ and  $\alpha , \beta \in \mathfrak{Brt} (\mathcal{H}, A)$, we have
		\begin{eqnarray*}
			\mathcal{S}(\alpha \ast \beta)(h) & = & ((\alpha \ast \beta)\circ t)^{-1}\circ (\alpha \ast \beta )\circ S(h) \\
			& = & ((\alpha \ast \beta)\circ t)^{-1} \left( \beta \circ t \circ \alpha \circ S(h_{(2)}) \,\beta \circ S(h_{(1)}) \right) \\
			& \stackrel{(*)}{=} & (\alpha \circ t)^{-1} \circ (\beta \circ t)^{-1} \left( \beta \circ t \circ \alpha \circ S(h_{(2)})\, \beta \circ S(h_{(1)}) \,\beta\circ t \left( e^\beta\right)\right)\\
			& = & (\alpha \circ t)^{-1} \left( \alpha \circ S(h_{(2)})\, e^\beta \,(\beta \circ t)^{-1} \circ \beta \circ S(h_{(1)}) \right)\\
			& = & (\alpha \circ t)^{-1} \circ \alpha \circ S(h_{(2)})\, (\alpha \circ t)^{-1} \circ (\beta \circ t)^{-1} \circ \beta \circ S(h_{(1)})\\
			& = & \alpha^\ast (h_{(2)})\,(\alpha\circ t)^{-1}\circ \beta^\ast (h_{(1)}),
		\end{eqnarray*}	
		in which, we used the result $((\alpha \ast \beta)\circ t)^{-1} =(\alpha \circ t)^{-1} \circ (\beta \circ t)^{-1}$ from the Remark \ref{convinv}.
		
		On the other hand,
		\begin{eqnarray*}
			\left(\mathcal{S}(\beta) \ast \mathcal{S}(\alpha)\right) (h) & = & \alpha^\ast\circ t \circ \beta^\ast (h_{(1)}) \,\alpha^\ast (h_{(2)}) \\
			& = & (\alpha \circ t)^{-1} \circ \alpha \circ S\circ t \circ \beta^\ast (h_{(1)})  \, \alpha^\ast (h_{(2)}) \\
			& = & (\alpha \circ t)^{-1} \circ \alpha \circ s \circ \beta^\ast (h_{(1)})  \,\alpha^\ast (h_{(2)})  \\
			& = & (\alpha \circ t)^{-1} \circ \beta^\ast (h_{(1)}) \,\alpha^\ast (h_{(2)}).
		\end{eqnarray*}	
		
		Consequently, $\mathcal{S}(\alpha \ast \beta)= \mathcal{S}(\beta) \ast \mathcal{S}(\alpha)$ and $\mathcal{S}$ is antimultiplicative.
		
		The equations (\ref{uustaru}) and (\ref{ustaruustar}) imply item (ii) of Axiom (QISG3) of Quantum Inverse Semigroups.
		
		Finally, for checking axiom (QISG4), we use the equations (\ref{aaestrela}) and (\ref{aestrelaa}). Then for $\alpha ,\beta \in \mathfrak{Brt}(\mathcal{H},A)$ and $h\in \mathcal{H}$,
		\begin{eqnarray*}
			\alpha_{(1)} \ast \mathcal{S}(\alpha_{(2)}) \ast \mathcal{S}(\beta_{(1)})\ast \beta_{(2)}(h) & = & (\alpha \ast \alpha^*)\ast (\beta^* \ast \beta)(h) \\
			& = & (\beta^* \ast \beta) \circ t \circ (\alpha \ast \alpha^*) (h_{(1)}) \,(\beta^* \ast \beta)(h_{(2)}) \\
			& = & (\beta^* \ast \beta) \circ t \, (\varepsilon(h_{(1)})\, e^\alpha) \,(\beta^* \ast \beta)(h_{(2)}) \\
			& = &  \varepsilon \circ t \left(\varepsilon (h_{(1)})\, e^\alpha\right) \varepsilon (h_{(2)}) \,\beta(1_{\mathcal{H}}) \\
			& = & \varepsilon (h_{(1)}) \,e^\alpha\,\varepsilon (h_{(2)}) \,\beta(1_{\mathcal{H}}) \\
			& = & \varepsilon (h)\, e^\alpha \,\beta(1_{\mathcal{H}}).			
		\end{eqnarray*}	
		The same result for $\mathcal{S}(\beta_{(1)})\ast \beta_{(2)} \ast \alpha_{(1)} \ast \mathcal{S}(\alpha_{(2)})(h)$.
		\vspace{10pt}
		
		Also, the pseudo antipode is unital. Indeed, for $h\in \mathcal{H}$,
		\begin{eqnarray*}
			\mathcal{S}(\varepsilon)(h) & = & (\varepsilon \circ t)^{-1} \circ \varepsilon \circ S(h)  \\
			&	=& \varepsilon \circ S(h)\\
			& = & \varepsilon \circ S\left(t(\varepsilon (h_{(1)}))\, h_{(2)}\right)\\
			& = & \varepsilon \left(S(h_{(2)})\, s(\varepsilon(h_{(1)}))\right) \\
			& = & \varepsilon \left(h_{(1)}S(h_{(2)})\right) \\
			& = & \varepsilon \circ t\circ \varepsilon (h) \\
			& = & \varepsilon (h) .
		\end{eqnarray*}
		
		Therefore, the algebra $\mathfrak{B}(\mathcal{H})$ is a unital quantum inverse semigroup.	  
	\end{proof}
	
	\begin{exmp}
		Let $H$ be a commutative Hopf algebra, considered as a Hopf algebroid over the field $\K$. The set of biretractions of $H$, which are global, coincides with the group of algebra morphisms between $H$ and $\K$, that is, the group $G(H^{\circ})$ of group-like elements of the finite dual Hopf algebra $H^{\circ}$.
	\end{exmp}
	
	\begin{exmp} \label{AA2}
		Let $A$ be a commutative Hopf algebra and consider the Hopf algebroid $\mathcal{H}=A\otimes A$, from \emph{Example \ref{AA}}. Let $M(A)$ be the set of multiplicative functions $\varphi: A\rightarrow A$, $E(A)\subseteq A$ the set of idempotent elements of $A$  and
		
		$$M(A)\times^b E(A)=\{ (\varphi, e) \in M(A)\times E(A) \textrm{ such that } \varphi|_{A e} : A e\longrightarrow A \varphi(e) \textrm{ is a bijection}\}.$$
		
		Consider the equivalence relation
		$$(\varphi, e) \sim (\psi, f) \,\,\Leftrightarrow \,\,e=f \textrm{ and } \varphi|_{Ae}=\psi|_{Ae}.$$
		
		Representing the class of an element $(\varphi, e)\in M(A)\times^b E(A)$ by $[\varphi,e],$ then the biretractions of $\mathcal{H}$ are classified by the set
		$$M(A)\ltimes E(A):=\left\{[\varphi,e] : \, (\varphi,e)\in M(A)\times^b E(A)\right\},$$
		which is a regular monoid with the multiplication $$[\varphi,e]\, [\psi,f]=[ \varphi\circ\psi, \psi^{-1}(e\psi(f))],$$
		unity $[\emph{\textrm{Id}}, 1_A]$ and $[\varphi,e]^\ast= [\varphi^{-1}, \varphi(e)]$, in which we are denoting by $\psi^{-1}$ the map $\left( \psi|_{Af} \right)^{-1}$ and by $\varphi^{-1}$ the map $\left( \varphi|_{Ae} \right)^{-1}$ 
		
		This multiplication is well-defined, because considering the function
		$$\varphi\circ\psi|_{A\psi^{-1}(e\psi(f))}: A\psi^{-1}(e\psi(f)) \longrightarrow A\varphi(e \psi(f)),$$
		we have that 
		\begin{itemize}
			\item $\varphi\circ\psi|_{A\psi^{-1}(e\psi(f))}$ is injective:
			
			\begin{align*}
				& 0 = \varphi\circ\psi (a\psi^{-1}(e\psi(f))) \\
				& 0= \varphi (\psi (a) e \psi (f)) \\
				\Rightarrow\,\,\, & 0 = \psi (a)\psi (f)e = \psi (a\psi^{-1}(e\psi(f)) = \psi (af\psi^{-1}(e\psi(f)) \\
				\Rightarrow\,\,\, & 0 = af\psi^{-1}(e\psi(f))= a\psi^{-1}(e\psi(f)).
			\end{align*}
			
			\item $\varphi\circ\psi|_{A\psi^{-1}(e\psi(f))}$ is surjective: given $a\,\varphi(e\psi(f))\in A\varphi(e \psi(f)),$ we have
			\vspace{7pt}
			\begin{align*}
				a\varphi(e\psi(f)) & = \varphi(\varphi^{-1} (a\varphi(e))\psi(f) )\\
				& = \varphi\circ\psi \circ \psi^{-1}(\varphi^{-1}(a\varphi(e))\psi(f)\\
				& = \varphi\circ\psi (\psi^{-1}(\varphi^{-1}(a\varphi(e)))\psi^{-1} (e\psi(f))) \in \varphi\circ\psi(A\psi^{-1}(e\psi(f))).
			\end{align*}
		\end{itemize}
		
		Given the element $[\varphi,e]\in M(A)\ltimes E(A),$ define
		\[\begin{array}{rccl} \alpha_{[\varphi,e]} : & A\otimes A & \rightarrow & A \\  \, & a\otimes b & \mapsto & \varphi (ae) \, b ,\end{array}
		\]
		which is a biretraction with $e^{\alpha_{[\varphi,e]}}=e$ and $\alpha_{[\varphi,e]}\circ t|_{Ae}=\varphi|_{Ae}.$
		
		The convolution product between two local biretractions $\alpha_{[\varphi,e]} , \alpha_{[\psi,f]}  \in \mathfrak{Brt}(A\otimes A ,A)$ is given by
		\begin{eqnarray*}
			\alpha_{[\varphi,e]}  \ast \alpha_{[\psi,f]}  (a\otimes b) & = & \alpha_{[\psi,f]}  \circ t \circ  \alpha_{[\varphi,e]}  (a\otimes 1_A )       \,\alpha_{[\psi,f]}  (1_A \otimes b) \\
			& = & \alpha_{[\psi,f]}  (\varphi (ae)\otimes 1_A )\, \psi (f )\, b \\
			& = & \psi (\varphi (ae)f) \,\psi (f )\, b \\
			& = & \psi (\varphi (ae)f)\,  b \\
			& = & \psi\circ \varphi (a\varphi^{-1}(f\varphi(e)))\, b\\
			& = & \alpha_{[\psi\circ\varphi, \varphi^{-1}(f\varphi (e))]}(a\otimes b)\\
			& = & \alpha_{[\psi,f]\, [\varphi,e]}(a\otimes b) 
		\end{eqnarray*}	
		for every $a,b\in A.$ Then there is an isomorphism of semigroups
		\[
		\begin{array}{rccl} \alpha : & (M(A)\ltimes E(A))^{op} & \rightarrow & \mathfrak{Brt}(A\otimes A ,A) \\
			& [\varphi,e]        & \mapsto     & \alpha_{[\varphi,e]} \end{array}  ,
		\]
		whose inverse is 
		\[
		\begin{array}{rccl} \varphi : & \mathfrak{Brt}(A\otimes A ,A) & \rightarrow &  (M(A)\ltimes E(A))^{op}  \\
			\,      & \alpha        & \mapsto     & [\alpha\circ t, e^\alpha] \end{array}  .
		\]
		
		Moreover, this is an isomorphism of regular monoids:
		\begin{itemize}
			\item $\alpha$ maps unity to unity: for every $a,b\in A,$
			$$   \alpha_{[Id_A,1_A]}(a\otimes b)  = id(a)b= ab=\varepsilon (a\otimes b).$$
			
			\item $\alpha$ maps pseudo-inverse to pseudo-inverse: for every $a,b\in A,$
			\begin{align*}
				\alpha^\ast_{[\varphi,e]}(a\otimes b) &= (\alpha_{[\varphi,e]}\circ t)^{-1}\circ \alpha_{[\varphi,e]} \circ S(a\otimes b)\\
				& = (\alpha_{[\varphi,e]}\circ t)^{-1}\circ \alpha_{[\varphi,e]} (b\otimes a)\\
				&= \varphi^{-1}(\varphi(be)a)\\
				& = b\,e\,\varphi^{-1}(a\varphi(e))\\
				& = \alpha_{[\varphi^{-1},\varphi(e)]}(a\otimes b)\\
				& = \alpha_{[\varphi,e]^\ast}(a\otimes b).
			\end{align*}
		\end{itemize}
	\end{exmp}

	\begin{exmp}
		Generalizing slightly the previous example, we can create local biretractions for the Hopf algebroid $\mathcal{H} =(A\otimes  A)[x,x^{-1}]$ from \emph{Example \ref{AALaurent}}, with $A$ commutative. Consider the set
		\begin{align*}
			(M(A)\times^b E(A))\times' A  =  \{(\varphi,e,p)\in M(A)\times E(A)\times A \,&\, \textrm{ such that } \varphi|_{Ae}:\, Ae\rightarrow A\varphi(e) \textrm{ is a } \\
			& \textrm{bijection and } \exists p'\in A :\, pp'=\varphi(e)\}.
		\end{align*}
		
		Observe that if $p',p''\in A$ both satisfy $pp'=\varphi(e)=pp'',$ then
		\begin{equation}\label{p''}
			p'\varphi(e)=p'pp''=p''pp'=p''\varphi(e).    
		\end{equation}

		Now, considering the equivalence relation
		$$(\varphi,e,p)\sim (\psi,f,q)\,\, \Leftrightarrow \,\, e=f,\,  \varphi|_{Ae}=\psi|_{Ae}\textrm{ and } p\varphi(e)= q\varphi(e) $$
		and representing by $[\varphi,e,p]$ the class of equivalent elements by this relation, we have that
		$$(M(A)\ltimes E(A))\ltimes A :=\left\{ [\varphi, e, p]:\, (\varphi,e,p) \in M(A)\times^b E(A)\times' A \right\}$$
		is a semigroup with the product
		$$[\varphi,e,p]\, [\psi,f,q] = [\varphi\circ\psi,\, \psi^{-1}( e\psi(f)),\,p\varphi(q)],$$
		unity $ [\emph{\textrm{Id}}_A,1_A,1_A]$ and $ [\varphi,e,p]^\ast= [\varphi^{-1},\varphi (e), \varphi^{-1}(p'\varphi(e))].$ The product is well-defined, because for a class $[\varphi,e,p] $ in $(M(A)\ltimes E(A))\ltimes A,$ we can take $(p\varphi(q))'=p'\varphi(q'):$
		\begin{align*}
			p\varphi(q)\, p'\varphi(q') & = pp'\varphi(qq')\\
			& = \varphi(e)\varphi(\psi(f))\\
			& = \varphi\circ\psi(\psi^{-1}(e\psi(f))).
		\end{align*}
		
		Then, given $[\varphi,e,p]\in (M(A)\ltimes E(A))\ltimes A,$ we can define for $n\in\mathbb{N},$
		\[
		\begin{array}{rccl} \alpha_{[\varphi,e,p]} : & (A\otimes A)[x,x^{-1}] & \rightarrow & A \\  \, & (a\otimes b)x^n  & \mapsto & \varphi (ae) b p^n\\
			& \,\,\,(a\otimes b)x^{-n} & \mapsto & \varphi(ae)b(p')^n,\end{array}
		\]
		which is also well-defined because of \emph{(\ref{p''})}.
		
		This map is a biretraction in $\mathcal{H}$ just like in the previous example and the convolution product between two local biretractions $\alpha_{[\varphi,e,p]} , \alpha_{[\psi,f,q]}  \in \mathfrak{Brt}((A\otimes A)[x,x^{-1}] ,A)$ is given by
		\begin{align*}
			\alpha_{[\varphi,e,p]}\ast \alpha_{[\psi,f,q]}((a\otimes b)x^n) & = \alpha_{[\psi,f,q]}\circ t \circ \alpha_{[\varphi,e,p]}((a\otimes 1_A)x^n) \,\alpha_{[\psi,f,q]}((1_A\otimes b)x^n)\\
			& = \psi(\varphi(ae)p^nf)\psi(f)bq^n\\
			& = \psi\circ\varphi (a\varphi^{-1}(f(\varphi(e)))) b\psi(p^n)q^n\\
			& = \alpha_{[\psi\circ\varphi,\varphi^{-1}(f\varphi(e)),q\psi(p)]} ((a\otimes b)x^n)\\
			& = \alpha_{[\psi,f,q][\varphi,e,p]}((a\otimes b)x^n)
		\end{align*}
		for every $(a\otimes b)x^n\in (A\otimes A)[x,x^{-1}]. $ Analogously, we have that
		$$ \alpha_{[\varphi,e,p]}\ast \alpha_{[\psi,f,q]}((a\otimes b)x^{-n})=\alpha_{[\psi,f,q][\varphi,e,p]}((a\otimes b)x^{-n}).$$
		
		Therefore, the map
		\[
		\begin{array}{rccl} \alpha : & ((M(A)\ltimes E(A))\ltimes A)^{op} & \rightarrow & \mathfrak{Brt}((A\otimes A)[x,x^{-1}] ,A) \\
			\,      & [\varphi,e,p]        & \mapsto     & \alpha_{[\varphi,e,p]} \end{array}  
		\]
		is an isomorphism of semigroups, whose inverse is given by
		\[
		\begin{array}{rccl} \varphi : & \mathfrak{Brt}((A\otimes A)[x,x^{-1}] ,A) & \rightarrow &  ((M(A)\ltimes E(A))\ltimes A)^{op}  \\
			\,      & \alpha        & \mapsto     & [\alpha\circ t,  e^\alpha,\alpha(x)] \end{array}  .
		\]
		
		In fact, for every $[\varphi,e,p]\in (M(A)\ltimes E(A))\ltimes A$ and $(a\otimes b)x^n\in (A\otimes A)[x,x^{-1}],$
		\begin{align*}
			\alpha_{[\alpha\circ t, e^\alpha,\alpha(x)]} ((a\otimes b)x^n) & = \alpha\circ t\left(a\,e^\alpha\right)b\alpha(x^n)\\
			& = \alpha\circ t(a)\alpha\circ s(b)\alpha(x^n)\\
			& = \alpha((a\otimes b)x^n)
		\end{align*}
		\vspace{-7pt}
		and
		\vspace{-7pt}
		\begin{align*}
			[\alpha_{[\varphi,e,p]}\circ t, e^{\alpha_{[\varphi,e,p]}},\alpha_{[\varphi,e, p]}(x)] & = [\varphi, e, \varphi(e)p]\\
			& = [\varphi, e, p].
		\end{align*}
		
		Moreover, this is an isomorphism of regular monoids, because $\alpha$ maps unity to unity
		$$   \alpha_{[\emph{\textrm{Id}}_A,1_A,1_A]}((a\otimes b)x^n)  = \emph{\textrm{Id}}_A (a)\,b\,(1_A)^n= ab=\varepsilon ((a\otimes b)x^n) .$$
		and also maps pseudo-inverse to pseudo-inverse:
		\begin{align*}
			\alpha^\ast_{[\varphi,e,p]}((a\otimes b)x^n) &= (\alpha_{[\varphi,e,p]}\circ t)^{-1}\circ \alpha_{[\varphi,e,p]} \circ S((a\otimes b)x^n)\\
			& = (\alpha_{[\varphi,e,p]}\circ t)^{-1}\circ \alpha_{[\varphi,e,p]} ((b\otimes a)x^{-n})\\
			&= \varphi^{-1}(\varphi(be)a(p')^n)\\
			& = be\,\varphi^{-1}(a(p')^n\varphi(e))\\
			& = \varphi^{-1}(a\,\varphi(e)) b \left( \varphi^{-1}(p'\varphi(e))\right)^n\\
			& = \alpha_{[\varphi^{-1},\varphi(e),\varphi^{-1}(p'\varphi(e))]}((a\otimes b)x^n)\\
			& = \alpha_{[\varphi,e,p]^\ast}((a\otimes b)x^n).
		\end{align*}

	\end{exmp}

	\subsection{Biretractions on the Hopf algebroid of representative functions of a discrete groupoid}

	\begin{prop}\label{map_bis_birret}
		Let $\G$ a groupoid and $\mathcal{H}$ the Hopf algebroid of representative functions of $\G$ from the section \emph{\ref{RepFun}}. The map $\alpha: \mathfrak{B}(\G)\rightarrow\mathfrak{Brt}(\mathcal{H}, A),$ $(u,X)\mapsto \alpha_{(u,X)}$ given by 
		\begin{equation}\label{alphamap}
			\alpha_{(u,X)}(\overline{\varphi \otimes_{T_\mathcal{E}}p})_x= \varphi(t\circ u(x))\left(\rho^{\mathcal{E}}_{u(x)}(p(x))\right) \llbracket x\in X\rrbracket ,
		\end{equation}
		for every $\overline{\varphi \otimes_{T_\mathcal{E}}p}\in \mathcal{H}$ and $x\in \mathcal{G}^{(0)},$ is a well-defined morphism of regular monoids. \end{prop}
	
	Before starting the proof of the proposition above, let us make two important remarks:
	\begin{itemize}
		\item Here, the brackets $\llbracket \underline{\quad} \rrbracket$, appearing in the expression (\ref{alphamap}), denote the Boolean value function, which is equal to $0$ if the sentence inside the brackets is false and it is equal to $1$ if the sentence is true. 
		\item Note here that the definition of the map $\alpha$ was inspired by the definition of the map $\zeta$ from \emph{Remark \ref{zeta}}.
	\end{itemize}
	
	\begin{proof}
		First, $\alpha$ can be written as
		$$\alpha_{(u,X)}(\overline{\varphi \otimes_{T_\mathcal{E}}p})_x= \zeta(\overline{\varphi \otimes_{T_\mathcal{E}}p})(u(x))\llbracket x\in X\rrbracket,$$
		hence each $\alpha_{(u,X)}$ is well-defined and multiplicative. Also, (BRT1) is valid, because
		$$(\alpha_{(u,X)}\circ \overline{s}(a))_x = \alpha_{(u,X)}(\overline{1_A\otimes_{T_\mathcal{I}}a})_x = a(x)\llbracket x\in X\rrbracket = a(x)\alpha_{(u,X)}(1_\mathcal{H})_x$$
		for every $x\in \G^{(0)}$ and $a\in A.$ To prove that $\alpha_{(u,X)}$ satisfies (BRT2)  for every bisection $(u,X)$ of $\G,$ remember that
		$(u,X)^\ast =(\overline{u},t\circ u(X)),$ with $\overline{u}(t\circ u(x))=u(x)^{-1}.$ Then,
		\begin{align*}
			\alpha_{(u,X)}\circ \overline{t}\left(\alpha_{(u,X)^\ast} (1_\mathcal{H})\right)_x & = \alpha_{(u,X)}\left(\overline{\alpha_{ (u,X)^\ast}(1_\mathcal{H})\otimes_{T_\mathcal{I}}1_A }\right)_x \\
			& = \alpha_{(u,X)^\ast} (1_\mathcal{H})_{t\circ u(x)}\llbracket x\in X\rrbracket\\
			& = \llbracket t\circ u(x)\in t\circ u(X)\rrbracket \llbracket x\in X\rrbracket\\
			& = \llbracket x\in X\rrbracket\\
			& = \alpha_{(u,X)}(1_\mathcal{H})_x
		\end{align*}
		for every $x\in\G^{(0)}.$  Moreover, the map
		$$\alpha_{(u,X)}\circ \overline{t}|_{A\, \alpha_{(u,X)^\ast} (1_\mathcal{H}) } : A\, \alpha_{(u,X)^\ast} (1_\mathcal{H}) \longrightarrow A\, \alpha_{(u,X)} (1_\mathcal{H}) $$
		is injective. In fact, for any $a\in A$ such that 
		$$ \alpha_{(u,X)}\circ \overline{t} \left(a\,  \alpha_{(u,X)^\ast}(1_\mathcal{H})\right) = 0,$$
		then 
		\begin{align}\label{injection}
			a(t\circ u(y))\llbracket y\in X\rrbracket & = a(t\circ u(y))\llbracket t\circ u(y)\in t\circ u(X)\rrbracket\llbracket y\in X\rrbracket \nonumber\\
			& = (a\,  \alpha_{(u,X)^\ast}(1_\mathcal{H}))_{t\circ u(y)}\llbracket y\in X\rrbracket \nonumber\\
			&=0
		\end{align}
		for every $y\in \G^{(0)}.$ Thus
		\begin{align*}
			\left( a\,  \alpha_{(u,X)^\ast}(1_\mathcal{H})\right)_x & = a(x)\llbracket x\in t\circ u(X)\rrbracket \\
			& = a(t\circ u((t\circ u)^{-1}(x))) \llbracket (t\circ u)^{-1}(x)\in X\rrbracket \llbracket x\in t\circ u(X)\rrbracket \\
			& \stackrel{(\ast)}{=} 0
		\end{align*}
		for every $x\in \mathcal{G}^{(0)},$ where we used the equation (\ref{injection}) in $(\ast)$ with $y=(t\circ u)^{-1}(x).$ Therefore $\alpha_{(u,X)}\circ \overline{t}|_{A\, \alpha_{(u,X)^\ast} (1_\mathcal{H}) }$ is injective.
		
		Now observe that
		\begin{align*}
			\alpha_{(u,X)}\,\circ &\,S (\overline{\varphi \otimes_{T_\mathcal{E}}p})_x \\
			& = \zeta\circ S (\overline{\varphi \otimes_{T_\mathcal{E}}p})(u(x))\,\llbracket x\in X\rrbracket \\
			& = \zeta(\overline{\varphi \otimes_{T_\mathcal{E}}p})((u(x))^{-1})\,\llbracket x\in X\rrbracket\\
			& = \varphi(t((u(x))^{-1}))\left( \rho^\mathcal{E}_{(u(x))^{-1}} \left(p(s( u(x))^{-1}\right)\right) \,\llbracket x\in X\rrbracket \\
			& = \varphi(t\circ\overline{u}(t\circ u(x)))\left( \rho^\mathcal{E}_{\overline{u}(t\circ u(x))} \left(p(t\circ u(x)\right)\right) \,\llbracket t\circ u(x)\in t\circ u(X)\rrbracket\\
			& = \alpha_{(u,X)^\ast}\left( \overline{\varphi \otimes_{T_\mathcal{E}}p}\right)_{t\circ u(x)}.
		\end{align*}
		Thus $\alpha_{(u,X)}\circ \overline{t}|_{A\, \alpha_{(u,X)^\ast} (1_\mathcal{H}) } : A\, \alpha_{(u,X)^\ast} (1_\mathcal{H}) \longrightarrow A\, \alpha_{(u,X)} (1_\mathcal{H})$ is surjective, because for every $a\in A$ and every $x\in \mathcal{G}^{(0)},$
		\begin{align*}
			\left(a\, \alpha_{(u,X)}(1_\mathcal{H})\right)_x & = \left(\alpha_{(u,X)}\circ \overline{s}(a)\right)_x\\
			& = \left(\alpha_{(u,X)}\circ S\circ \overline{t}(a)\right)_x\\
			& = \alpha_{(u,X)^\ast}(\overline{t}(a))_{t\circ u(x)} \\
			& = \alpha_{(u,X)}\circ \overline{t} \left(\alpha_{(u,X)^\ast} (\overline{t}(a))\right)_x.
		\end{align*}
		
		Therefore, $e^{\alpha_{(u,X)}}=\alpha_{(u,X)^\ast}(1_\mathcal{H})$ satisfies (BRT2) and $\alpha_{(u,X)}$ is a local biretraction.

		Now, for $(u,X)$ and $(v,Y)$ local bisections of $\G$ with $(u,X)\cdot (v,Y)=(uv,Z),$ 
		
		\begin{small}
			\begin{align*}
				&  \alpha_{(u,X)} \ast  \alpha_{(v,Y)}(\overline{\varphi\otimes_{T_\mathcal{E}}p})_x \\
				& =\sum_{i=1}^n \alpha_{(v,Y)}\circ \overline{t}\circ\alpha_{(u,X)} (\overline{\varphi\otimes_{T_\mathcal{E}}e_i})_x\,\alpha_{(v,Y)}(\overline{e_i^*\otimes_{T_\mathcal{E}}p})_x\\
				& =\sum_{i=1}^n \alpha_{(v,Y)}(\overline{\alpha_{(u,X)}(\overline{\varphi\otimes_{T_\mathcal{E}}e_i})\otimes_{T_\mathcal{I}}1_A})_x \,\alpha_{(v,Y)}(\overline{e_i^*\otimes_{T_\mathcal{E}}p})_x\\
				& =\sum_{i=1}^n \alpha_{(u,X)}(\overline{\varphi\otimes_{T_\mathcal{E}}e_i})_{t\circ v(x)} \left(\rho^\mathcal{I}_{v(x)}(1_A(x))\right) \llbracket x\in Y\rrbracket \,\alpha_{(v,Y)}(\overline{e_i^*\otimes_{T_\mathcal{E}}p})_x\\
				& =\sum_{i=1}^n \varphi(t\circ u\circ t\circ v(x)) \left(\rho^{\mathcal{E}}_{u\circ t\circ v(x)} (e_i(t\circ v(x)))\right)  e_i^*(t\circ v(x)) \left( \rho^\mathcal{E}_{v(x)}(p(x))\right) \llbracket x\in Y\rrbracket\llbracket t\circ v(x)\in X\rrbracket\\
				& =\sum_{i=1}^n \varphi(t\circ uv(x)) \left( e^*_i\left(\rho^{\varepsilon}_u\right) e_i\right) (t\circ v(x))\left( \rho^\mathcal{E}_{v(x)}(p(x))\right) \llbracket x\in Z\rrbracket\\
				& = \varphi(t\circ uv(x)) \rho^{\mathcal{E}}_{u\circ t \circ v(x)}\rho^\mathcal{E}_{v(x)}(p(x))\llbracket x\in Z\rrbracket \\
				& = \varphi (t\circ uv(x))\left( \rho^\mathcal{E}_{uv(x)}(p(x))\right)\llbracket x\in Z\rrbracket \\
				& =\alpha_{(uv,Z)}(\overline{\varphi\otimes_{T_\mathcal{E}}p})_x
		\end{align*}\end{small}
		for every $\overline{\varphi\otimes_{T_\mathcal{E}} p} \in\mathcal{H}$ and every $x\in \mathcal{G}^{(0)}.$ Consequently, $\alpha$ is a morphism of semigroups. 
		Finally, with $ i:\mathcal{G}^{(0)} \rightarrow\mathcal{G}$ being the inclusion map of the groupoid, we have that
		\begin{align*}
			\alpha_{\left( i, \G^{(0)}\right)}(\overline{\varphi \otimes_{T_\varepsilon}p})_x & = \varphi(t\circ i(x)) \left( \rho^\varepsilon_{ i(x)}(p(x))\right)\llbracket x\in \G^{(0)}\rrbracket\\
			& = \varphi(x)(p(x))\\
			& = \varepsilon(\overline{\varphi \otimes_{T_\varepsilon}p})_x
		\end{align*}
		and
		\begin{align*}
			\alpha^\ast_{\left(u, X\right)}(\overline{\varphi \otimes_{T_\varepsilon}p})_x & = (\alpha_{\left(u, X\right)}\circ \overline{t})^{ -1}\circ \alpha_{\left(u, X\right)}\circ S (\overline{\varphi \otimes_{T_\varepsilon}p})_x\\
			& = (\alpha_{\left(u, X\right)}\circ \overline{t})^{ -1}\circ \alpha_{(u,X)^\ast}\left(\overline{\varphi \otimes_{T_\mathcal{E}}p}\right)_{t\circ u(x)}\\
			& = \alpha_{(u,X)^\ast}\left(\overline{\varphi \otimes_{T_\mathcal{E}}p}\right)_x
		\end{align*}
		for every $\overline{\varphi\otimes_{T_\mathcal{E}}p}\in\mathcal{H}$ and $x\in\G^{(0)}.$ Therefore, $\alpha$ is a morphism of regular monoids.

	\end{proof}

	\begin{prop}\label{finitegroupoid}
		Let $\G$ be a finite and transitive groupoid and $\mathcal{H}=\mathcal{R}_\K(\G)$ the Hopf algebroid of representative functions of $\G.$ Then there exists an isomorphism of regular monoids between the bisections $\mathcal{B}(\G)$ of $\G$ and the set of the biretractions  $\mathfrak{B}(\mathcal{H})$ of $\mathcal{H}.$
	\end{prop}
	\begin{proof}
		We saw on Remark \ref{transitivo} that the groupoid $\G$ can be seen as the groupoid $\G^{(0)}\times G\times\G^{(0)},$ where $G$ is a group, and that $\mathcal{H}\cong A\otimes_\K R(G)\otimes_\K A,$ with $R(G)$ being the Hopf algebra of representative functions of the group $G.$ Recall from Example \ref{XGX} the Hopf algebroid structure of $A\otimes_\K R(G)\otimes_\K A$ given by the expressions (\ref{structure}). Besides that, if $(u,X)\in \mathcal{B}(\G),$ we can write
		\begin{center}
			\begin{tabular}{rl}
				$u:X$ & $\longrightarrow \G^{(0)}\times G\times \G^{(0)}$  \\
				$x$ & $\longmapsto (\lambda(x),\phi(x),x)$ 
			\end{tabular}
		\end{center}
		with $\phi:X\rightarrow\G$ and $\lambda:X\rightarrow \lambda(X)\subseteq \G^{(0)}$ a bijection,  because $x=s\circ u(x)$ for all $x\in X$ and $t\circ u=\lambda: X\rightarrow \lambda(X)$ is a bijection. Again from the Example \ref{XGX}, equation (\ref{afb}),  the morphism $\alpha$ from Proposition \ref{map_bis_birret}
		can be written for $\G$ as
		$$ \alpha_{(u,X)}(a\otimes f\otimes b)_x = a(\lambda(x))\, f(\phi(x))\, b(x)\,\llbracket x\in X\rrbracket$$
		for every $a\otimes f\otimes b\in A\otimes_\K R(G)\otimes_\K A$ and $x\in\G^{(0)}.$ We want to prove that the morphism $\alpha:\mathcal{B}(\G)\rightarrow \mathcal{B}rt(A\otimes_\K R(G)\otimes_\K A, A),$ $(u,X)\mapsto \alpha_{(u,X)}$ is bijective.
		
		First, suppose that $(u,X)$ and $(v,Y)$ are both bisections of $\G$ with $$u(x)=(\lambda(x),\phi(x),x)\qquad v(y)=(\lambda'(y),\phi'(y),y)$$
		and $\alpha_{(u,X)}=\alpha_{(v,Y)}.$ Then
		$$\llbracket x\in X\rrbracket = \alpha_{(u,X)}(1_\mathcal{H})(x)=\alpha_{(v,Y)}(1_\mathcal{H})(x)= \llbracket x\in Y\rrbracket,$$
		which implies that $X=Y.$ Also, since $\G^{(0)}$ is finite, we can take $a\in A$ such that $a|_{\lambda(X)}$ is bijective. Thus for $x\in X,$
		$$a(\lambda(x))=\alpha_{(u,X)}(a\otimes 1_{R(G)}\otimes 1_A)_x= \alpha_{(v,Y)}(a\otimes 1_{R(G)}\otimes 1_A)_x = a(\lambda'(x))$$
		implies that $\lambda=\lambda'.$ Similarly, we have that $\phi=\phi'$ and, consequently, $(u,X)=(v,Y).$ Therefore, $\alpha$ is injective.
		
		On the other hand, let $\beta:A\otimes_\K R(G)\otimes_\K A\rightarrow A$ be a local biretraction. Then by definintion,
		\begin{align*}
			& \beta(1_A\otimes 1_{R(G)}\otimes a)=\beta\circ\overline{s}(a) = a\,\beta(1_\mathcal{H}); \\
			& \beta\circ \overline{t}(a)=\beta(a\otimes 1_{R(G)}\otimes 1_A)
		\end{align*}
		for every $a\in A$ and there exists $e^\beta\in A$ such that $\beta\circ \overline{t}(e^\beta)=\beta( 1_\mathcal{H})$ with 
		$$\beta\circ \overline{t}|_{A\, e^\beta}: A\, e^\beta\longrightarrow A\,\beta(1_\mathcal{H})$$
		being a bijection. Since $\beta(1_\mathcal{H})$ and $e^\beta$ are idempotents, we have that $\beta(1_\mathcal{H})=\chi_X$ and $e^\beta=\chi_Y$ for some $X,Y\subseteq \G^{(0)}.$ Denoting $\chi_x:=\chi_{\{x\}},$ we have that
		\begin{equation}\label{b}
			\beta(1_A\otimes 1_{R(G)}\otimes a)_x=a(x)\,\beta(1_\mathcal{H})_x=a(x)\llbracket x\in X\rrbracket.
		\end{equation}
		Also,
		\begin{align*}
			& \chi_X=\beta(1_\mathcal{H})=\sum_{x\in\G^{(0)}} \beta(\chi_x\otimes 1_{R(G)}\otimes 1_A) = \sum_{x\in\G^{(0)}}\beta\circ\overline{t}(\chi_x)\\
			& \chi_X= \beta(1_\mathcal{H})=\beta\circ\overline{t}(\chi_Y)= \sum_{x\in Y} \beta(\chi_x\otimes 1_{R(G)}\otimes 1_A) = \sum_{x\in Y}\beta\circ\overline{t}(\chi_x)
		\end{align*}
		and if $x\neq y,$ then $\beta\circ \overline{t}(\chi_x)\,\beta\circ\overline{t}(\chi_y)= \beta\circ\overline{t}(\chi_x\chi_y)=0.$ Thus $\beta\circ\overline{t}(\chi_x)=0$ for all $x\in\G^{(0)}\setminus Y$ and there exists a bijection $\lambda:X\rightarrow Y,$ $x\mapsto \lambda(x)$ such that
		$$\beta\circ\overline{t}(\chi_{\lambda(x)})=\chi_x.$$ Hence for every $a\in A$ and $x\in\G^{(0)},$
		\begin{align}\label{a}
			\beta(a\otimes 1_{R(G)}\otimes 1_A)_x & = \sum_{y\in\G^{(0)}}a(y)\,\beta(\chi_y  \otimes 1_{R(G)}\otimes 1_A)_x \nonumber\\
			&=\sum_{y\in Y}a(y)\,\beta\circ\overline{t}(\chi_y)_x\nonumber\\
			& =a(\lambda(x))\llbracket x\in X\rrbracket. 
		\end{align}
		
		Finally, since $\G$ is transitive and finite, $R(G)=Fun(G,\K)$ \cite[Theorem III.1.5]{Simon}, hence every $f\in R(G)$ can be written as 
		$$f(g)=\sum_{h\in G} f(h)\, p_h,$$
		where $p_h(g)=\llbracket g=h\rrbracket$ for all $g\in G.$ Then
		$$\chi_X=\beta(1_\mathcal{H})=\sum_{g\in G} \beta(1_A\otimes p_g\otimes 1_A)$$
		with $\beta(1_A\otimes p_g\otimes 1_A)\, \beta(1_A\otimes p_h\otimes 1_A)=0$ whenever $g\neq h.$ Thus we can define a map $\phi: X\rightarrow G$ that takes each $x\in X$ to the unique $g=\phi(x)\in G$ such that $\beta(1_A\otimes p_{\phi(x)}\otimes 1_A)_x=1.$ Therefore,
		\begin{align}\label{f}
			\beta(1_A\otimes f\otimes 1_A)_x & = \sum_{g\in G} f(g)\,\beta(1_A\otimes p_g\otimes 1_A)_x\nonumber\\
			& = f(\phi(x))\llbracket x\in X\rrbracket.
		\end{align}
		
		And using the expressions (\ref{b}), (\ref{a}) and (\ref{f}), $\beta$ can be written for every $a\otimes f\otimes b$ in $A\otimes_\K R(G)\otimes_\K A,$ as
		\begin{align*}
			\beta(a\otimes f\otimes b)_x= &     \beta(a\otimes 1_{R(G)}\otimes 1_A)_x\,\beta(1_A\otimes f\otimes 1_A)_x \, \beta(1_A\otimes 1_{R(G)}\otimes b)_x\\
			& = a(\lambda(x))\, f(\phi(x))\, b(x)\, \llbracket x\in X\rrbracket\\
			& = \alpha_{(u,X)}(a\otimes f\otimes b)_x
		\end{align*}
		with $u:X\rightarrow Y\subseteq\G^{(0)},$ $x\mapsto (\lambda(x),\phi(x),x).$ Hence $\beta=\alpha_{(u,X)}$ and $\alpha$ is surjective.
	\end{proof}

	\begin{rmk}
		As a particular case from the finite and transitive groupoids, take the groupoid $\G=X\times X,$ with $X$ being a finite set. Thus a bisection $u:Y\subseteq X\rightarrow X$ of $\G$ can be written for an element $y\in Y$ as
		$$u(y)=(\lambda(y),y),$$
		where $\lambda:Y\rightarrow \lambda(Y)$ is a bijection, that is, any bisection of $\G$ is determined by a subset $Y\subseteq X$ and a bijection $\lambda: Y\rightarrow \lambda (Y)\subseteq X.$ From the Proposition \ref{finitegroupoid}, the pair $(\lambda, Y)$ also determines the biretractions for the Hopf algebroid of the representative functions of $\G.$ 
		
		On the other hand, from Example \ref{cartesiano}, the representative functions of $\G$ are given by $\mathcal{R}_\K(\G)\cong A\otimes_\K A,$ where $A=Fun(X,\K).$ From Example \ref{AA2}, a biretraction for $A\otimes_\K A$ with $A=Fun(X,\K)$ is determined by a pair $[\varphi,e]$ such that $\varphi:A\rightarrow A$ is multiplicative, $e^2=e\in A$ and $\varphi|_{Ae}:A\, e\rightarrow A\varphi(e)$ is a bijection. 
		
		These two characterizations of the biretractions are the same, because since $e$ and $\varphi(e)$ are idempotents in $A,$ there exist $Z,Y\subseteq X$ such that $e=\chi_Z$ and $\varphi(e)=\chi_Y.$ And since $X$ is finite and $\varphi$ is multiplicative, there exists a bijection $\lambda: Y\rightarrow Z$ such that for each $y\in Y,$ $\varphi(\chi_{\lambda(y)})=\chi_y.$ Therefore $[\varphi,e]$ is also determined by a subset $y\subseteq X$ and a bijection $\lambda:Y\rightarrow Z\subseteq X.$
	\end{rmk}
	
	\begin{rmk}
		A natural question would be about the relationship between bisections of algebraic groupoids (groupoid schemes) and biretractions of commutative Hopf algebroids. Given a commutative Hopf algebroid $(\mathcal{H} , A)$, consider the associated groupoid scheme 
		\[
		\left( \mathcal{G} =\text{Hom}_{Alg} (\mathcal{H}, \underline{\quad} ) \;, \; \mathcal{G}^{(0)} =\text{Hom}_{Alg} (A,\underline{\quad} ) \right)
		\] 
		A local bisection on this groupoid scheme is a pair $(v,X)$ in which $X\subseteq \mathcal{G}^{(0)}$ is an affine scheme and $v:X\Rightarrow \mathcal{G}$ is a natural transformation. The functor $X:\underline{Alg}_{\Bbbk} \rightarrow \underline{Set}$ would be represented by a quotient algebra $A/I$ for a given ideal $I\trianglelefteq A$ and, associated to the natural transformation $v$, there is a morphism of algebras $\alpha_v :\mathcal{H}\rightarrow A/I$ satisfying some conditions corresponding to the axioms (i) and (ii) of Definition \ref{bisection}. In order to deal with this problem, we need another approach for biretractions on a commutative Hopf algebroid $\mathcal{H}$, associating them to ideals of the base algebra $A$ instead of idempotent elements of $A$, which is an approach slightly more general than that used here.
		
		Another deeper question is whether one can define an inverse semigroup scheme $\Sigma$ associated to the set of bisections of a groupoid scheme $(\mathcal{G}, \mathcal{G}^{(0)})$ and verify the structure of QISG of the commutative algebra $H_{\Sigma}$ which represents it. These and other questions related to commutative Hopf algebroids are treated in a work in progress.
	\end{rmk}

	\subsection{Biretractions over noncommutative Hopf algebroids with commutative base algebras}

	We can go one step further and work with a noncommutative Hopf algebroid over a commutative algebra. In this case we have only one base algebra, which is commutative, but we still have two different structures of a left-bialgebroid and of a right bialgebroid. The definition of a biretraction for this structure should be a generalization of the definition for commutative Hopf algebroids.

	Let's consider a Hopf algebroid $\mathcal{H}$ over a commutative algebra $A$ such that $s_l=t_r=t$ and $s_r=t_l=s.$ In this case we can use the exact same definition of biretraction that we used in the commutative case: a biretraction for $\mathcal{H}$ is a multiplicative linear map $\alpha:\mathcal{H} \rightarrow A$ satisfying
	\begin{itemize}
		\item[(BRT1)] $\alpha \circ s (a) =a\,\alpha (1_{\mathcal{H}})$ for every $a\in A.$
		\item[(BRT2)] There exists $e^\alpha\in A$ such that $\alpha\circ t(e^\alpha)=\alpha(1_\mathcal{H})$ and
		$$\alpha \circ t|_{A\, e^\alpha} : A\, e^\alpha \longrightarrow A\, \alpha(1_\mathcal{H})$$ is a bijection.
	\end{itemize} 
	Denote the set of local biretractions of $\mathcal{H}$ by $\mathfrak{Brt} (\mathcal{H}, A).$
	
	\begin{rmk}
		Exactly like in the commutative case, we have that for a biretraction $\alpha:\mathcal{H}\rightarrow A,$ $\alpha(1_\mathcal{H})$ and $e^\alpha$ are idempotent elements of $A$ and $e^\alpha$ satisfying \emph{(BRT2)} is  also unique.
	\end{rmk}
	
	\begin{rmk}
		Because $A$ is commutative and $\alpha$ is multiplicative, we have that $\alpha(hk)=\alpha(kh)$ for every $h,k\in \mathcal{H}.$
	\end{rmk}
	
	\begin{rmk}
		Being $\mathcal{H}$ a Hopf algebroid over two algebras $A$ and $\overline{A},$ we have
		\begin{align*}
			\varepsilon_l\circ S(h) & = \varepsilon_l\circ S(t_l\circ \varepsilon_l(h_{(2)})\,h_{(1)})\\
			& = \varepsilon_l (S(h_{(1)})\, s_l\circ\varepsilon_l(h_{(2)}))\\
			& = \varepsilon_l (S(h_{(1)})\, h_{(2)})\\
			& = \varepsilon_l\circ s_r\circ \varepsilon_r(h)
		\end{align*}
		for all $h\in\mathcal{H}.$ Similarly, we have $\varepsilon_r\circ S=\varepsilon_r\circ s_l\circ \varepsilon_l.$ Therefore, if $\mathcal{H}$ is a Hopf algebroid with $A=\overline{A}$ and $s_l=t_r=t$ and $s_r=t_l=s$ then we have $\varepsilon_l\circ S=\varepsilon_r$ and $\varepsilon_r\circ S=\varepsilon_l.$
	\end{rmk}
	
	\begin{rmk}
		The counits $\varepsilon_l$ and $\varepsilon_r$ are not always biretractions, because they are not  ne- \newline cesssarily multiplicative functions, but given a biretraction $\alpha$ and using the notation $\Delta_r(h)=h^{(1)}\otimes_A h^{(2)}$ we have that
		\begin{align*}
			e^\alpha\varepsilon_l(h) & = (\alpha\circ t)^{-1} \alpha (t \circ \varepsilon_l)(h)\\
			& = (\alpha\circ t)^{-1}\alpha (h^{(1)}\, S(h^{(2)}))
		\end{align*}
		for every $h\in\mathcal{H}.$	Then $e^\alpha\varepsilon_l$ is a biretraction with $e^{e^\alpha\varepsilon_l}=e^\alpha.$ And using the identity $\varepsilon_r=\varepsilon_l\circ S,$ we have that $e^\alpha\,\varepsilon_r$ is also a biretraction. More than that, for every $h\in \mathcal{H},$
		$$\alpha(1_\mathcal{H})\,\varepsilon_r(h)=\alpha\circ (s\circ \varepsilon_r)(h)=\alpha(S(h_{(1)})\, h_{(2)}),$$
		which is also multiplicative. Then $\alpha(1_\mathcal{H})\,\varepsilon_r$ is a biretraction with $e^{\alpha(1_\mathcal{H})\,\varepsilon_r}=\alpha(1_\mathcal{H})$ and using $\varepsilon_l=\varepsilon_r\circ S,$ so is $\alpha(1_\mathcal{H})\,\varepsilon_l.$

	\end{rmk}
	
	\begin{thm}\label{regular_semigroup}
		Let $\mathcal{H}$ be a Hopf algebroid  over a commutative algebra $A$ such that $s_l=t_r=t,$ $s_r=t_l=s.$ Then the set $\mathfrak{Brt} (\mathcal{H}, A)$ of local biretractions of $\mathcal{H}$ is a regular semigroup.
	\end{thm}
	\begin{proof}
		
		Define the convolution product between two biretractions $\alpha$ and $\beta$ by the expression
		$$(\alpha\ast \beta)(h)= \beta(\alpha(h_{(1)})\triangleright h_{(2)})= \beta t\alpha(h_{(1)}) \beta(h_{(2)}).$$
		
		Like in the commutative case, this product is associative and well-defined. Now define the pseudo-inverse for any biretraction $\alpha \in \mathfrak{Brt} (\mathcal{H}, A)$ and $h\in \mathcal{H}$ as
		\begin{align*}
			\alpha^*(h) & := (\alpha \circ t)^{-1} \alpha (S (\varepsilon_l (h^{(1)}) \triangleright h^{(2)}))\\
			& \,= (\alpha\circ t)^{-1}\circ \alpha (S((t\circ\varepsilon_l) (h^{(1)})\, h^{(2)}))\\
			& \,= (\alpha\circ t)^{-1}\circ \alpha( S(h^{(2)}) \, (s\circ\varepsilon_l) (h^{(1)}))\\
			& \,= (\alpha\circ t)^{-1} \left(\varepsilon_l(h^{(1)})\,(\alpha \circ S)(h^{(2)})\right).
		\end{align*}
		
		Observe that all maps $\alpha\circ t,$ $S,\, \Delta_r$ and $t\circ\varepsilon_l$ are multiplicative or antimultiplicative, thus $\alpha^\ast$ is multiplicative. Moreover, we have for every $a\in A$ that 
		$$\alpha^\ast\circ s(a)=(\alpha\circ t)^{-1}(\varepsilon_l (1_\mathcal{H})\, \alpha ((S\circ s)(a)))= (\alpha\circ t)^{-1} (\alpha \circ t)(a)=a e^\alpha=a\,\alpha^\ast(1_\mathcal{H})$$
		and 
		\begin{align*}
			\alpha^\ast\circ t(a\, \alpha(1_\mathcal{H})) & = (\alpha\circ t)^{-1}((\varepsilon_l\circ t) (a\, \alpha(1_\mathcal{H}))\, \alpha (S(1_\mathcal{H})))\\
			& = (\alpha\circ t)^{-1} (a\, \alpha(1_\mathcal{H})),
		\end{align*}
		that is, $\alpha^\ast\circ t|_{A\alpha(1_\mathcal{H})}= (\alpha\circ t)^{-1}|_{A\alpha(1_\mathcal{H})}.$ Then $\alpha^\ast$ is a biretraction with $e^{\alpha^\ast}=\alpha(1_\mathcal{H}).$ Also, \begin{align}\label{alpha_estrela}
			\alpha\ast\alpha^\ast (h) & = \alpha^\ast t\alpha (h_{(1)})\,\alpha^\ast (h_{(2)})   \nonumber\\
			& = (\alpha\circ t)^{-1}((\varepsilon_l\circ t) \alpha(h_{(1)}) \,\alpha (S(1_\mathcal{H}))) \, \, (\alpha\circ t)^{-1} \left(\varepsilon_l \left({h_{(2)}}^{(1)}\right)\alpha\circ S \left({h_{(2)}}^{(2)}\right)\right) \nonumber \\
			& = (\alpha\circ t)^{-1}  \alpha \left(h_{(1)} \, (s\circ     \varepsilon_l ) \left({h_{(2)}}^{(1)}\right)  S \left({h_{(2)}}^{(2)} \right)\right) \nonumber \\
			& = (\alpha\circ t)^{-1}  \alpha \left( {h^{(1)}}_{(1)}\, (s\circ\varepsilon_l) \left({h^{(1)}}_{(2)}   \right) S(  h^{(2)}   )\right) \nonumber \\
			& = (\alpha\circ t)^{-1} \alpha \,(h^{(1)}S(h^{(2)})) \nonumber \\
			& = (\alpha\circ t)^{-1} (\alpha\circ t) \varepsilon_l (h) \nonumber \\
			& = e^\alpha\,\varepsilon_l(h)
		\end{align}
		and
		\begin{align}\label{estrela_alpha}
			\alpha^\ast \ast \alpha \,(h) & = \alpha t \alpha^\ast (h_{(1)})\,\alpha (h_{(2)})  \nonumber \\
			& = (\alpha\circ t)  (\alpha\circ t)^{-1} \left( \varepsilon_l\left({h_{(1)}}^{(1)}\right)\alpha\circ S\left( {h_{(1)}}^{(2)} \right)    \right) \alpha(h_{(2)}) \nonumber \\
			& = \varepsilon_l(h^{(1)})\, \alpha\left( S\left({h^{(2)}}_{(1)} \right) {h^{(2)}}_{(2)}  \right) \nonumber\\
			& = \varepsilon_l(h^{(1)})\, (\alpha\circ s) (\varepsilon_r (h^{(2)}))\nonumber \\
			& = \varepsilon_l(h^{(1)})\,\varepsilon_r(h^{(2)})\,\alpha (1_\mathcal{H}) \nonumber \\
			& = \varepsilon_l(s\circ \varepsilon_r(h^{(2)}) h^{(1)}) \,\alpha(1_\mathcal{H}) \nonumber \\
			& \stackrel{(\ast)}{=} \varepsilon_l (h^{(1)} (s\circ \varepsilon_r) (h^{(2)}))\,\alpha(1_\mathcal{H}) \nonumber\\
			& = \alpha(1_\mathcal{H})\, \varepsilon_l(h)
		\end{align}
		for every $h\in\mathcal{H}.$ Recall that for a biretraction $\alpha,$ we have $\alpha(hk)=\alpha(kh)$ for every $h,k\in \mathcal{H},$ which was used in $(\ast)$ for the biretraction $\alpha(1_\mathcal{H})\,\varepsilon_l.$
		
		Now using the identities (\ref{alpha_estrela}) and (\ref{estrela_alpha}), we get
		\begin{align}\label{alpha}
			\alpha\ast \alpha^\ast \ast \alpha \,(h) & = \alpha t  (\alpha\ast \alpha^\ast)(h_{(1)}) \,\alpha(h_{(2)}) \nonumber\\
			& =(\alpha \circ t) (e^\alpha\varepsilon_l(h_{(1)}))\, \alpha(h_{(2)}) \nonumber \\
			& = \alpha(t\circ\varepsilon_l(h_{(1)})h_{(2)}) \nonumber\\
			& = \alpha(h)
		\end{align}
		and
		\begin{align}\label{estrela}
			\alpha^\ast \ast\alpha\ast \alpha^\ast (h)& = \alpha^\ast t  (\alpha^\ast \ast\alpha)(h_{(1)})\, \alpha^\ast (h_{(2)}) \nonumber\\
			& = (\alpha\circ t)^{-1} \left( (\varepsilon_l\circ t) (\alpha^\ast \ast \alpha)(h_{(1)})\,\alpha(1_\mathcal{H}))\, \,\varepsilon_l\left({h_{(2)}}^{(1)}\right)\, \alpha\circ S\left( {h_{(2)}}^{(2)}\right)\right)\nonumber \\
			& = (\alpha\circ t)^{-1}\left(\alpha(1_\mathcal{H}) \,\varepsilon_l(h_{(1)}) \,\varepsilon_l \left({h_{(2)}}^{(1)}\right)\, \alpha\circ S\left( {h_{(2)}}^{(2)}\right)    \right)\nonumber\\
			& = (\alpha\circ t)^{-1}\left(\varepsilon_l \left({h^{(1)}}_{(1)}\right) \varepsilon_l\left( {h^{(1)}}_{(2)}\right)\, \alpha\circ S(h^{(2)})   \right) \nonumber\\
			& = (\alpha\circ t)^{-1}\left( \varepsilon_l\left(t\circ \varepsilon_l \left({h^{(1)}}_{(1)}\right) {h^{(1)}}_{(2)}\right)\, \alpha\circ S(h^{(2)})   \right) \nonumber\\
			& = (\alpha\circ t)^{-1}(\varepsilon_l(h^{(1)})\,\alpha \circ S(h^{(2)})) = \alpha^\ast(h)
		\end{align}
		for all $h\in \mathcal{H}.$ Therefore, $\mathfrak{Brt} (\mathcal{H}, A)$ is a regular semigroup.
	\end{proof}
	
	\begin{rmk}
		Observe that given a biretraction $\alpha: \mathcal{H}\rightarrow A,$
		\begin{align*}
			\left((e^\alpha\,\varepsilon_l) \ast \alpha\right)(h) & = \alpha\circ t( e^\alpha\,\varepsilon_l(h_{(1)}))\,\alpha(h_{(2)})\\
			& = \alpha(t\circ\varepsilon_l(h_{(1)})h_{(2)})=\alpha(h)
		\end{align*}
		and
		\begin{align*}
			\left(\alpha\ast (\alpha(1_\mathcal{H})\, \varepsilon_l)\right)(h) & = \alpha(1_\mathcal{H})\, (\varepsilon_l\circ t)\alpha(h_{(1)}) \,\alpha(1_\mathcal{H})\,\varepsilon_l(h_{(2)})\\
			& = \alpha(h_{(1)})\,\varepsilon_l(h_{(2)})\\
			& = \alpha \left((s\circ\varepsilon_l)(h_{(2)})h_{(1)}\right)= \alpha(h)
		\end{align*} 
		for every $h\in\mathcal{H}.$ Also note that
		\begin{align*}
			(e^\alpha\varepsilon_l)^\ast (h)&= (e^\alpha\varepsilon_l \circ t)^{-1}(\varepsilon_l(h^{(1)})\, e^\alpha\varepsilon_l \circ S(h^{(2)}))\\
			& = e^\alpha\varepsilon_l(h^{(1)})\,\varepsilon_r(h^{(2)})= e^\alpha \varepsilon_l(h)
		\end{align*}
		and analogously, $(\alpha(1_\mathcal{H})\,\varepsilon_l)^\ast =\alpha(1_\mathcal{H})\,\varepsilon_l.$
	\end{rmk}
	
	Now consider the free vector space generated by the biretractions of $\mathcal{H}$ and extend linearly the convolution product to this space. Then, we have an algebra structure on the space $\K \mathfrak{Brt} (\mathcal{H}, A)$, henceforth denoted by $\mathfrak{B} (\mathcal{H})$.
	
	\begin{thm}
		Let $\mathcal{H}$ be a Hopf algebroid over a commutative algebra $A$ such that $s_l=t_r=t,$ $s_r=t_l=s.$ Then the algebra $\mathfrak{B}(\mathcal{H})$, generated by the set of biretractions of $\mathcal{H}$ with the convolution product is a quantum inverse semigroup.
	\end{thm}
	
	\begin{proof}
		With the comultiplication $\underline{\Delta} :\mathfrak{B} (\mathcal{H}) \rightarrow \mathfrak{B} (\mathcal{H})$ defined as $\underline{\Delta}(\alpha )=\alpha \otimes \alpha$ for every $\alpha\in \mathfrak{Brt} (\mathcal{H}, A),$ we have that $\underline{\Delta}$ is multiplicative, just like in the commutative case. Also, defining $\mathcal{S}(\alpha)=\alpha^\ast,$ we have from the expressions (\ref{alpha}) and (\ref{estrela}) from Theorem \ref{regular_semigroup} that $I\ast \mathcal{S}\ast I= I$ and $\mathcal{S}\ast I\ast \mathcal{S}=\mathcal{S}.$ Moreover, $\mathcal{S}$ is antimultiplicative: for every $\alpha,\beta \in \mathfrak{Brt} (\mathcal{H}, A)$ and $h\in\mathcal{H},$
		
		\begin{small}
			\begin{align*}
				\mathcal{S}(\alpha\ast\beta)(h) & = (\alpha\ast\beta)^\ast (h)\\
				& = ((\alpha\ast\beta)\circ t)^{-1}(\varepsilon_l(h^{(1)}) \,(\alpha\ast\beta)\circ S(h^{(2)}))\\
				& = (\alpha\circ t)^{-1}\circ (\beta\circ t)^{-1} \left(\varepsilon_l(h^{(1)})\,\beta\circ t\circ\alpha \left(\left(S(h^{(2)})\right)_{(1)}\right) \, \beta \left(\left(S(h^{(2)})\right)_{(2)}\right)     \right)\\
				& \stackrel{(\ast)}{=} (\alpha\circ t)^{-1}\left( (\beta\circ t)^{-1} \left(\varepsilon_l(h^{(1)})\, \beta\circ S(h^{(2)(1)}) \right) \, \alpha\circ S(h^{(2)(2)})\right)\\
				& = (\alpha\circ t)^{-1}\left(\beta^\ast (h^{(1)})\, \alpha\circ S(h^{(2)})\right),
		\end{align*}\end{small}
		where in $(\ast)$ we used the property $\Delta_l\circ S = (S\otimes_l S)\circ \Delta_r^{cop},$ which holds for any Hopf algebroid. Conversely,
		
		\begin{small}
			\begin{align*}
				(\mathcal{S}(\beta)\,\,\ast &\,\, \mathcal{S}(\alpha))(h) = (\beta^\ast \ast \alpha^\ast)(h)\\
				& = \alpha^\ast\circ t \circ \beta^\ast (h_{(1)})\, \alpha^\ast (h_{(2)})\\
				& = (\alpha\circ t)^{-1} (\varepsilon_l\circ t\circ \beta^\ast (h_{(1)})\,\alpha(1_\mathcal{H}))\, \alpha^\ast (h_{(2)})\\
				& = (\alpha\circ t)^{-1} \left( (\beta\circ t)^{-1} \left( \varepsilon_l\left({h_{(1)}}^{(1)}\right)\, \beta\circ S\left({h_{(1)}}^{(2)}\right) \right) \,\varepsilon_l\left({h_{(2)}}^{(1)}\right) \alpha\circ S \left({h_{(2)}}^{(2)}\right)\right)\\
				& = (\alpha\circ t)^{-1} \left( (\beta\circ t)^{-1} \left( \varepsilon_l\left(\aga{1}{1}{1}\right)\, \beta\circ S\left(\aga{1}{2}{1}\right) \,\beta\circ t\circ \varepsilon_l\left({h^{(1)}}_{(2)}\right)\right) \alpha\circ S (h^{(2)})\right)\\
				& = (\alpha\circ t)^{-1} \left( (\beta\circ t)^{-1} \left( \varepsilon_l\left(h^{(1)(1)}\right)\, \beta\left( S\left({h^{(1)(2)}}_{(1)}\right) \, t\circ \varepsilon_l\left({h^{(1)(2)}}_{(2)}\right)\right)\right) \alpha\circ S (h^{(2)})\right)\\
				& = (\alpha\circ t)^{-1} \left( (\beta\circ t)^{-1} \left( \varepsilon_l\left(h^{(1)(1)}\right)\, \beta\circ S(h^{(1)(2)})\right) \alpha\circ S (h^{(2)})\right)\\
				& = (\alpha\circ t)^{-1}\left(\beta^\ast (h^{(1)})\, \alpha\circ S(h^{(2)})\right).
		\end{align*}\end{small}

		Consequently, $\mathcal{S}(\alpha\ast\beta) =  \mathcal{S}(\beta)\ast \mathcal{S}(\alpha)$ and the axiom (QISG3) is verified.
		
		Finally, for checking axiom (QISG4) for any $\alpha ,\beta \in \mathfrak{Brt}(\mathcal{H},A),$ the expressions (\ref{alpha_estrela}) and (\ref{estrela_alpha}) imply that
		\begin{eqnarray*}
			\alpha_{(1)} \ast \mathcal{S}(\alpha_{(2)}) \ast \mathcal{S}(\beta_{(1)})\ast \beta_{(2)}(h) & = & (\alpha \ast \alpha^*)\ast (\beta^* \ast \beta)(h) \\
			& = & (\beta^* \ast \beta)  t  (\alpha \ast \alpha^*) (h_{(1)})\, (\beta^* \ast \beta)(h_{(2)}) \\
			& = & \beta(1_\mathcal{H})\,(\varepsilon_l\circ t) (\alpha \ast \alpha^*) (h_{(1)}) \,\varepsilon_l(h_{(2)})\\
			& = & \beta(1_\mathcal{H})\, e^\alpha\,\varepsilon_l (h_{(1)})\,\varepsilon_l(h_{(2)})\\
			& = & \beta(1_\mathcal{H})\, e^\alpha\, \varepsilon_l(h)
		\end{eqnarray*}	
		The same result for $\mathcal{S}(\beta_{(1)})\ast \beta_{(2)} \ast \alpha_{(1)} \ast \mathcal{S}(\alpha_{(2)})(h)$.
		
		Therefore, $\mathfrak{B} (\mathcal{H})$ is a Quantum Inverse Semigroup.
	\end{proof}

	\begin{rmk}\label{s=t}
		Consider a Hopf algebroid $\mathcal{H}$ over a commutative algebra $A$ with $s=s_l=t_l=s_r=t_r.$ A local biretraction for $\mathcal{H}$ is a linear and multiplicative map $\alpha:\mathcal{H}\rightarrow A$ that satisfies $\alpha \circ s (a) =a\,\alpha (1_{\mathcal{H}})$ for every $a\in A$ and there exists $e^\alpha\in A$ such that $\alpha\circ s(e^\alpha)=\alpha(1_\mathcal{H})$ and
		$\alpha \circ s|_{A\, e^\alpha} : A\, e^\alpha \longrightarrow A\, \alpha(1_\mathcal{H})$ is a bijection.
		Combining both conditions we have
		\begin{align*}
			\alpha\circ s(\alpha(1_\mathcal{H})\, e^\alpha) & = \alpha\circ s(\alpha (1_\mathcal{H}))\, \alpha\circ s(e^\alpha)\\
			& = \alpha (1_\mathcal{H})\,\alpha (1_\mathcal{H})\,\alpha (1_\mathcal{H}) \\
			& = \alpha (1_\mathcal{H}).
		\end{align*}
		Since $\alpha(1_\mathcal{H})\, e^\alpha\in A\, e^\alpha$ then $\alpha (1_\mathcal{H})\, e^\alpha = e^\alpha.$ Therefore
		$$\alpha (1_\mathcal{H}) = \alpha\circ s (e^\alpha) = e^\alpha\, \alpha (1_\mathcal{H}) = e^\alpha.$$
		
		Moreover, for every $a\in A,$
		$$\alpha\circ s (a\,\alpha (1_\mathcal{H})) = a\alpha (1_\mathcal{H})\, \alpha (1_\mathcal{H})= a\, \alpha (1_\mathcal{H}).$$
		Consequently,  we can describe a local biretraction for $\mathcal{H}$ as a linear and multiplicative map $\alpha:\mathcal{H}\rightarrow A$ such that $\alpha\circ s|_{A\,\alpha (1_\mathcal{H})} = \emph{\textrm{Id}}_{A\,\alpha (1_\mathcal{H})}.$
	\end{rmk}
	
	\begin{exmp}\label{weakalgebroid}
		Recall the definition of a weak Hopf algebra from \emph{Example \ref{weak}}. A weak Hopf algebra $(H,\mu,\eta,\Delta,\varepsilon, S)$ has a structure of Hopf algebroid over the algebras $H_t=\varepsilon_t(H)$ and $H_s=\varepsilon_s(H)$ given by
		$$s_r(x)=x \hspace{30pt} t_r(x)=\varepsilon (x1_{(1)}) 1_{(2)} \hspace{30pt} \Delta_r=\pi_{H_s}\circ\Delta \hspace{30pt} \varepsilon_r=\varepsilon_s$$
		for every $x\in H_s,$ where $\pi_{H_s}:H\otimes_\K H\rightarrow H \otimes_{H_s} H$ and
		$$s_l(x)=x \hspace{30pt} t_l(x)=\varepsilon (1_{(2)} x) 1_{(1)} \hspace{30pt} \Delta_l=\pi_{H_t}\circ\Delta \hspace{30pt} \varepsilon_l=\varepsilon_t$$
		for every $x\in H_t,$ where $\pi_{H_t}:H\otimes_\K H\rightarrow H \otimes_{H_t} H.$
		
		Observe that for every $x\in H_s,$ $x$ can be written as $x=\varepsilon_s (h)= 1_{(1)}\,\varepsilon (h1_{(2)})$ for some $h\in H.$ Then
		\begin{align*}
			\varepsilon_s(x) & = 1_{(1)}\,\varepsilon(x 1_{(2)})\\
			& = 1_{(1)}\,\varepsilon (1_{(1')}\,\varepsilon (h 1_{(2')})1_{(2)})\\
			& = 1_{(1)}\,\varepsilon (h 1_{(2')})\,\varepsilon (1_{(1')}1_{(2)})\\
			& = 1_{(1)}\,\varepsilon (h1_{(2)})= x.
		\end{align*}
		Similarly, we have that $\varepsilon_t(x)= \varepsilon(1_{(1)} x)1_{(2)}=x$ for every $x\in H_t.$
		
		Now suppose that $H_t=H_s$ and that $A:=H_t=H_s$ is commutative. Then, for every $x\in A,$ we have that
		$$1_{(1)}\varepsilon(x 1_{(2)})=x=\varepsilon(1_{(1)} x)1_{(2)},$$
		which implies that
		\begin{align*}
			t_r(x) & = \varepsilon (x1_{(1)})\,1_{(2)}\\
			& = \varepsilon (\varepsilon (1_{(1')}x)\,1_{(2')}1_{(1)})\,1_{(2)}\\
			& = \varepsilon (1_{(2')}1_{(1)})\,\varepsilon (1_{(1')}x)\,1_{(2)}\\
			& = \varepsilon (1_{(1')}1_{(2)})\,\varepsilon (1_{(1)}x)\,1_{(2')}\\
			& = \varepsilon(1_{(1)}x)\,1_{(2)}= x
		\end{align*}
		and
		\begin{align*}
			t_l(x) & = \varepsilon (1_{(2)}x)\,1_{(1)}\\
			& = \varepsilon (1_{(2)}1_{(1')}\,\varepsilon(x1_{(2')}))\,1_{(1)}\\
			& = \varepsilon(x1_{(2')})\,\varepsilon (1_{(2)}1_{(1')})\,1_{(1)}\\
			& = \varepsilon(x1_{(2)})\,\varepsilon (1_{(1)}1_{(2')})\,1_{(1')}\\
			& = \varepsilon (x1_{(2)})\,1_{(1)}= x.
		\end{align*}
		Therefore, we have that $s_l=t_r=s_r=t_l$ are all the inclusion map $A\rightarrow \mathcal{H}.$ We also have that
		$$x=\varepsilon(x1_{(1)})1_{(2)}= \varepsilon(x1_{(2)})1_{(1)} = \varepsilon(1_{(2)}x)1_{(1)}= \varepsilon(1_{(1)}x)1_{(2)}$$
		for every $x\in A$ and if $h\in H.$
		
		Then by the Remark \ref{s=t}, a local biretraction for a weak Hopf algebra with $A:=H_t=H_s$ commutative is a linear and multiplicative map $\alpha:\mathcal{H}\rightarrow A$ such that $\alpha|_{A\,\alpha(1_\mathcal{H})}= Id_{A\, \alpha(1_\mathcal{H})}.$ 
	\end{exmp}

	\begin{exmp}
		As a particular case from the previous example, consider a finite groupoid $\mathcal{G}$ and its groupoid algebra $\K\mathcal{G}$ given by
		$$\K\mathcal{G}= \left\{ \sum_{g\in\mathcal{G}} a_g\, \delta_g\, |\, g\in\mathcal{G},\, a_g\in\K \right\}$$
		with product $\delta_g\delta_h=\delta_{gh}$ if $g,h\in\mathcal{G}$ are multipliable and $\delta_g\delta_h=0,$ otherwise. $\K\mathcal{G}$ is an algebra with  unity
		$$1_{\K\mathcal{G}}=\sum_{x\in \mathcal{G}^{(0)}}\delta_{1_x}$$
		and a coalgebra with structure given in its base elements by $\Delta(\delta_g)=\delta_g\otimes\delta_g$ and $\varepsilon(\delta_g)=1.$ From the \emph{Example \ref{weak}}, $\K\mathcal{G}$ is a weak Hopf algebra with
		$$\varepsilon_t(\delta_g)= \varepsilon(1_{(1)}\delta_g)\, 1_{(2)}= \sum_{x\in\mathcal{G}^{(0)}}\varepsilon(\delta_{1_x} \delta_g)\,\delta_{1_x}=\varepsilon(\delta_{1_{t(g)}} \delta_g)\, \delta_{1_{t(g)}}=\delta_{1_{t(g)}},$$
		$$\varepsilon_s(\delta_g)= 1_{(1)}\,\varepsilon(\delta_g 1_{(2)})= \sum_{x\in\mathcal{G}^{(0)}}\delta_{1_x} \varepsilon( \delta_g \delta_{1_x})= \delta_{1_{s(g)}} \varepsilon(\delta_g\,\delta_{1_{s(g)}} )=\delta_{1_{s(g)}}$$
		and $S(\delta_g)=\delta_{g^{-1}}$ for every $g\in\mathcal{G}.$ Finally, $\K\mathcal{G}$ also has a Hopf algebroid structure over the algebra $A=\langle \delta_{1_x}\,|\, x\in\mathcal{G}^{(0)}\rangle$ given by $s_l=t_l=s_r=t_r$ being the inclusion maps $A\rightarrow \K\mathcal{G},$
		$$\Delta_l=\Delta_r=\pi_A\circ\Delta \qquad \varepsilon_l=\varepsilon_t \qquad \varepsilon_r=\varepsilon_s$$
		and the same $S.$
		
		Observe that $A$ is a commutative algebra. Hence by the \emph{Remark \ref{s=t}}, a biretraction for $\K\mathcal{G}$ is a linear and multiplicative map $\alpha: \K\mathcal{G} \rightarrow A$ such that $\alpha|_{A\alpha(1_{\K\mathcal{G}})}=\text{\emph{Id}}_{A\alpha(1_{\K\mathcal{G}})}.$ Now we have for any $\alpha: \K\mathcal{G} \rightarrow A$ biretraction,
		\begin{itemize}
			\item $\alpha(1_{\K\mathcal{G}})$ is an idempotent. Then, $\alpha(1_{\K\mathcal{G}})$ can be written as
			$$\alpha(1_{\K\mathcal{G}})= \sum_{x\in X}\delta_{1_x}$$
			for some $X\subseteq\mathcal{G}^{(0)}.$ If $X=\mathcal{G}^{(0)},$ we have a global biretraction.
			
			\item If $y\in X,$
			$$\delta_{1_y}=\delta_{1_y}\alpha(1_{\K\mathcal{G}})= \alpha(\delta_{1_y}\alpha(1_{\K\mathcal{G}}))= \sum_{x\in X} \alpha(\delta_{1_y} \delta_{1_x})= \alpha(\delta_{1_y}).$$
			Then
			$$\sum_{x\in X}\delta_{1_x}= \alpha(1_{\K\mathcal{G}})= \sum_{y\in\mathcal{G}^{(0)}}\alpha(\delta_{1_y})= \sum_{x\in X}\delta_{1_x}+ \sum_{z\in\mathcal{G}^{(0)} \setminus X}\alpha(\delta_{1_z}),$$
			which implies that $\sum_{z\in\mathcal{G}^{(0)} \setminus X}\alpha(\delta_{1_z})=0,$ hence
			$$\alpha(\delta_{1_y})= \alpha(\delta_{1_y})\left( \sum_{z\in\mathcal{G}^{(0)} \setminus X}\alpha(\delta_{1_z})\right) = 0$$
			for every $y\in\mathcal{G}^{(0)}\setminus X.$
			
			\item Now for any $g\in\mathcal{G},$ we can write
			$$\alpha(\delta_g)=\sum_{y\in\mathcal{G}^{(0)}} a_y^g \delta_{1_y}.$$
			If $s(g)\notin X,$ then
			$$\alpha(\delta_g)=\alpha(\delta_g\delta_{1_{s(g)}})=\alpha(\delta_g)\,\alpha(\delta_{1_{s(g)}})=0.$$
			If $t(g)\notin X,$ then
			$$\alpha(\delta_g)=\alpha(\delta_{1_{t(g)}}\delta_g)= \alpha(\delta_{1_{t(g)}})\,\alpha(\delta_g)=0.$$
			If $s(g)$ and $t(g)$ are in $X,$
			$$\alpha(\delta_g)=\alpha(\delta_g\delta_{1_{s(g)}})= \alpha(\delta_g)\,\alpha(\delta_{1_{s(g)}})= \sum_{y\in\mathcal{G}^{(0)}} a_y^g \delta_{1_y}\delta_{1_{s(g)}}= a_{s(g)}^g\delta_{ 1_{s(g)}}$$
			and
			$$\alpha(\delta_g)=\alpha(\delta_{1_{t(g)}}\delta_g)= \alpha(\delta_{1_{t(g)}})\,\alpha(\delta_g)= \sum_{y\in\mathcal{G}^{(0)}} a_y^g \delta_{1_{t(g)}}\delta_{1_y}= a_{t(g)}^g\delta_{ 1_{t(g)}}$$
			with $a_{s(g)}^g,\, a_{t(g)}^g\in\K.$ Hence for $\alpha(\delta_g)$ to be nonzero, we need $s(g)=t(g)\in X.$ Moreover, we have that if $s(g)=t(g)=x\in X$ then $\alpha(\delta_g)=a_g\,\delta_{1_x}$ with $a_g\in\K \setminus \{0\}.$ In fact, if $\alpha(\delta_g)=0$ then
			$$0=\alpha(\delta_g)\,\alpha(\delta_{g^{-1}})= \alpha(\delta_{1_{t(g)}})=\alpha(\delta_{1_x}) = \delta_{1_x},$$
			which is a contradiction.
			
			\item $\mathfrak{Brt}(\K\mathcal{G},A)$ is commutative: for any $\alpha,\beta\in\mathfrak{Brt} (\K\mathcal{G},A)$ we have
			$$\alpha(\delta_g)= \begin{cases} a_g\,\delta_{1_x}, & \textrm{ if } s(g)=t(g)=x \in X\subseteq \mathcal{G}^{(0)}\\ 0, & \textrm{ otherwise}  \end{cases}$$
			and
			$$\beta(\delta_g)= \begin{cases} b_g\,\delta_{1_y}, & \textrm{ if } s(g)=t(g)=y \in Y\subseteq \mathcal{G}^{(0)}\\ 0, & \textrm{ otherwise}  \end{cases}$$
			with $a_g\in\K\setminus\{0\},$ $b_g\in\K\setminus \{0\},$ $a_{1_x}=1$ and $b_{1_y}=1$ for every $x\in X$ and $y\in Y.$ Then
			\begin{align*}
				(\alpha\ast\beta)(\delta_g) & = \beta\circ t\circ\alpha (\delta_g)\, \beta(\delta_g) = \beta\circ\alpha(\delta_g)\,\beta(\delta_g)\\
				& = \beta(a_g\,\delta_{1_x})\,\beta(\delta_g)\, \llbracket s(g)=t(g)=x\in X\rrbracket\\
				& = a_g b_g\,\delta_{1_x} \, \llbracket s(g)=t(g)=x\in X\cap Y\rrbracket\\
				& = (\beta\ast\alpha)(\delta_g)
			\end{align*}
			for every $g\in\mathcal{G}.$ Observe that this means that $\mathfrak{Brt}(\K\mathcal{G},A)$ is an inverse semigroup, with $\alpha^\ast$ given by
			\begin{align*}
				\alpha^\ast(\delta_g) & = (\alpha\circ t)^{-1}\left( \varepsilon_l(\delta_g)\,\alpha \circ S(\delta_g)\right) \\
				& = \delta_{1_{t(g)}}\,\alpha(\delta_{g^{-1}})= \alpha(\delta_{1_{t(g)}}\,\delta_{g^{-1}})\\
				& = \alpha(\delta_{g^{-1}})  = \alpha\circ S(\delta_{g})
			\end{align*}
			for every $g\in\mathcal{G}.$ $\mathfrak{Brt}(\K\mathcal{G},A)$ also has a unity $\emph{\textbf{1}}:\K\mathcal{G}\rightarrow A$ given by
			$$\emph{\textbf{1}}(\delta_g)=\delta_{1_x}\, \llbracket s(g)=t(g)=x\rrbracket$$
			for every $g\in\mathcal{G}.$
		\end{itemize}
		
		With these remarks, we can represent the biretractions using the characters from the isotropy groups $G_x=\{g\in \mathcal{G}\,|\, s(g)=t(g)=x\}.$ Being $\mathcal{G}^{(0)}=\{x_1,\ldots,x_n\},$ consider the algebra
		$$\mathcal{F}=\prod_{i=1}^n\{\varphi_i:G_{x_i} \rightarrow \K\setminus\{0\} \textrm{ morphism of groups}\} \cup \{0=\varphi_i: G_{x_i}\rightarrow \K\}$$
		with the pointwise product. The elements of $\mathcal{F}$ are $n-$tuple of characters from the isotropy groups of $G$ or zero maps.  $\mathcal{F}$ is also a commutative inverse semigroup with $(\varphi_1,\ldots,\varphi_n)^\ast = (\varphi_1^\ast,\ldots,\varphi_n^\ast),$ where
		$$\varphi_i^\ast(g)=\begin{cases}\varphi_i(g^{-1}), & \textrm{ if }\varphi_i\neq 0\\ 0, &\textrm{ if } \varphi_i=0.\end{cases}$$
		For each $(\varphi_1,\ldots, \varphi_n)\in\mathcal{F}$ and $g\in\mathcal{G},$ we can define
		$$\alpha_{(\varphi_1,\ldots, \varphi_n)}(\delta_g) = \begin{cases}
			\varphi_i(g)\,\delta_{1_{x_i}}, & \text{ if } s(g)=t(g)=x_i\\
			0, & \text{ if }s(g)\neq t(g).
		\end{cases}$$
		Then for every $(\varphi_1,\ldots,\varphi_n),\, (\psi_1,\ldots,\psi_n)\in \mathcal{F}$ and $g\in\mathcal{G,}$
		\begin{align*}
			\alpha_{(\varphi_1,\ldots, \varphi_n)}\ast \alpha_{(\psi_1,\ldots,\psi_n)}(\delta_g) & = \alpha_{(\psi_1,\ldots,\psi_n)}\circ \alpha_{(\varphi_1,\ldots, \varphi_n)}(\delta_g) \,\alpha_{(\psi_1,\ldots,\psi_n)}(\delta_g)\\
			& = \alpha_{(\psi_1,\ldots,\psi_n)}(\varphi_i(g)\, \delta_{1_{x_i}}) \, \alpha_{(\psi_1,\ldots,\psi_n)}(\delta_g)\, \llbracket s(g)=t(g)=x_i\rrbracket\\
			& = \varphi_i(g)\psi_i(g)\psi_i(1_{x_i})\, \delta_{1_{x_i}}\, \llbracket s(g)=t(g)=x_i\rrbracket\\
			& = \varphi_i(g)\psi_i(g)\, \delta_{1_{x_i}}\, \llbracket s(g)=t(g)=x_i\rrbracket \\
			& = \alpha_{(\varphi_1\psi_1,\ldots,\varphi_n \psi_n)}( \delta_g)\\
			& = \alpha_{(\varphi_1,\ldots,\varphi_b)(\psi_1, \ldots,\psi_n)}(\delta_g)
		\end{align*}
		and the map
		\begin{center}
			\begin{tabular}{rcll}
				$\alpha:$ & $\mathcal{F}$ & $\longrightarrow$ & $\mathfrak{Bir}(\K\mathcal{G},A)$\\
				& $(\varphi_1,\ldots,\varphi_n)$ & $\longmapsto$ & $\alpha_{(\varphi_1,\ldots,\varphi_n)}$
			\end{tabular}
		\end{center}
		is an isomorphism of inverse semigroups, because
		\begin{align*}
			\alpha_{(\varphi,\ldots,\varphi_n)^\ast}(\delta_g) & = \varphi_i^\ast (g)\,\delta_{1_{x_i}}\, \llbracket s(g)=t(g)=x_i\rrbracket\\
			& = \varphi_i(g^{-1})\,\delta_{1_{x_i}} \, \llbracket s(g^{-1})=t(g^{-1})=x_i\rrbracket\\
			& = \alpha_{(\varphi_1,\ldots,\varphi_n)}(\delta_{g^{-1}})\\
			& = \alpha_{(\varphi_1,\ldots,\varphi_n)}\circ S(\delta_g)\\
			& = \alpha^\ast_{(\varphi_1,\ldots,\varphi_n)}(\delta_g)
		\end{align*}
		for every $g\in\mathcal{G}.$ The map $\alpha$ also takes unity to unity, because
		\begin{align*}
			\alpha_{(\emph{\textbf{1}}_1,\ldots,\emph{\textbf{1}}_n)}(\delta_g) & = \emph{\textbf{1}}_i(\delta_{1_{x_i}})\,\delta_{1_{x_i}}\,\llbracket s(g)=t(g)=x_i\rrbracket\\
			& = \delta_{1_{x_i}}\,\llbracket s(g)=t(g)=x_i\rrbracket = \emph{\textbf{1}}(\delta_g)
		\end{align*}
		for every $g\in\mathcal{G},$ with $\emph{\textbf{1}}_i:G_{x_i}\rightarrow \K\setminus\{0\}$ given by $g\mapsto 1$ for every $i=1,\ldots,n.$
		
		Observe that $\mathfrak{Brt}(\K\mathcal{G},A)$ is a commutative inverse semigroup with unity, but is not necessarily a group. In fact, for a biretraction $\alpha_{(\varphi_1,\ldots,\varphi_n)}$ with $\varphi_i\neq 0$ for every $i$ such that $x_i\in X\subseteq\mathcal{G}^{(0)},$
		\begin{align*}
			\left(\alpha_{(\varphi_1,\ldots,\varphi_n)} \ast\alpha^\ast_{(\varphi_1,\ldots,\varphi_n)}\right) (\delta_g) & = \varphi_i(g)\varphi_i(g^{-1}) \,\delta_{1_{x_i}}\, \llbracket s(g)=t(g)=x_i\in X\rrbracket\\
			& = \varphi_i(1_{x_i})\,\delta_{1_{x_i}}\, \llbracket s(g)=t(g)=x_i\in X\rrbracket\\
			& = \delta_{1_{x_i}}\, \llbracket s(g)=t(g)=x_i\in X\rrbracket
		\end{align*}
		for every $g\in\mathcal{G},$ which is not the unity of $\mathfrak{Brt}(\K\mathcal{G},A),$ unless $X=\mathcal{G}^{(0)},$ that is, unless $\alpha_{(\varphi_1,\ldots,\varphi_n)}$ is a global biretraction. In particular, $\mbox{Gl}\mathfrak{Brt}(\K\mathcal{G},A)$ is an abelian group.
		
	\end{exmp}

	\begin{exmp}[The algebraic quantum torus]
		Consider an algebra $T_q$ over $\C,$ generated by two invertible elements $U$ and $V$ satisfying $UV=q\,VU,$ with $q \in \C^\times.$ The algebra $T_q$ has a structure of Hopf algebroid over the commutative $\C$-algebra $A=\C[U]:$
		\begin{itemize}
			\item $s=s_l=t_l=s_r=t_r : A\rightarrow T_q$ is the inclusion map;
			\item $\Delta_l(U^nV^m)= U^nV^m\otimes_A V^m\,\,$ and $\,\,\varepsilon_l(U^nV^m)=U^n;$
			\item $\Delta_r(V^mU^n)= V^mU^n\otimes_A V^m\,\,$ and $\,\,\varepsilon_r(V^mU^n)=U^n;$
			\item $S(U^nV^m)=V^{-m}U^n.$
		\end{itemize}
		
		Observe that the only idempotent of $A$ is $1.$ Then we can only have global biretractions for $T_q.$ By the Remark \ref{s=t}, a global biretraction for $T_q$ can be described as a linear and multiplicative map $\alpha:T_q\rightarrow A$ such that $\alpha|_A = Id_A.$
		
		Moreover, since $\alpha$ is multiplicative, we have that
		$$\alpha(V)\,\alpha(V^{-1})=\alpha(V^{-1})\,\alpha(V)=\alpha(V^{-1}V)=\alpha(1_\C)=1_\C \,\, \Rightarrow \,\, \alpha(V)^{-1}=\alpha(V^{-1}),$$
		which implies that $\alpha(V)$ is invertible in $A,$ and consequently,
		$$U\alpha(V)=\alpha(UV) = q\,\alpha(VU) = q\,U\alpha(V)$$
		$$\Rightarrow q=1_\C.$$
		
		So we only have global biretractions for the commutative torus $T_1.$ In this case, we have that a global biretraction for $T_1$ is a multiplicative and linear map $\alpha: T_1\rightarrow A$ such that $\alpha(U)=U$ and $\alpha(V)= q_\alpha \, U^{t_\alpha},$ with $q_\alpha \in\mathbb{C}$ and $t_\alpha\in \mathbb{Z}.$ 
		
		Moreover, since the zero map is not a global biretraction, any global biretraction  $\alpha: T_1\rightarrow A$ is in fact a morphism of algebras (since $\alpha(1_{T_1}) = 1_A)$.  $T_1$ and $A$ are algebras of Laurent polynomials, $T_1 = \C[U,U^{-1}, V, V^{-1}]$ and $A = \C[U,U^{-1}]$. As algebra, $T_1$ is isomorphic to $A\otimes_{\mathbb{C}} A$, but its structure as a Hopf algebroid doesn't coicide with that given in Example \ref{AA}. General arguments from algebraic geometry show that algebra morphisms $\alpha: T_1 \to A$ correspond to maps $f : \C^\times \to \C^\times \times \C^\times$ whose entries are Laurent polynomials in $z$, i.e., $f(z) = (p_1(z),p_2(z))$ with $p_i(U) \in A$. 
		Given such a map, the associated morphism of algebras is $\alpha(U) = p_1(U), \alpha(V) = p_2(V) $. 
		
		In particular, given a real number $\theta$ and an integer $n$, the biretraction $\alpha : T_1 \to A $ given by $\alpha (V) = e^{2 \pi i \theta}U^n $ corresponds to the the map $f: \C^\times \to \C^\times \times \C^\times, f(z) = (z, e^{2 \pi i \theta}z^n)$. The restriction of $f$ to the unit circle $S^1$ yields the map 
		$$g: S^1 \to S^1 \times S^1, \qquad e^{2 \pi i t} \mapsto (e^{2 \pi i t}, e^{ 2 \pi i (\theta + tn)}).$$
		Hence biretractions of the Hopf algebroid $T_1$ include imersions of $T_1$ in $T^2.$ Also, we can say that $\alpha$ rolls up the unit circle $S^1$ around the torus $T^2.$ Indeed, observe that
		\begin{align*}
			\alpha\ast\alpha(V)& = \alpha\circ\alpha(V)\, \alpha(V) = \alpha(e^{2\pi i\theta}U^n)\, e^{2\pi i\theta}U^n\\
			& = e^{2\pi i\theta}\alpha(U)^n\, e^{2\pi i\theta}U^n = e^{2\pi i\, 2\theta}U^{2n}
		\end{align*}
		and, analogously, 
		$$\alpha^k:=\underbrace{\alpha\ast \cdots\ast \alpha}_{k \, \text{times}}(V) = e^{2\pi i\, k\theta}U^{kn}$$
		Consequently, we can associate $\alpha^k$ with the restriction
		$$g_k: S^1 \to S^1 \times S^1, \qquad e^{2 \pi i t} \mapsto (e^{2 \pi i t}, e^{ 2 \pi i k(\theta + tn)}).$$
		
		We remark that $g$ is a closed curve that starts and ends at $(1,e^{2\pi i\theta})$ for $t=0$ and for $t=1,$ and runs along the torus as shown in \emph{Figure \ref{toro}}.
		\begin{figure}[h]
			\centering
			\includegraphics[scale=0.5]{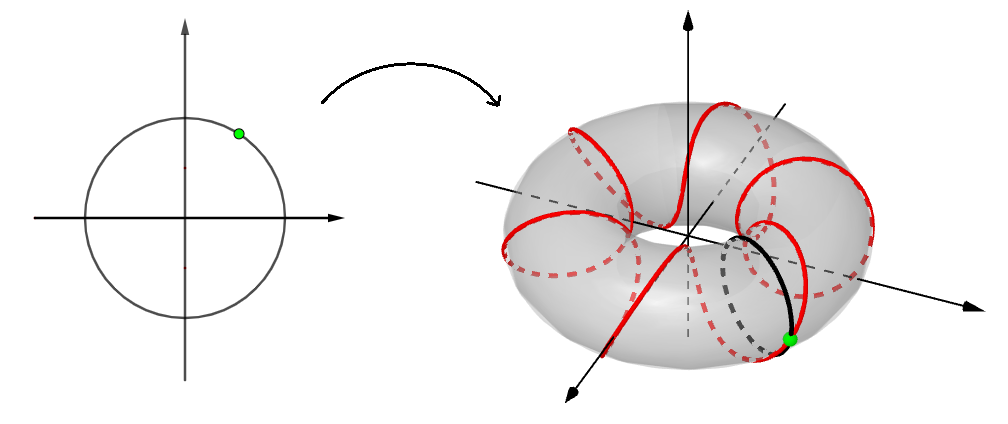}
			\caption{Representation of the curve $g$}
			\label{toro}
		\end{figure}
		
		Observe that the curve $g_2$ acts similarly to $g$ but rolls up twice as fast (vertically) along the torus, starting and ending at $(1,\, e^{4\pi i\theta}).$ In general, the map $g_k$ rolls up the torus $k-$times faster than $g$ vertically, starting and ending at $(1,e^{2k\pi i\theta}).$
	\end{exmp}
	
	\section{Conclusions and outlook}
	
	In this work, we introduced local biretractions over a Hopf algebroid with commutative base algebra. It is not clear, at this point, how to extend this approach to cover Hopf algebroids with noncommutative bases; the balance between noncommutativity of the Hopf algebroid and the locality of the bisections is, so far, the most difficult problem to solve. It is worth noting that 
	the noncommutativity of a Hopf algebroid and of its base algebra was adequately addressed at the level of global biretractions by 
	Xiao Han and Giovanni Landi in their work on the Eheresmann-Schauenburg bialgebroid associated to a noncommutative principal bundle \cite{Han}.
	The group of gauge transformations of a quantum principal bundle (a Hopf-Galois extension) was proved to be isomorphic to the group of global bisections of the Ehresmann-Schauenburg bialgebroid (see \cite{Han}, Proposition 4.6). 
	
	The next steps in this research program are to develop a general theory of local biretractions of Hopf algebroids and trying to establish some categorical equivalence between Hopf algebroids and Quantum inverse semigroups following the ideas introduced in \cite{buss} using, on one side the quantum inverse semigroup of local biretractions of a Hopf algebroid and on the other hand the ``germ'' Hopf algebroid of action defined by a quantum inverse semigroup.

	\section{Aknowledgements}
	The authors wish to thank Paolo Saracco for fruitful discussions and suggestions to improve this work during his visit to the Maths Department of UFSC in February 2022. 
	This work has been supported by the Brazilian research agencies CAPES, Coordena\c{c}\~ao 
	de Aperfei\c{c}oamento de Pessoal de N\'ivel Superior, and CNPq, National Council for Scientific and Technological Development. The first author was partially by the CNPq grant
	309469/2019-8 and the third author was partially supported by CAPES, Finance Code 001.

\end{document}